\documentclass[DIV=12,numbers=enddot,leqno,bibliography=totoc]{scrartcl}
\usepackage{./q-HodgeStyle_minimal}
\title{\texorpdfstring{$q$}{q}-Hodge complexes and refined \texorpdfstring{$\TC^-$}{TC-}}
\author{Samuel Meyer and Ferdinand Wagner}

\newcommand{\AdicGm}{\IT}

\newcommand{\Akup}{\A_{\ku,  p}^*}
\newcommand{\AKUp}{\A_{\KU,  p}}

\begin{document}
	\maketitle
	
	\begin{abstract}
		\textbf{Abstract. --- }
		As a consequence of Efimov's proof of rigidity of the $\infty$-category of localising motives \cite{EfimovRigidity}, Efimov and Scholze have constructed refinements of localising invariants such as $\THH$ and $\TC^-$. These refinements often contain vastly more information than the original invariant.
		
		In this article we explain a general recipe how to compute the refinements in certain situations. We then apply this recipe to compute the homotopy groups of $\TCref(\ku\otimes\IQ/\ku)$ and $\TCref(\KU\otimes\IQ/\KU)$. The result has a rather surprising geometric description and contains non-trivial information modulo any prime, in contrast to the unrefined $\TC^-$.
	\end{abstract}

	\tableofcontents
	\renewcommand{\SectionPrefix}{\textrm{\S}}
	\renewcommand{\SubsectionPrefix}{\textrm{\S}}
	
	\newpage

	\section{Introduction}\label{sec:Intro}
	
	Topological Hochschild homology ($\THH$) and its variants $\TC^-$ and $\TP$ can be used to construct powerful cohomology theories for $p$-adic formal schemes, such as prismatic cohomology \cite{BMS2,Prismatic}. However, in the setting of rigid-analytic varieties over $\IQ_p$, or varieties over $\IQ$, they are less useful: When evaluated on rational inputs, these invariants will be rational themselves, and so any cohomology theory one might construct from them will never admit interesting comparisons to, say, étale cohomology with torsion coefficients.
	
	In this article we'll study a refinement of $\THH$/$\TC^-$ due to Efimov and Scholze that allows us to get around this shortcoming, while still being somewhat computable.	We hope that this will give rise to some interesting arithmetic cohomology theories.

	\subsection{Refined localising invariants}\label{subsec:RefinedInvariantsIntro}
	
	The construction of refinements of $\THH$ and $\TC^-$ is based on Efimov's rigidity theorem (\cref{thm:EfimovRigidity} below). The notion of \emph{rigidity} for symmetric monoidal $\infty$-categories was introduced by Gaitsgory and Rozenblyum (see \cite[Definition~{\chref{1.9.1.2}[I.9.1.2]}]{GaitsgoryRozenblyumDerivedI}). We'll work with the following variant of their definition, which is equivalent to the original one by \cite[Corollary~\chref{4.57}]{RamziLocallyRigid}:
	\begin{defi}\label{def:Rigid}
		A presentable stable symmetric monoidal $\infty$-category%
		\footnote{We always assume that the tensor product commutes with colimits in both variables (see \cref{par:Notation}).}
		$\Ee$ is \emph{rigid} if the following two conditions are satisfied:
		\begin{alphanumerate}
			\item The tensor unit $\IUnit\in \Ee$ is compact.\label{enum:TensorUnitCompact}
			\item $\Ee$ is generated under colimits by objects of the form $X\simeq \colimit(X_1\rightarrow X_2\rightarrow \dotsb)$, where each $X_{n}\rightarrow X_{n+1}$ is \emph{trace-class}, that is, induced by a morphism $\IUnit \rightarrow X_{n+1}^\vee\otimes X_n$ (see \cref{def:Nuclear}).\label{enum:GeneratedByBasicNuclear}
		\end{alphanumerate}
	\end{defi}
	
	Generalising the construction of $\Mot^\loc$ by Blumberg--Gepner--Tabuada \cite{BlumbergGepnerTabuada}, Efimov introduces a presentable stable symmetric monoidal $\infty$-category $\Mot_\Ee^\loc$ of \emph{localising motives over $\Ee$} \cite[Definition~\chref{1.20}]{EfimovLimits} and shows the following deep result:
	
	\begin{thm}[\cite{EfimovRigidity}; see {\cite[Theorem~\chref{4.7.1}]{SheavesOnManifolds}} for the case of $\Mot^\loc$]\label{thm:EfimovRigidity}
		If $\Ee$ is a rigid presentable stable symmetric monoidal $\infty$-category, then the same is true for $\Mot_\Ee^\loc$.
	\end{thm}
	
	Efimov and Scholze observed that this theorem has the following curious consequence:
	
	\begin{numpar}[Refined localising invariants \textmd{(Efimov--Scholze)}.]\label{par:RefinedInvariants}
		Let $T$ be a localising invariant over $\Ee$, that is, a colimit-preserving functor
		\begin{equation*}
			T\colon \Mot_\Ee^\loc\longrightarrow \Dd
		\end{equation*}
		into a presentable stable $\infty$-category $\Dd$. If $T$ is equipped with a symmetric monoidal structure, then \cref{thm:EfimovRigidity} implies that there's a unique symmetric monoidal factorisation
		\begin{equation*}
			\begin{tikzcd}
				\Mot_\Ee^\loc\rar["T"]\drar[dashed,"T^\mathrm{ref}"'] & \Dd\\
				& \Dd^\mathrm{rig}\uar
			\end{tikzcd}
		\end{equation*}
		This factorisation $T^\mathrm{ref}\colon \Mot_\Ee^\loc\rightarrow \Dd^\mathrm{rig}$ is the \emph{refinement of $T$} defined by Efimov--Scholze. Here $\Dd^\mathrm{rig}$ denotes the \emph{rigidification of $\Dd$} in the sense of \cite[Construction~\chref{4.75}]{RamziLocallyRigid}; see also \cite[Proposition~\chref{1.23}]{EfimovLimits}. We recall from these references that $\Dd^\mathrm{rig}$ can be described as the full sub-$\infty$-category of $\Ind(\Dd)$%
		\footnote{Applying $\Ind(-)$ to large $\infty$-categories causes set-theoretical issues, but $\Nuc\Ind(-)$ is fine; see~\cref{rem:NucInd}.}
		generated under colimits by ind-objects of the form $\indcolim_{i\in\IQ}x_i$, where all transition maps $x_i\rightarrow x_j$ for rational numbers $i<j$ are trace-class. If $\Dd$ is \emph{locally rigid} and its tensor unit is $\omega_1$-compact, then it suffices to consider $\IZ_{\geqslant 0}$-indexed ind-objects instead of $\IQ$-indexed ones. In other words, in this case
		\begin{equation*}
			\Dd^\mathrm{rig}\overset{\simeq}{\longrightarrow}\Nuc\Ind(\Dd)
		\end{equation*}
		is given by the \emph{nuclear} objects in $\Ind(\Cc)$ in the sense of \cref{def:Nuclear}. See \cite[Theorem~\chref{4.2}]{EfimovLimits}.
	\end{numpar}
	\begin{rem}\label{rem:TC-refOC}
		The refinement procedure from \cref{par:RefinedInvariants} is very sensitive to the choice of $\Ee$. This is a feature, not a bug, as it offers a lot of flexibility, even if we stick to the case where $T$ is topological Hochschild homology. For example, we could consider the $p$-completed $\THH$ functor
		\begin{equation*}
			\THH(-;\IZ_p)\colon \Mot^\loc\longrightarrow \bigl(\Sp^{\B S^1}\bigr)_p^\complete
		\end{equation*}
		to obtain a refinement $\THHref(-;\IZ_p)$. But for a complete non-archimedean algebraically closed field $C$, we could also define $\THHref_{/\Oo_C}(-;\IZ_p)$ to be the refinement of the functor
		\begin{equation*}
			\THH(-;\IZ_p)\colon \Mot_{\Oo_C}^\loc\longrightarrow \Mod_{\THH(\Oo_\Cc;\IZ_p)}\bigl(\Sp^{\B S^1}\bigr)_p^\complete\,,
		\end{equation*}
		where we only accept motives over $\Oo_C$ as input.%
		\footnote{Historically, $\THHref_{/\Oo_C}(-;\IZ_p)$ is the first refined invariant. Efimov and Scholze have sketched a computation of $\THHref_{/\Oo_C}(C;\IZ_p)$ \cite{ScholzeTC-OC}, by reducing the problem to the known computation of $\THH(\Oo_C/p^\alpha;\IZ_p)$ for all $\alpha\geqslant 1$ (compare \cref{thm:RefinedInvariantsIntro} below).}
		These two refinements are completely different, as the forgetful functor $\Mot_{\Oo_C}^\loc\rightarrow \Mot^\loc$ doesn't preserve trace-class morphisms.
	\end{rem}
	
	The refinement $T^\mathrm{ref}$ typically contains vastly more information than $T$ itself, as we'll discuss in the case of $\THHref(\IQ)$ below. Our first goal in this article is to give a recipe for computing $T^\mathrm{ref}$ in certain cases.
	
	\begin{thm}[see \cref{thm:RefinedInvariants}]\label{thm:RefinedInvariantsIntro}
		Assume we're in the situation of \cref{par:RefinedInvariants}, with $\Dd$ locally rigid. Let $\Ee\rightarrow\Xx$ be a strongly continuous symmetric monoidal functor into another rigid symmetric monoidal presentable stable $\infty$-category, such that $\Xx$ is smooth and proper as an $\Ee$-module. Suppose we're given the following data:
		\begin{alphanumerate}
			\item[V] A tower of $\IE_1$-algebras in $\Xx$ of the form $V_0\leftarrow V_1\leftarrow V_2\leftarrow \dotsb$, such that each $V_r$ is dualisable in $\Xx$ and contained in the thick tensor ideal generated by $V_0$. Moreover, for all $r\geqslant 0$, the induced map $V_{r+1}\otimes V_r\rightarrow V_r\otimes V_r$ factors through the multiplication\label{enum:VIntro}
			\begin{equation*}
				V_{r+1}\otimes V_r\overset{\mu}{\longrightarrow} V_r
			\end{equation*}
			as a map of $V_{r+1}$-$V_r$-bimodules.
		\end{alphanumerate}
		Let $\Uu\subseteq \Xx$ be the full sub-$\infty$-category spanned by those $U\in\Xx$ for which $\Hhom_\Xx(V_0,U)\simeq 0$. Then there exists a cofibre sequence of the following form in $\Nuc\Ind(\Dd)$:
		\begin{equation*}
			\indcolim_{r\geqslant 0}T\bigl(\RMod_{V_r}(\Xx)\bigr)^\vee\longrightarrow T(\Xx)\longrightarrow T^\mathrm{ref}(\Uu)\,.
		\end{equation*}
	\end{thm}
	With the language developed in \cref{subsec:KillingAlgebras}, we could say that \emph{$\Uu$ is obtained from $\Xx$ by killing $V_0$} and \emph{$T^\mathrm{ref}(\Uu)$ is obtained from $T(\Xx)$ by killing the idempotent pro-algebra $\prolim_{r\geqslant 0}T(\RMod_{V_r}(\Xx))$}.
	
	\begin{numpar}[How to apply \cref{thm:RefinedInvariantsIntro}.]
		Even though \cref{thm:RefinedInvariantsIntro} looks quite technical, we'll verify in \cref{subsec:Burklund} that it covers many cases of interest. This is due to the following observation (see \cref{cor:E1FactorisationSpalpha}): Let $v\colon \Ii\rightarrow \IUnit_\Xx$ be a morphism from a dualisable object such that the cofibre $\IUnit_\Xx/v$ admits a right-unital multiplication. Then Burklund's tower of $\IE_1$-algebras \cite[Theorem~\chref{1.5}]{BurklundMooreSpectra}
		\begin{equation*}
			\IUnit_\Xx/v^2\longleftarrow \IUnit_\Xx/v^3\longleftarrow \IUnit_\Xx/v^4\longleftarrow \dotsb
		\end{equation*}
		satisfies the conditions from \cref{thm:RefinedInvariantsIntro}\cref{enum:VIntro}. Thus, to compute, for example, $\THHref(\IS[1/p])$ for a prime~$p$, one can choose a Burklund-style tower of $\IE_1$-structures on $\IS/p^\alpha$ for sufficiently large~$\alpha$ to obtain a cofibre sequence
		\begin{equation*}
			\indcolim_{\alpha}\THH(\IS/p^\alpha)^\vee\longrightarrow \THH(\IS)\longrightarrow \THHref\bigl(\IS\bigl[\localise{p}\bigr]\bigr)\,.
		\end{equation*}
		In a similar way, we'll explain how one could attempt computations such as $\THHref(\IQ)$, $\THHref(\IS[x])$, or $\THHref(\L_n^f\IS_{(p)})$---which brings us to our main question:
	\end{numpar}
	
	\subsection{What's \texorpdfstring{$\THHref(\IQ)$}{THHref(Q)}?}\label{subsec:THHrefQIntro}
	
	We'll explain in \cref{subsec:NewCohomologyTheoryIntro} why the answer to this question should be interesting, but let us already remark that it has to be non-trivial: As we'll see below, $\THHref(\IQ)_p^\complete\not\simeq 0$ for all primes~$p$. So in contrast to $\THH(\IQ)\simeq \IQ$, the refined version $\THHref(\IQ)$ contains non-trivial $p$-complete information for any prime~$p$.
	
	However, computing $\THHref(\IQ)$, or just its $p$-completions, is a highly non-trivial task: As we've seen above, this would involve computing $\THH(\IS/p^\alpha)$, or at least a pro-system of the form $\prolim_\alpha \THH(\IS/p^\alpha)$, which seems currently out of reach.
	
	Scholze and Efimov have suggested that a more approachable goal would be to compute $\THHref((\MU\otimes\IQ)/\MU)$ and then to attack the original question---to the extent in which that's possible---via Adams--Novikov descent. Here we let $\THHref(-/k)$ denote the refinement of
	\begin{equation*}
		\THH(-/k)\colon \Mot_k^\loc\rightarrow \Mod_k(\Sp)^{\B S^1}
	\end{equation*}
	for any $\IE_\infty$-ring spectrum~$k$.
	
	While we still don't know what happens for $k=\MU$, the purpose of this article is to give an answer for $k=\ku$ and $k=\KU$. To this end, we'll introduce the following variant of $\THHref$, which will make it easier to formulate the result in geometric terms.
	
	\begin{numpar}[Refined $\TC^-$.]
		If $k$ is complex orientable and $t\in\pi_{-2}(k^{\h S^1})$ is a chosen complex orientation, then taking $S^1$-fixed points induces a symmetric monoidal equivalence
		\begin{equation*}
			(-)^{\h S^1}\colon \Mod_k(\Sp)^{\B S^1}\overset{\simeq}{\longrightarrow}\Mod_{k^{\h S^1}}(\Sp)_t^\complete
		\end{equation*}
		between $k$-modules with $S^1$-action and $t$-complete $k^{\h S^1}$-modules (see \cref{lem:S1ActiontComplete}). In particular, we can view $\TC^-(-/k)$ as a symmetric monoidal functor $\Mot_k^\loc\rightarrow \Mod_{k^{\h S^1}}(\Sp)_t^\complete$, which contains the same information as $\THH(-/k)$. Applying refinement, we obtain the functor
		\begin{equation*}
			\TCref(-/k)\colon \Mot_k^\loc\longrightarrow \Nuc(k^{\h S^1})
		\end{equation*}
		of by Efimov and Scholze. Here $\Nuc(k^{\h S^1})\coloneqq \Nuc\Ind(\Mod_{k^{\h S^1}}(\Sp)_t^\complete)$ denotes Efimov's $\infty$-category  of \emph{nuclear $k^{\h S^1}$-modules}.
	\end{numpar}
	
	So it will be enough to compute $\TCref(\ku\otimes\IQ/\ku)$ and $\TCref(\KU\otimes\IQ/\KU)$.	
	
	\begin{numpar}[$q$-Hodge filtrations and $q$-Hodge complexes.]
		Suppose we've chosen an $\IE_1$-structure on $\IS/m$ for some integer~$m$. Then \cite[Theorems~\chref{4.27} and~\chref{5.63}]{qdeRhamku} show that
		\begin{align*}
			\pi_{2*}\TC^-\bigl((\ku\otimes\IS/m)/\ku\bigr)&\cong \fil_{\qHodge}^\star\qdeRham_{(\IZ/m)/\IZ}\,,\\
			\pi_{2*}\TC^-\bigl((\KU\otimes\IS/m)/\KU\bigr)&\cong \qHodge_{(\IZ/m)/\IZ}[\beta^{\pm1}]\,,
		\end{align*}
		where $\qdeRham_{-/\IZ}$ denotes the derived $q$-de Rham complex,  $\fil_{\qHodge}^\star$ is a \emph{$q$-Hodge filtration} in the sense of \cite[Definition~\chref{3.2}]{qWittHabiro} (which in the case of $\IZ/m$ just amounts to a $q$-deformation of the Hodge filtration, as the additional compatibilities are trivial), and $\qHodge_{(\IZ/m)/\IZ}$ is the associated \emph{$q$-Hodge complex}
		\begin{equation*}
			\qHodge_{(\IZ/m)/\IZ}\coloneqq \colimit\Bigl( \fil_{\qHodge}^0\qdeRham_{(\IZ/m)/\IZ}\xrightarrow{(q-1)} \fil_{\qHodge}^1\qdeRham_{(\IZ/m)/\IZ}\xrightarrow{(q-1)}\dotsb\Bigr)_{(q-1)}^\complete\,.
		\end{equation*}
		Thanks to Burklund's result \cite[Theorem~\chref{1.5}]{BurklundMooreSpectra}, we can choose a coinitial sub-poset $\IN^{\lightning}\!\subseteq \IN$ of positive integers, partially ordered by divisibility, together with compatible $\IE_1$-structures on $\IS/m$ for $m\in\IN^{\lightning}\!$. This leads to the following result:
	\end{numpar}

	\begin{thm}[see \cref{thm:RefinedTC-qHodge}]\label{thm:RefinedTC-Intro}
		$\TCref((\ku\otimes\IQ)/\ku)$ and $\TCref((\KU\otimes\IQ)/\KU)$ are concentrated in even degrees, and their even homotopy groups are described as follows:
		\begin{alphanumerate}
			\item $\pi_{2*}\TCref((\ku\otimes\IQ)/\ku)\cong \A_{\ku}^*$, where $\A_{\ku}^*$ is the idempotent nuclear graded $\IZ[\beta]\llbracket t\rrbracket$-algebra obtained by killing the pro-idempotent $\prolim_{m\in\IN^{\lightning}\!}\Fil_{\qHodge}^*\qhatdeRham_{(\IZ/m)/\IZ}$.
			\item $\pi_{2*}\TCref((\KU\otimes\IQ)/\KU)\cong \A_{\KU}[\beta^{\pm 1}]$, where $\A_{\KU}$ is the idempotent nuclear $\IZ\qpower$-algebra obtained by killing the pro-idempotent $\prolim_{m\in\IN^{\lightning}\!}\qHodge_{(\IZ/m)/\IZ}$.
		\end{alphanumerate}
	\end{thm}
	\cref{thm:RefinedTC-Intro} provides a description of the desired homotopy rings in terms of the $q$-Hodge filtrations on $\qdeRham_{(\IZ/m)/\IZ}$. In \cref{subsec:ElementaryProof}, we'll describe $\fil_{\qHodge}^\star\qdeRham_{(\IZ/m)/\IZ}$ in terms of explicit generators. This leads to a much more explicit description of $\A_{\ku}^*$ and $\A_{\KU}$ in terms of rings of overconvergent functions on certain adic spaces. For simplicity, we'll work with $\TCref((\ku_p^\complete\otimes\IQ)/\ku_p^\complete)$ and $\TCref((\KU_p^\complete\otimes\IQ)/\KU_p^\complete)$ instead. 
	Let us first formulate the result for $\KU_p^\complete$, as it is easier to state. We put
	\begin{equation*}
		\AKUp\coloneqq \pi_0\TCref\bigl((\KU_p^\complete\otimes\IQ)/\KU_p^\complete\bigr)\,,
	\end{equation*}
	so $\pi_{2*}\TCref((\KU_p^\complete\otimes\IQ)/\KU_p^\complete)\cong \AKUp[\beta^{\pm 1}]$. Let also $X\coloneqq \Spa \IZ_p\qpower\smallsetminus\{p=0,q=1\}$ be the \enquote{analytic locus} where $p$ or $q-1$ is invertible. Then $\AKUp$ has the following description, confirming a conjecture of Scholze and Efimov.
	\begin{thm}\label{thm:OverconvergentNeighbourhood}
		Let $Z\subseteq X$ denote the union of the closed subsets $\Spa(\IF_p\qLaurent,\IF_p\qpower)$ and $\Spa(\IQ_p(\zeta_{p^n}),\IZ_p[\zeta_{p^n}])$ for all $n\geqslant 0$. Let $Z^\dagger$ denote the overconvergent neighbourhood of $Z$ in $X$ and $\Oo(Z^\dagger)$ the nuclear $\IZ_p\qpower$-algebra of overconvergent functions on $Z$. Then
		\begin{equation*}
			\AKUp\cong \Oo(Z^\dagger)\,.
		\end{equation*}
	\end{thm}
	In \cref{fig:Rainbow} we show a picture of $Z^\dagger$. It should be reminiscent of Scholze's famous prismatic picture (a nice depiction of which can be found in \cite[p.\ \chpageref{4}]{HesselholtNikolausHandbook}), but the rays are \enquote{overconvergently blurred} and the \enquote{origin} $\{p=0,q=1\}$ has been removed. 
	
	\begin{figure}[ht]
		\centering\includegraphics{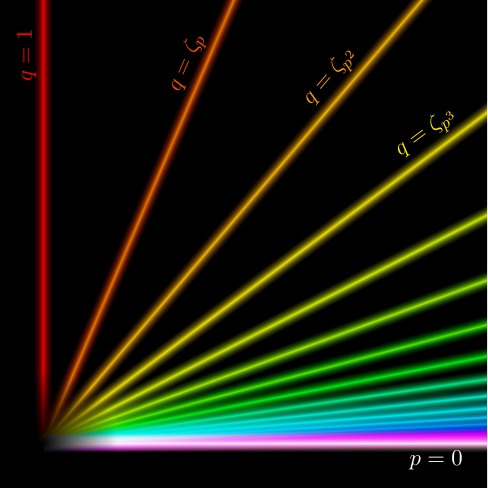}
		\caption{The analytic spectrum of $\AKUp\cong \Oo(Z^\dagger)$.}\label{fig:Rainbow}
	\end{figure}
	
	Since $Z^\dagger$ visibly contains the entire infinitesimal neighbourhood of $\{p=0\}$ except for the \enquote{origin}, we see that $\TCref((\KU_p^\complete\otimes \IQ)/\KU_p^\complete)_p^\complete\neq 0$. In particular, it follows that $\THHref(\IQ)_p^\complete\not\simeq 0$, as we've claimed above.%
	To formulate a similar geometric result for $\ku_p^\complete$, consider the ungraded ring $\IZ[\beta,t]_{(p,t)}^\complete$ with its $(p,t)$-adic topology. We wish to encode the graded $(p,t)$-complete ring $\IZ_p[\beta]\llbracket t\rrbracket$ in terms of an action of $\IG_m$ on $\Spa \IZ[\beta,t]_{(p,t)}^\complete$, as usual---but we have to be careful: Since we wish that $t$ is a topologically nilpotent elements in non-zero graded degree, we can only act by units $u$ \enquote{of norm $\abs{u}=1$}. More precisely, we have to replace $\IG_m$ by the \enquote{adic unit circle} $\AdicGm\coloneqq \Spa(\IZ[u^{\pm 1}],\IZ[u^{\pm 1}])$.
	
	With this modification, everything works (as we'll elaborate in \cref{subsec:GradedAdic}): Declaring $\beta$ and $t$ to have degree $2$ and $-2$, respectively, determines an action of $\AdicGm$ on $\Spa \IZ[\beta,t]_{(p,t)}^\complete$, and we can identify $\IZ_p[\beta]\llbracket t\rrbracket$ with the structure sheaf on $(\Spa \IZ[\beta,t]_{(p,t)}^\complete)/\AdicGm$, where the quotient is always taken in the derived (or \enquote{stacky}) sense. We also let $X^*\coloneqq \Spa \IZ[\beta,t]_{(p,t)}^\complete\smallsetminus\{p=0,\beta t=0\}$. Since $p$ and $\beta t$ are homogeneous, $X^*$ inherits an action of $\AdicGm$. Putting
	\begin{equation*}
		\Akup\coloneqq \pi_{2*}\TCref\bigl((\ku_p^\complete\otimes\IQ)/\ku_p^\complete\bigr)\,,
	\end{equation*}
	we see that $\Akup$ is a graded $\IZ_p[\beta]\llbracket t\rrbracket$-module, hence we can regard it as a quasi-coherent sheaf on $(\Spa \IZ[\beta,t]_{(p,t)}^\complete)/\AdicGm$. As we'll see, it is already the pushforward of a sheaf on the open substack $X^*/\AdicGm$. This sheaf, which we'll also denote $\Akup$, can be described as follows:
	\begin{thm}\label{thm:OverconvergentNeighbourhoodEquivariant}
		Let $Z^*\subseteq X^*$ be union of the $\AdicGm$-equivariant closed subsets $\{p=0\}$ and $\{[p^n]_{\ku}(t)=0\}$ for all $n\geqslant 0$, where $[p^n]_{\ku}(t)\coloneqq ((1+\beta t)^{p^n}-1)/\beta$ denotes the $p^n$-series of the formal group law of $\ku$. Let $Z^{*,\dagger}$ denote the overconvergent neighbourhood of $Z^*$. Then $Z^{*,\dagger}$ inherits a $\AdicGm$-action and
		\begin{equation*}
			\Akup\cong \Oo_{Z^{*,\dagger}/\AdicGm}\,.
		\end{equation*}
	\end{thm}
	
	\subsection{New cohomology theories for \texorpdfstring{$\IQ$}{Q}-varieties}\label{subsec:NewCohomologyTheoryIntro}
	
	Let us end with a bit of speculation. It should be possible to adapt the formalism of even filtrations from \cite{EvenFiltration} to $\TCref(-/\ku)$ and $\TCref(-/\KU)$. For a smooth variety $X$ over $\IQ$, this would allow us to construct cohomology theories $\R\Gamma_{\ku}(X)$ and $\R\Gamma_{\KU}(X)$; the former comes naturally equipped with a filtration:
	\begin{align*}
		\fil^\star \R\Gamma_{\ku}(X)&\coloneqq \gr_{\ev, \h S^1}^*\TCref\bigl((\ku\otimes X)/\ku\bigr)\,,\\
		\R\Gamma_{\KU}(X)&\coloneqq \gr_{\ev,\h S^1}^0\TCref\bigl((\KU\otimes X)/\KU\bigr)\,.
	\end{align*}
	This article can be viewed as a computation of the coefficients of $\fil^*\R\Gamma_{\ku}(-)$ and $\R\Gamma_{\KU}(-)$.
	
	\begin{numpar}[Relation to $q$-de Rham/$q$-Hodge cohomology.]
		Morally, $\fil^*\R\Gamma_{\ku}(X)$ should be the \enquote{$q$-Hodge-filtered $q$-de Rham cohomology of $X$} and $\R\Gamma_{\KU}(X)$ should be the \enquote{$q$-Hodge cohomology of $X$}.
		
		We remark that there's a naive definition of $q$-de Rham cohomology of $\IQ$-varieties (obtained, for example, by applying \cite[Theorem~\chref{A.1}]{qWittHabiro} for $A=\IQ$), but it would just be a $(q-1)$-completed base change of ordinary de Rham cohomology. By contrast, $\R\Gamma_{\ku}(X)$ and $\R\Gamma_\KU(X)$ will be non-trivial modulo any prime~$p$, and so they ought to be much more interesting. In particular, we hope to find not only comparison isomorphisms with de Rham cohomology, but also with étale cohomology of $X_{\ov \IQ}$ with torsion coefficients.
	\end{numpar}
	\begin{numpar}[Relation to Habiro cohomology.]
		We expect that $\R\Gamma_{\KU}(-)$ naturally descends from $\IZ\qpower$ to the \emph{Habiro ring} $\Hh\coloneqq \limit_{m\in\IN}\IZ[q]_{(q^m-1)}^\complete$. In particular, its ring of coefficients $\A_{\KU}$ should admit a Habiro descent $\Aa_{\KU}$ satisfying $\A_{\KU}\cong \Aa_{\KU}\soltimes_{\Hh}\IZ\qpower$.
		
		This descent should arise as follows: We explain in \cite[\S{\chref[section]{5}}]{qdeRhamku} how descent to the Habiro ring corresponds to making the $S^1$-action on $\THH(-/\ku)$ \emph{genuine} with respect to all finite subgroups $C_m\subseteq S^1$; or more precisely, it corresponds to turning $\THH(-/\ku)$ into a \emph{cyclonic spectrum}. One could then apply the refinement procedure to the functor
		\begin{equation*}
			\THH(-/\KU)\colon \Mot_\KU^\loc\longrightarrow \cat{CycnSp}
		\end{equation*}
		valued in cyclonic spectra. Via an appropriate cyclonic even filtration, it should then be possible to to construct the desired Habiro descent $\R\Gamma_\Hh(X)$ of $\R\Gamma_{\KU}(X)$. Moreover, we hope that $\R\Gamma_\Hh(X)$ admits a \emph{stacky} approach, given by an appropriate cyclonic version of the \emph{even stack} of \cite{EvenStacks}, and we expect that the resulting \emph{Habiro stack} $X^\Hh$ is closely related to Scholze's construction \cite{HabiroCohomologyLecture}.
	\end{numpar}
	
	\begin{numpar}[Higher chromatic bases.]\label{par:HigherChromaticSpeculation}
		We would be very interested in the calculation for $\MU$ or any higher chromatic base like $\mathrm{BP}\langle n\rangle$ or $\E_n$, and we're curious to see whether the deformed de Rham complexes from \cite{DevalapurkarMisterka} make an appearance. The final goal should be to work directly with the refinement of
		\begin{equation*}
			\THH(-)\colon \Mot^\loc\longrightarrow \cat{CyctSp}\,,
		\end{equation*}
		valued in \emph{cyclotomic spectra}, and to describe the cyclotomic even stack of $\THHref(X)$ when $X$ is a $\IQ$-variety. The result might be close to the finest possible information that one can squeeze out of $\THH(-)$.
		
		Here we should point out that $\THHref(\IQ)$ is an $\IE_\infty$-algebra over the $K$-theory spectrum $K(\IQ)$, which vanishes upon $K(n)$-localisation for $n\geqslant 2$. Due to the delicate nature of the refinement, this doesn't mean that the answer over a higher chromatic base would be trivial, and $\TCref(-/\MU)$ should still contain strictly more information than $\TCref(-/\ku)$, but that information will necessarily be rather subtle.%
		\footnote{Here's one way to think about this: $\TCref(\ku\otimes\IQ/\ku)$ should see the algebraic locus where \enquote{$v_1\neq 0$}. In the world of adic spaces this corresponds to the condition \enquote{$\abs{v_1}\geqslant 1$}. We expect that $\TCref(\MU\otimes\IQ/\MU)$ is able to see a certain part the locus where \enquote{$0<\abs{v_1}<1$}.}
	\end{numpar}
	
	\subsection{Overview of this article}
	
	In \cref{sec:RefinedInvariants}, we study refined localising invariants in general. After a few generalities in \crefrange{subsec:TraceClassNuclear}{subsec:SmoothProper}, we'll explain a method to compute refinements in \cref{subsec:Recipe}. We'll then show in \cref{subsec:Burklund} that Burklund's $\IE_1$-structures satisfy the necessary assumptions for the method to be applicable.
	
	In \cref{sec:RefinedTC-}, we'll then apply the method to compute the homotopy groups of $\TCref((\ku\otimes\IQ)/\ku)$ and $\TCref((\KU\otimes\IQ)/\KU)$. We'll first derive a preliminary description in terms of certain $q$-Hodge filtrations $\fil_{\qHodge}^\star\qdeRham_{(\IZ/m)/\IZ}$ in \cref{subsec:RefinedTC-}. Afterwards, we'll construct explicit generators of these filtrations in \cref{subsec:ElementaryProof}. This will finally allow us to prove the explicit descriptions of \cref{thm:OverconvergentNeighbourhood,thm:OverconvergentNeighbourhoodEquivariant} in \cref{sec:Overconvergent}.
	
	\begin{numpar}[Notation and conventions.]\label{par:Notation}
		Throughout the article, we freely use the language of $\infty$-categories and we'll adopt the following conventions:
		\begin{alphanumerate}
			\item \textbf{Stable $\boldsymbol\infty$-categories.} We let $\Sp$ denote the $\infty$-category of spectra. For an ordinary ring $R$, we let $\Dd(R)$ denote the derived $\infty$-category of $R$. We often implicitly regard objects of $\Dd(R)$ as spectra via the Eilenberg--MacLane functor $\H$, but we'll always suppress this functor in our notation. For a stable $\infty$-category $\Cc$, we let $\Hom_\Cc(-,-)$ denote the mapping spectra in $\Cc$. The shift functor and its inverse will always be denoted by $\Sigma$ and $\Sigma^{-1}$ (even for $\Dd(R)$), to avoid confusion with shifts in graded or filtered objects.
			
			\item \textbf{Symmetric monoidal $\boldsymbol\infty$-categories.} If no confusion can occur, we denote the tensor unit by $\IUnit$ and the tensor product by $\otimes$. If $\Cc$ is symmetric monoidal, we let $\Alg(\Cc)$ and $\CAlg(\Cc)$ denote the $\infty$-categories of $\IE_1$-algebras and $\IE_\infty$-algebras in $\Cc$, respectively.
			
			Whenever we consider a symmetric monoidal $\infty$-category $\Cc$ which is stable or presentable, we always implicitly assume that the tensor product commutes with finite colimits or arbitrary colimits, respectively. In the presentable case, we let $\Hhom_\Cc(-,-)$ denote the \emph{internal Hom} in $\Cc$ and $X^\vee\coloneqq \Hhom_\Cc(X,\IUnit)$ the \emph{dual} of an object $X \in \Cc$.
			
			\item \textbf{Graded and filtered objects.} For a stable $\infty$-category $\Cc$, we let $\Gr(\Cc)$ and $\Fil(\Sp)$ denote the $\infty$-categories of \emph{graded} and \emph{\embrace{descendingly} filtered objects in $\Cc$}. The shift in graded or filtered objects will be denoted $(-)(1)$. An object with a descending filtration is typically denoted
			\begin{equation*}
				\fil^\star X=\Bigl(\dotsb\leftarrow \fil^nX\leftarrow \fil^{n+1}X\leftarrow\dotsb\Bigl)
			\end{equation*}
			and we let $\gr^*X$ denote the \emph{associated graded}, given by $\gr^nX\coloneqq \cofib(\fil^{n+1}X\rightarrow \fil^nX)$. We mostly work with filtrations that are constant in degrees $\leqslant 0$ (such as the Hodge filtration). In this case we'll abusingly write $\fil^\star X=(\fil^0X\leftarrow \fil^1 X\leftarrow\dotsb)$; this should be interpreted as the constant $\fil^0X$-valued filtration in degrees $\leqslant 0$.
			
			If $\Cc$ is symmetric monoidal and the tensor product $-\otimes-$ commutes with colimits in both variables, we equip $\Gr(\Cc)$ and $\Fil(\Cc)$ with their canonical symmetric monoidal structures given by Day convolution. We'll use the fact that $\Fil(\Cc)\simeq \Mod_{\IUnit_{\Gr}[t]}\Gr(\Cc)$, where $\IUnit_{\Gr}$ denotes the tensor unit in $\Gr(\Cc)$ and $t$ sits in graded degree~$-1$; see e.g.\ \cite[Proposition~\chref{3.2.9}]{RaksitFilteredCircle}. Under this equivalence, passing to the associated graded corresponds to \enquote{modding out~$t$}, i.e.\ the base change $\IUnit_{\Gr}\otimes_{\IUnit_{\Gr}[t]}-$.
			
			Sometimes we also consider \emph{ascending} filtrations. Ascendingly filtered objects will be denoted $\fil_\star X=(\dotsb\rightarrow \fil_nX\rightarrow \fil_{n+1}X\rightarrow \dotsb)$ and the associated graded by $\gr_*X$, where $\gr_nX\coloneqq \cofib(\fil^{n-1}X\rightarrow \fil^nX)$.\label{conv:Filtrations}
			
			\item \textbf{Condensed mathematics.} Whenever we use condensed mathematics, we work in the light condensed setting. We'll distinguish between the words \emph{static} (\enquote{un-animated}) for a spectrum concentrated in degree~$0$, and \emph{discrete} (\enquote{un-condensed}) for a condensed spectrum with the discrete topology.
			
			\item \textbf{Derived quotients.} For an $\IE_1$-ring spectrum $R$, a homotopy class $f\in\pi_n(R)$, and a left- or right-$R$-module $M$, we denote 
			\begin{equation*}
				M/f\coloneqq \cofib\left(f\colon \Sigma^n M\rightarrow M\right)\,.
			\end{equation*}
			For several homotopy classes $f_1,\dotsc,f_r$, we let $M/(f_1,\dotsc,f_r)\coloneqq (\dotsb(M/f_1)/f_2\dotsb)/f_r$. Similarly, if $R^*$ is a graded $\IE_1$-ring spectrum, $f\in \pi_n(R^i)$, and $M^*$ is a left or right-$R$-module, we put
			\begin{equation*}
				M^*/f\coloneqq \cofib\bigl(f\colon \Sigma^n M(i)\rightarrow M\bigr)
			\end{equation*}
			and define $M^*/(f_1,\dotsc,f_r)$ analogously. The same notation will also be used in the filtered setting, by regarding filtered objects as graded $\IUnit_{\Gr}[t]$-modules, as explained above.
			
			\item \textbf{Completions.} For an $\IE_\infty$-ring spectrum $R$, finitely many  homogeneous homotopy classes $f_1,\dotsc,f_r\in\pi_*(R)$, and and an $R$-module spectrum $M$, we let
			\begin{equation*}
				\widehat{M}_{(f_1,\dotsc,f_r)}\coloneqq \lim_{n\geqslant 1}M/(f_1^n,\dotsc,f_r^n)
			\end{equation*}
			denote the \emph{$(f_1,\dotsc,f_r)$-adic completion} of $M$. Since the completion only depends on the ideal $I=(f_1,\dotsc,f_r)\subseteq \pi_*(R)$, we often just write $\widehat{M}_I$ (or $(-)_I^\complete$ for longer arguments). If $R$ is an ordinary ring, this recovers the notion of \emph{derived $I$-completion}; in particular, all completions in this article will be derived. For the $p$-completions of $\IZ$ and the sphere spectrum $\IS$ we omit the hat and just write $\IZ_p$ and $\IS_p$.
			
			We let $\Mod_R(\Sp)_{I}^\complete\subseteq \Mod_R(\Sp)$, or $\widehat{\Dd}_{I}(R)\subseteq \Dd(R)$ for ordinary rings $R$, denote the full sub-$\infty$-category spanned by the \emph{$I$-complete objects}, that is, those $M$ for which $M\simeq \widehat M_{I}$. The following fact will be used countless times: If $M$ is $(f_1,\dotsc,f_r)$-complete, and the homotopy groups of $M/(f_1,\dotsc,f_r)$ vanish in some degree~$d$, then also the homotopy groups of $M$ must vanish in degree~$d$. Completion can analogously be defined in the graded or filtered setting, and then an analogue of this fact will still be true.
			
			\item \textbf{Derived ($\boldsymbol q$-)de Rham complexes.} We let $\deRham_{R/A}$ and $\qdeRham_{R/A}$ denote the derived de Rham complex and the derived $q$-de Rham complex of $R$ over $A$, respectively (the latter is only defined if $A$ is a $\Lambda$-ring).
		\end{alphanumerate}
		
		
	\end{numpar}
	\begin{numpar}[Acknowledgments.]
		We are grateful to Peter Scholze and Sasha Efimov for proposing this question and explaining many technical points of the theory. Moreover, it was Scholze who pointed out that the filtration on $\qdeRham_{(\IZ/p^\alpha)/\IZ_p}$, that we found in the homotopy groups of $\TC^-((\ku/p^\alpha)/\ku)$, should indeed be canonical, despite the second author's initial conviction that this couldn't possibly be true---this observation is what led the second author to revisit the theory of $q$-Hodge filtrations/complexes in \cite{qWittHabiro,qdeRhamku}. Special thanks are also due to Sanath Devalapurkar and Arpon Raksit for generously sharing and explaining their (by then) unpublished results on the connection between $q$-de Rham cohomology and $\ku$. Furthermore, would like to thank Gabriel Angelini-Knoll, Johannes Anschütz, Ben Antieau, Ko Aoki, Guido Bosco, Robert Burklund, Jeremy Hahn, Lars Hesselholt, Deven Manam, Florian Riedel, and Juan Esteban Rodr\'{i}guez Camargo for helpful discussions.
		
		This work was carried out while F.W.\ was a Ph.D.\ student at the University/Max Planck Institute for Mathematics in Bonn and he would like to thank these institutions for their hospitality. F.W.\ was supported by DFG through Peter Scholze's Leibniz-Preis.
	\end{numpar}
	
	\newpage
	
	\section{Refined localising invariants and how to compute them}\label{sec:RefinedInvariants}
	
	In this section we'll present Efimov--Scholze's construction of refined localising invariants and we'll explain a method for computing them in the case of certain \enquote{open submotives} of \enquote{smooth and proper} rigid symmetric monoidal $\infty$-categories over some base (these notions will be made precise below). As a consequence, we'll get a recipe for computing $\THHref(\IQ)$, which we'll carry out (after base change to $\ku$) in \crefrange{sec:RefinedTC-}{sec:Overconvergent}, but the method would apply just as well to other cases like $\THHref(\L_n^f\IS_{(p)}/\IS_{(p)})$ or $\THHref(\IS[x])$.
	
	\subsection{Trace-class morphisms and nuclear objects}\label{subsec:TraceClassNuclear}
	
	In this subsection we briefly review the two notions in the title. These will be used countless times in the rest of the article. Throughout, we let $\Cc$ be a presentable symmetric monoidal $\infty$-category; by convention (see \cref{par:Notation}), this includes the assumption that the tensor product commutes with colimits in both variables.
	\begin{defi}\label{def:TraceClass}
		A morphism $\varphi\colon X\rightarrow Y$ in $\Cc$ is called \emph{trace-class} if there exists morphism $\eta\colon \IUnit \rightarrow X^\vee\otimes Y$ in $\Cc$ such that $\varphi$ is the composition
		\begin{equation*}
			X\simeq X\otimes\IUnit\xrightarrow{X\otimes \eta} X\otimes X^\vee\otimes Y\xrightarrow{\ev_X\otimes Y}\IUnit \otimes Y\simeq Y\,.
		\end{equation*}
		We often call $\eta$ the \emph{classifier of $\varphi$} and say that \emph{$\eta$ witnesses $\varphi$ being trace-class}.
	\end{defi}
	
	Trace-class morphism have a number of nice properties. We'll often use the properties from  \cite[Lemma~\chref{8.2}]{Complex} as well as the following lemma.
	\begin{lem}\label{lem:TraceClassAbstractNonsense}
		Let $F\colon \Cc\rightarrow \Dd$ be a symmetric monoidal functor between presentable symmetric monoidal $\infty$-categories. By abuse of notation, we use $(-)^\vee$ to denote both the predual in $\Cc$ and in $\Dd$.
		\begin{alphanumerate}
			\item There exists a natural transformation $F((-)^\vee)\Rightarrow F(-)^\vee$.\label{enum:DualTransformation}
			\item If $X\rightarrow Y$ is trace-class in $\Cc$, then $Y^\vee\rightarrow X^\vee$ is trace-class in $\Cc$ and $F(X)\rightarrow F(Y)$ is trace-class in $\Dd$.\label{enum:DualTraceClass}
			\item The commutative square in $\Dd$ formed by the morphisms from \cref{enum:DualTransformation} and \cref{enum:DualTraceClass}\label{enum:DualTraceClassLift}
			\begin{equation*}
				\begin{tikzcd}
					F(Y^\vee)\rar\dar & F(X^\vee)\dar\\
					F(Y)^\vee\rar\urar[dashed] & F(X)^\vee
				\end{tikzcd}
			\end{equation*}
			admits a canonical diagonal map $F(Y)^\vee\rightarrow F(X^\vee)$ that makes both triangles commute.
		\end{alphanumerate}
	\end{lem}
	\begin{proof}
		The natural transformation from \cref{enum:DualTransformation} is adjoint to $F((-)^\vee)\otimes_{\Dd} F(-)\Rightarrow \IUnit_\Dd$, which is in turn given by applying $F$ to the evaluation $(-)^\vee\otimes (-)\Rightarrow \IUnit_\Cc$.
		
		Now let $X\rightarrow Y$ be trace-class in $\Cc$ with classifier $\IUnit_\Cc\rightarrow X^\vee\otimes Y$. If we apply $F$ to the classifier and compose with the morphism $F(X^\vee)\rightarrow F(X)^\vee$ from \cref{enum:DualTransformation}, we obtain a morphism $\IUnit_\Dd\rightarrow F(X^\vee)\otimes_{\Dd} F(Y)\rightarrow F(X)^\vee\otimes_{\Dd} F(Y)$, which witnesses $F(X)\rightarrow F(Y)$ being trace-class. If we compose instead with $Y\rightarrow Y^{\vee\vee}$, we obtain $\IUnit_\Cc\rightarrow X^\vee\otimes_\Cc Y\rightarrow X^\vee\otimes_\Cc Y^{\vee\vee}$, which witnesses $Y^\vee\rightarrow X^\vee$ being trace-class. This shows \cref{enum:DualTraceClass}. To show \cref{enum:DualTraceClassLift}, we construct the diagonal map $F(Y)^\vee\rightarrow F(X^\vee)$ as follows:
		\begin{equation*}
			F(Y)^\vee\longrightarrow F(X^\vee\otimes_\Cc Y)\otimes_\Dd F(Y)^\vee\simeq F(X^\vee)\otimes_{\Dd}F(Y)
			\otimes_\Dd F(Y)^\vee\longrightarrow F(X^\vee)\,.
		\end{equation*}
		Here we use the classifier $\IUnit_\Cc\rightarrow X^\vee\otimes_\Cc Y$ and the evaluation map for $F(Y)$.
	\end{proof}
	
	\begin{defi}\label{def:Nuclear}
		In addition to the assumptions above, let us now assume that $\Cc$ is stable, compactly generated, and the tensor unit $\IUnit$ is compact.
		\begin{alphanumerate}
			\item An object $X\in\Cc$ is called \emph{nuclear} if every morphism $P\rightarrow X$ from a compact object $P$ is trace-class.
			\item We call $X$ \emph{basic nuclear} if it can be written as $X\simeq \colimit (X_0\rightarrow X_1\rightarrow \dotsb)$ such that each transition map $X_n\rightarrow X_{n+1}$ is trace-class.
		\end{alphanumerate}
		We let $\Nuc(\Cc)\subseteq \Cc$ denote the full sub-$\infty$-category spanned by the nuclear objects.
	\end{defi}
	\begin{thm}\label{thm:Nuclear}
		Let $\Cc$ be a presentable stable symmetric monoidal $\infty$-category such that $\Cc$ is compactly generated and the tensor unit $\IUnit\in\Cc$ is compact.
		\begin{alphanumerate}
			\item $\Nuc(\Cc)$ is stable and closed under colimits and tensor products in $\Cc$.\label{enum:Nuc}
			\item $\Nuc(\Cc)$ is $\omega_1$-compactly generated and the $\omega_1$-compact objects are precisely the basic nuclears.\label{enum:Nucw1Generated}
			\item If $F\colon \Cc\rightarrow \Dd$ is a symmetric monoidal colimit-preserving functor into another presentable symmetric monoidal $\infty$-category, then $F$ restricts to a functor $F\colon \Nuc(\Cc)\rightarrow \Nuc(\Dd)$.\label{enum:NucFunctorial}
		\end{alphanumerate}
	\end{thm}
	\begin{proof}
		Parts~\cref{enum:Nuc} and~\cref{enum:Nucw1Generated} are \cite[Theorem~\chref{8.6}]{Complex}. By \cref{lem:TraceClassAbstractNonsense}\cref{enum:DualTraceClass}, $F$ preserves trace-class maps, hence basic nuclear objects and thus all nuclear objects by \cref{enum:Nucw1Generated}. This proves~\cref{enum:NucFunctorial}.
	\end{proof}
	
	\begin{rem}\label{rem:NucInd}
		If $\Cc$ is a small stable symmetric monoidal $\infty$-category, then \cref{thm:Nuclear} can be applied to $\Ind(\Cc)$. Since every trace-class map in $\Ind(\Cc)$ factors through a compact object by \cite[Lemma~\chref{8.4}]{Complex}, we see that the basic nuclear objects in $\Ind(\Cc)$ are of the form $\indcolim(X_1\rightarrow X_2\rightarrow \dotsb)$, where each $X_n\rightarrow X_{n+1}$ is trace-class in $\Cc$.
		
		If $\Cc$ is a presentable stable symmetric monoidal $\infty$-category, one can still make sense of $\Nuc\Ind(\Cc)$ without running into set-theoretic problems. Indeed, if $\kappa$ is a sufficiently large regular cardinal such that $\Cc$ is $\kappa$-compactly generated and $\IUnit$ is $\kappa$-compact, then every trace-class morphism in $\Cc$ factors through a $\kappa$-compact object. Thus every basic nuclear ind-object is equivalent to one in which each $X_n$ is $\kappa$-compact and so the basic nuclear objects in form an essentially small $\infty$-category. We may then define $\Nuc\Ind(\Cc)$ as $\Ind_{\omega_1}(-)$ of the $\infty$-category of basic nuclear objects.
	\end{rem}

	\subsection{Killing (pro-)algebra objects}\label{subsec:KillingAlgebras} 
	
	In this subsection we review the general formalism for passing to the \enquote{open complement} of an algebra object. We'll follow \cite[Lecture~\href{https://youtu.be/38PzTzCiMow?list=PLx5f8IelFRgGmu6gmL-Kf_Rl_6Mm7juZO&t=5523}{13}]{AnalyticStacks}. Throughout, let's fix a presentable stable symmetric monoidal $\infty$-category $\Cc$.
	
	\begin{numpar}[Killing algebras.]\label{par:KillingAlgebras}
		Let $A\in\Cc$ be an object equipped maps $\mu\colon A\otimes A\rightarrow A$ and $\IUnit\rightarrow A$ such that $\mu$ is left-unital (or right-unital; this doesn't matter). We let $\Cc^A\subseteq \Cc$ be the full sub-$\infty$-category spanned by those $U\in\Cc$ for which
		\begin{equation*}
			\Hhom_\Cc(A,U)\simeq 0\,,
		\end{equation*}
		where $\Hhom_\Cc$ denotes the internal $\Hom$ of $\Cc$ (see \cref{par:Notation}), Clearly $\Cc^A$ is closed under limits in $\Cc$. If $\kappa$ is a sufficiently large cardinal such that $S\otimes A$ are $\kappa$-compact for all $S$ in a set of generators for $\Cc$, then $\Cc^A$ is also closed under $\kappa$-filtered colimits. By the $\infty$-categorical reflection theorem \cite{ReflectionTheorem}, it follows that the inclusion $\Cc^A\rightarrow \Cc$ admits a left adjoint $j^*\colon \Cc\rightarrow \Cc^A$. Since $\Cc^A$ is also clearly closed under $\Hhom_\Cc(Y,-)$ for any $Y\in\Cc$, we see that
		\begin{equation*}
			j^*(X\otimes Y)\overset{\simeq}{\longrightarrow}j^*\bigl(j^*(X)\otimes Y\bigr)
		\end{equation*}
		is an equivalence for all $X,Y\in\Cc$. By abstract nonsense about symmetric monoidal localisations (see \cite[Proposition~\chref{2.2.1.9}]{HA}), it follows that $\Cc^A$ and $j^*\colon \Cc\rightarrow \Cc^A$ can be equipped with canonical symmetric monoidal structures and the inclusion $\Cc^A\rightarrow \Cc$ with a lax symmetric monoidal structure. In particular, $j^*(\IUnit)$ is an $\IE_\infty$-algebra in $\Cc$. We'll often say that \emph{$\Cc^A$ is obtained from $\Cc$ by killing~$A$} and \emph{$j^*(\IUnit)$ is obtained from $\IUnit$ by killing $A$}.
	\end{numpar}
	Our first goal is now to give a formula for $j^*$ in certain cases.
	\begin{lem}\label{lem:KillingAlgebras}
		Let $\Ii\coloneqq\fib(\IUnit\rightarrow A)$. Then for every $X\in \Cc$ the canonical map
		\begin{equation*}
			\eta_X\colon X\simeq \Hhom_\Cc(\IUnit,X)\longrightarrow \Hhom_\Cc(\Ii,X)
		\end{equation*}
		becomes an equivalence upon applying $\Hom_\Cc(-,U)$ for any $U\in \Cc^A$.
	\end{lem}
	\begin{proof}
		It's enough to show $\Hom_\Cc(\fib(\eta_X),U)\simeq 0$. Note that the fibre $\fib(\eta_X)\simeq \Hhom_\Cc(A,X)$ is a weak $A$-module in the sense that there exists a unital multiplication map
		\begin{equation*}
			A\otimes \Hhom_\Cc(A,X)\rightarrow \Hhom_\Cc(A,X)\,.
		\end{equation*}
		In particular, $\Hhom_\Cc(A,X)$ is a retract of $A\otimes\Hhom_\Cc(A,X)$ and so it suffices to show that $\Hom_\Cc(-,U)$ vanishes on the latter. Now $\Hom_\Cc(A\otimes Y,U)\simeq \Hom_\Cc(Y,\Hhom_\Cc(A,U))\simeq 0$ holds for all $Y\in\Cc$, so we conclude.
	\end{proof}
	\begin{prop}\label{prop:KillingAlgebras}
		With notation as above, suppose that one of the following two conditions is satisfied:
		\begin{alphanumerate}
			\item For all $X\in\Cc$, we recursively put $X_0\coloneqq X$ and $X_{n+1}\coloneqq \Hhom_\Cc(\Ii,X_n)$. Then the diagram\label{enum:KillingAlgebrasIdempotent}
			\begin{equation*}
				X\overset{\eta_X}{\longrightarrow}X_1\xrightarrow{\eta_{X_1}}X_2\xrightarrow{\eta_{X_2}}\dotsb
			\end{equation*}
			stabilises at some finite stage \embrace{for example, this is satisfied if $A$ is idempotent---then the colimit always stabilises after the first step}.
			\item The functor $\Hhom_\Cc(A,-)$ commutes with sequential colimits \embrace{for example, this is satisfied if $A$ is dualisable in $\Cc$}.\label{enum:KillingAlgebrasCompact}
		\end{alphanumerate}
		Then $j^*(X)$ is the colimit of the diagram from \cref{enum:KillingAlgebrasIdempotent} for all $X\in\Cc$. 
	\end{prop}
	\begin{proof}
		Let us denote the colimit of the diagram from \cref{enum:KillingAlgebrasIdempotent} by $X_\infty$. Then \cref{lem:KillingAlgebras} ensures that $\Hom_\Cc(X_\infty,U)\rightarrow \Hom_\Cc(X,U)$ is an equivalence for all $U\in\Cc^A$, so we only need to check $X_\infty\in\Cc^A$; that is, $\Hhom_\Cc(A,X_\infty)\simeq 0$. Equivalently, $\eta_{X_\infty}\colon X_\infty \rightarrow \Hhom_\Cc(\Ii,X_\infty)$ needs to be an equivalence. But either of the two assumptions above makes sure that $\Hhom_\Cc(\Ii,-)$ commutes with the colimit defining $X_\infty$ and so $\eta_{X_\infty}$ is an equivalence by construction.
	\end{proof}
	We'll now explain a variant of the construction above in a pro-/ind-setting.
	\begin{numpar}[Killing pro-algebras]
		We keep $\Cc$ a presentable symmetric monoidal stable $\infty$-category. The tensor product on $\Cc$ extends to symmetric monoidal structures on $\Pro(\Cc)$ and $\Ind(\Cc)$.%
		\footnote{We'll ignore the set-theoretic difficulties that arise with applying $\Pro(-)$ and $\Ind(-)$ to large $\infty$-categories. In all cases of interest, we can safely replace $\Cc$ by its $\kappa$-compact objects $\Cc^\kappa\subseteq \Cc$ for some large enough regular cardinal $\kappa$ (usually $\kappa=\omega_1$ is enough).}
		Observe that $\Hhom_\Cc$ can also be extended to a functor
		\begin{equation*}
			\Pro(\Cc)^\op\otimes\Ind(\Cc)\simeq \Ind(\Cc^\op)\otimes\Ind(\Cc)\xrightarrow{\Ind(\Hhom_\Cc)} \Ind(\Cc)\,,
		\end{equation*}
		which, by abuse of notation, we still denote $\Hhom_\Cc$.
		Explicitly,
		\begin{equation*}
			\Hhom_\Cc\Bigl(\prolim_{j\in J} Y_j,\indcolim_{k\in K} Z_k\Bigr)\simeq \indcolim_{(j,k)\in J^\op\times K} \Hhom_\Cc(Y_j,Z_k)\,.
		\end{equation*}
		Let now $A\coloneqq\prolim_{i\in I} A_i\in \Pro(\Cc)$ be a pro-object equipped with maps $\mu\colon A\otimes A\rightarrow A$ and $\IUnit\rightarrow A$ such that $\mu$ is left-unital. We let $\Ind(\Cc)^A\subseteq \Ind(\Cc)$ denote the full sub-$\infty$-category spanned by those ind-objects for which
		\begin{equation*}
			\Hhom_\Cc(A,M)\simeq 0\,.
		\end{equation*}
		Our goal is again to describe a left adjoint $j^*\colon \Ind(\Cc)^A\rightarrow \Ind(\Cc)$ of the inclusion. To this end, let $\Ii\coloneqq \fib(\IUnit\rightarrow A)$ and consider the canonical maps $\eta_X\colon X\simeq \Hhom_\Cc(\IUnit,X)\rightarrow \Hhom_\Cc(\Ii,X)$ for all $X\in\Ind(\Cc)$, as in \cref{lem:KillingAlgebras}.
	\end{numpar}
	\begin{lem}\label{lem:KillingProAlgebras}
		The inclusion of $\Ind(\Cc)^A$ admits a left adjoint $j^*\colon \Ind(\Cc)\rightarrow\Ind(\Cc)^A$, which can be explicitly described as follows: For $X\in \Cc$ we recursively put $X_0\coloneqq X$ and $X_{n+1}\coloneqq \Hhom_\Cc(\Ii,X_n)$. Then
		\begin{equation*}
			j^*(X)\simeq \colimit\Bigl(X\overset{\eta_X}{\longrightarrow}X_1\xrightarrow{\eta_{X_1}}X_2\xrightarrow{\eta_{X_2}}\dotsb\Bigr)\,.
		\end{equation*}
	\end{lem}
	\begin{proof}
		Since $\Hhom_\Cc(A,-)\colon \Ind(\Cc)\rightarrow \Ind(\Cc)$ preserves filtered colimits, we can argue as in the proof of \cref{prop:KillingAlgebras} to see that $j^*(X)\in\Ind(\Cc)^A$. It remains to show that the canonical morphism $X\rightarrow j^*(X)$ induces equivalences
		\begin{equation*}
			\Hom_{\Ind(\Cc)}\bigl(j^*(X),U\bigr)\overset{\simeq}{\longrightarrow}\Hom_{\Ind(\Cc)}(X,U)
		\end{equation*}
		for all $U\in \Ind(\Cc)^A$. It will be enough to show the same for $\eta_X$, or equivalently, that $\Hom_{\Ind(\Cc)}(\Hhom_\Cc(A,X),U)\simeq 0$. To this end, let $M\in\Ind(\Cc)$ be any object for which the natural transformation $\Hhom_\Cc(A,-)\Rightarrow \Hhom_\Cc(\IUnit,-)\simeq (-)$ admits a section.%
		\footnote{Intuitively, the condition should be that $M$ admits a unital multiplication $A\otimes M\rightarrow M$, but this doesn't make sense in our setting. So we replace this by the condition that $\Hhom_\Cc(A,M)\rightarrow M$  admits a section.}
		Via such a section $M\rightarrow \Hhom_\Cc(A,M)$, the identity on $\Hom_{\Ind(\Cc)}(M,U)$ factors through
		\begin{equation*}
			\Hom_{\Ind(\Cc)}\bigl(\Hhom_\Cc(A,M),\Hhom_\Cc(A,U)\bigr)\simeq 0\,,
		\end{equation*}
		and so $\Hom_{\Ind(\Cc)}(M,U)\simeq 0$. Since such a section exists for $M=\Hhom_\Cc(A,X)$, we conclude.
	\end{proof}
	\begin{numpar}[Killing idempotent pro-algebras.]
		Suppose that $A$ is \emph{idempotent} in $\Pro(\Cc)$, that is, $\IUnit\rightarrow A$ induces an equivalence
		\begin{equation*}
			A\simeq \IUnit\otimes A\overset{\simeq}{\longrightarrow} A\otimes A\,.
		\end{equation*}
		Let us spell out how $j^*(\IUnit)$ looks like in this case: We write $A= \prolim A_i$ and denote by $(-)^\vee\coloneqq \Hhom_\Cc(-,\IUnit)$ the predual in $\Cc$. Then \cref{lem:KillingProAlgebras} implies that there is a cofibre sequence
		\begin{equation*}
			\indcolim_{i\in I^\op}A_i^\vee\longrightarrow \IUnit\longrightarrow j^*(\IUnit)\,.
		\end{equation*}
		For idempotent $A$, we check in \cref{lem:ProIdempotentAbstract} below that $j^*\colon \Ind(\Cc)\rightarrow \Ind(\Cc)^A$ can be equipped with a symmetric monoidal structure (we don't know if this works in general---the argument from \cref{par:KillingAlgebras} doesn't seem to work anymore). As a consequence, $j^*(\IUnit)$ will be an $\IE_\infty$-algebra in $\Ind(\Cc)$. We'll say that \emph{$j^*(\IUnit)$ is obtained from $\IUnit$ by killing the idempotent pro-algebra $A$}.
	\end{numpar}
	\begin{lem}\label{lem:ProIdempotentAbstract}
		Suppose that $A$ is an idempotent pro-object. Then for all $X,Y\in\Ind(\Cc)$, the canonical morphism
		\begin{equation*}
			j^*(X\otimes Y)\overset{\simeq}{\longrightarrow}j^*\bigl(j^*(X)\otimes Y\bigr)
		\end{equation*}
		is an equivalence. In particular, there's a canonical way to equip $j^*\colon \Ind(\Cc)\rightarrow \Ind(\Cc)^A$ with a symmetric monoidal structure.
	\end{lem}
	\begin{proof}
		By \cref{lem:KillingProAlgebras} and idempotence of $A$, $j^*(X)\simeq \cofib(\Hhom_\Cc(A,X)\rightarrow X)$. Thus, to show the first assertion, we may equivalently show that the canonical morphism
		\begin{equation*}
			\Hhom_\Cc\bigl(A,\Hhom_\Cc(A,X)\otimes Y\bigr)\overset{\simeq}{\longrightarrow}\Hhom_\Cc(A,X)\otimes Y
		\end{equation*}
		induced by $\IUnit\rightarrow A$ is an equivalence. To see this, first observe that this morphism has a left inverse given by
		\begin{equation*}
			\Hhom_\Cc(A,X)\otimes Y\simeq \Hhom_\Cc\bigl(A,\Hhom_\Cc(A,X)\bigr)\otimes Y\longrightarrow \Hhom_\Cc\bigl(A,\Hhom_\Cc(A,X)\otimes Y\bigr)
		\end{equation*}
		using idempotence of $A$ and $Y\simeq \Hhom_\Cc(\IUnit,Y)$. Now, in general, let $M\in \Ind(\Cc)$ be an ind-object for which $\Hhom_\Cc(A,M)\rightarrow M$ has a left inverse. We can then exhibit $\Hhom_\Cc(A,M)\rightarrow M$ as a retract of $\Hhom_\Cc(A,\Hhom_\Cc(A,M))\rightarrow \Hhom_\Cc(A,M)$. But the latter is an equivalence by pro-idempotence of $A$, so already $\Hhom_\Cc(A,M)\rightarrow M$ must be an equivalence.
		
		This finishes the proof that $j^*(X\otimes Y)\rightarrow j^*(j^*(X)\otimes Y)$ is an equivalence. By abstract nonsense about symmetric monoidal structures on localisations (see \cite[Proposition~\chref{2.2.1.9}]{HA}), it follows that $j^*$ can be canonically equipped with a symmetric monoidal structure.
	\end{proof}
	
	\begin{rem}
		In general, $j^*(\IUnit)$ is not an idempotent $\IE_\infty$-algebra in $\Ind(\Cc)$; it is idempotent if and only if $A^\vee\coloneqq\indcolim_{i\in I^\op}A_i^\vee$ is an \emph{ind-idempotent coalgebra} in the sense that $A^\vee\rightarrow \IUnit$ induces an equivalence $A^\vee\otimes A^\vee\simeq A^\vee$ in $\Ind(\Cc)$.
	\end{rem}
	
	In the following lemma we'll study a special situation in which this is the case.
	
	\begin{lem}\label{lem:NuclearIdempotentAbstract}
		Let $A = \prolim_{i\in I}A_i$ be an idempotent pro-object whose transition maps are eventually trace-class in the sense that for all $i\in I$ there exists an object $j\rightarrow i$ such that $A_j\rightarrow A_i$ is trace-class. Let $A^\vee\coloneqq \indcolim_{i\in I}A_i^\vee$. Then the canonical map
		\begin{equation*}
			X\otimes A^\vee\overset{\simeq}{\longrightarrow} \Hhom_\Cc(A,X)
		\end{equation*}
		is an equivalence for all $X\in \Ind(\Cc)$. In particular, this implies:
		\begin{alphanumerate}
			\item $A^\vee$ is an idempotent coalgebra in $\Ind(\Cc)$ with eventually trace-class transition maps.\label{enum:IndIdempotentCoalgebra}
			\item $j^*(\IUnit)$ is an idempotent nuclear $\IE_\infty$-algebra in $\Ind(\Cc)$, $\Ind(\Cc)^A\subseteq \Ind(\Cc)$ is precisely the full sub-$\infty$-category of $j^*(\IUnit)$-modules, and $-\otimes j^*(\IUnit)\simeq j^*(-)$.\label{enum:NuclearIdempotentAbstract}
			\item If $F\colon \Cc\rightarrow \Dd$ is any symmetric monoidal functor of presentable symmetric monoidal $\infty$-categories, then $F(j^*(\IUnit))$ is obtained by killing the idempotent pro-algebra $F(A)$.\label{enum:IdempotentBasechange}
		\end{alphanumerate}
	\end{lem}
	\begin{proof}[Proof sketch]
		We can construct an inverse of $X\otimes A^\vee\rightarrow \Hhom_\Cc(A,X)$ as follows: Fix some $i\in I$, choose $j\rightarrow i$ such that $A_j\rightarrow A_i$ is trace-class and let $\IUnit\rightarrow A_i\otimes A_j^\vee$ be the corresponding classifier. Then consider the composition
		\begin{equation*}
			\Hhom_\Cc(A_i,X)\longrightarrow \Hhom_\Cc(A_i,X)\otimes A_i\otimes A_j^\vee\longrightarrow X\otimes A_j^\vee\,.
		\end{equation*}
		In the first map, we tensor $\Hhom_\Cc(A_i,X)$ with the classifier above. In the second map we use the evaluation $\Hhom_\Cc(A_i,X)\otimes A_i\rightarrow X$. It's straightforward but a little tedious to check that
		\begin{gather*}
			X\otimes A_i^\vee\longrightarrow \Hhom_\Cc(A_i,X)\longrightarrow X\otimes A_j^\vee\,\\
			\Hhom_\Cc(A_i,X)\longrightarrow X\otimes A_j^\vee\longrightarrow \Hhom_\Cc(A_j,X)
		\end{gather*}
		agree with the transition maps in the ind-objects $X\otimes A^\vee$ and $\Hhom_\Cc(A,X)$, respectively; we'll omit the argument.
		
		Proving that these maps assemble into an inverse map $X\otimes A^\vee\rightarrow\Hhom_\Cc(A,X)$ requires a non-trivial argument, since we're working in an $\infty$-category, but there's an easier way to show that $X\otimes A^\vee\rightarrow \Hhom_\Cc(A,X)$ is an equivalence: Equivalences are detected by $\pi_0\Hom_{\Ind(\Cc)}(Z,-)$, where $Z$ ranges through all compact objects of $\Ind(\Cc)$; now any morphism from a compact object factors through $X\otimes A_i^\vee$ or $\Hhom_\Cc(A_i,X)$ for some $i\in I$, and so the observations above will be enough.
		
		To show \cref{enum:IndIdempotentCoalgebra}, plug in $X\simeq A^\vee$: We obtain $A^\vee\otimes A^\vee\simeq \Hhom_\Cc(A,A^\vee)\simeq (A\otimes A)^\vee$. This proves idempotence as a coalgebra, because $(A\otimes A)^\vee\simeq A^\vee$ follows by dualising $A\simeq A\otimes A$. If $j\rightarrow i$ is large enough so that $A_j\rightarrow A_i$ is trace-class, then the dual transition map $A_i^\vee\rightarrow A_j^\vee$ is again trace-class by \cref{lem:TraceClassAbstractNonsense}\cref{enum:DualTraceClass}. This shows \cref{enum:IndIdempotentCoalgebra}.
		
		For \cref{enum:NuclearIdempotentAbstract}, since we've shown that $A^\vee$ is an idempotent coalgebra in $\Ind(\Cc)$, it follows that $j^*(\IUnit)$ is an idempotent algebra. Also $A^\vee$ is a nuclear object in $\Ind(\Cc)$, since every map $Z\rightarrow A^\vee$ from a compact object factors through a trace-class morphism and is therefore trace-class itself. Since $\IUnit$ is nuclear too, it follows that $j^*(\IUnit)$ is nuclear. $X\otimes j^*(\IUnit)\simeq j^*(X)$ follows immediately from the above equivalence $X\otimes A^\vee\simeq \Hhom_\Cc(A,X)$. Since the inclusion $\Ind(\Cc)^A\rightarrow \Ind(\Cc)$ is lax monoidal by \cref{lem:ProIdempotentAbstract}, it factors through a functor
		\begin{equation*}
			\Ind(\Cc)^A\rightarrow \Mod_{j^*(\IUnit)}\bigl(\Ind(\Cc)\bigr)\,.
		\end{equation*}
		Since $j^*(\IUnit)$ is idempotent, $\Mod_{j^*(\IUnit)}(\Ind(\Cc))\subseteq \Ind(\Cc)$ is the full sub-$\infty$-category spanned by the objects of the form $X\otimes j^*(\IUnit)$. Hence we also get an inclusion $\Ind(\Cc)^A\subseteq \Mod_{j^*(\IUnit)}(\Ind(\Cc))$. On the other hand, every object of the form $X\otimes j^*(\IUnit)\simeq j^*(X)$ is contained in $\Ind(\Cc)^A$. This finishes the proof of \cref{enum:NuclearIdempotentAbstract}.
		
		To show~\cref{enum:IdempotentBasechange}, we only need $\indcolim_{i\in I}F(A_i^\vee)\simeq \indcolim_{i\in I}F(A_i)^\vee$. If $A_j\rightarrow A_i$ is trace-class, \cref{lem:TraceClassAbstractNonsense}\cref{enum:DualTraceClassLift} provides a map $F(A_i)^\vee\rightarrow F(A_j^\vee)$ in the reverse direction. By a formal argument as above, this is enough to show the desired equivalence.
	\end{proof}

	\subsection{Generalities on refined localising invariants}\label{subsec:SmoothProper}
	
	Throughout this subsection and the next, we fix the following notation: Let $\Pr_{\mathrm{st}}^\L$ denote the $\infty$-category of presentable stable $\infty$-categories and colimit-preserving functors. For a regular cardinal~$\kappa$, we denote by $\Pr_{\mathrm{st},\kappa}^\L\subseteq \Pr_\mathrm{st}^\L$ the non-full sub-$\infty$-category spanned by the $\kappa$-compactly generated presentable stable $\infty$-categories and those colimit-preserving functors that also preserve $\kappa$-compact objects (equivalently, the right adjoint preserves $\kappa$-filtered colimits). We equip these $\infty$-categories with the Lurie tensor product and we let $\Pr_\mathrm{st}^\mathrm{dual}\subseteq \Pr_\mathrm{st}^\L$ denote the non-full sub-$\infty$-category spanned by the dualisable objects and the \emph{strongly continuous} functors, that is, those functors whose right adjoint still preserves all colimits.
	
	We also let $\Ee\in \CAlg(\Pr_{\mathrm{st}}^\L)$ be a rigid presentable stable symmetric monoidal $\infty$-category in the sense of \cref{def:Rigid}. We denote
	\begin{equation*}
		\Pr_\Ee^\L\coloneqq \Mod_\Ee(\Pr_{\mathrm{st}}^\L)\quad\text{and}\quad \Pr_{\Ee,\kappa}^\L\coloneqq \Mod_\Ee(\Pr_{\mathrm{st},\kappa}^\L)\,,
	\end{equation*}
	the latter assuming that $\Ee$ is $\kappa$-compactly generated. If $\Ee\simeq \Mod_k(\Sp)$ is the $\infty$-category of modules over some $\IE_\infty$-ring spectrum~$k$, we'll usually abbreviate these as $\Pr_k^\L$ and $\Pr_{k,\kappa}^\L$, respectively.
	
	\begin{numpar}[Localising motives over~$\Ee$.]
		We define the \emph{$\infty$-category of dualisable $\Ee$-modules} as the module $\infty$-category $\Cat_\Ee^\mathrm{dual}\coloneqq \Mod_\Ee(\Cat_\mathrm{st}^\mathrm{dual})$.%
		\footnote{$\Cat_\Ee^\mathrm{dual}$ can be defined without assuming that $\Ee$ is rigid, but usually it won't agree with $\Mod_\Ee(\Cat_\mathrm{st}^\mathrm{dual})$. See \cite[\S{\chref[subsection]{1.3}}]{EfimovLimits}.}
		Following Efimov \cite[Definition~\chref{1.20}]{EfimovLimits}, we let the \emph{$\infty$-category $\Mot_\Ee^\loc$ of localising motives over~$\Ee$} be the recipient of the universal localising invariant on dualisable $\Ee$-modules.
		
		In the case where $\Ee\simeq \Mod_k(\Sp)$ is the $\infty$-category of modules over some $\IE_\infty$-ring spectrum~$k$, we'll write $\Mot_k^\loc$ instead; this agrees with the $\infty$-category of localising motives over~$k$ defined by Blumberg--Gepner--Tabuada \cite{BlumbergGepnerTabuada}.
	\end{numpar}
	
	\begin{lem}\label{lem:TrefBabyCase}
		Let $M_{(-)}\colon \IQ\rightarrow \Mot_\Ee^\loc$ be a diagram such that $M_i\rightarrow M_j$ is trace-class for all rational numbers $i<j$. Then
		\begin{equation*}
			T^\mathrm{ref}\Bigl(\colimit_{i\in\IQ}M_i\Bigr)\simeq \indcolim_{i\in\IQ}T(M_i)\,.
		\end{equation*}
		If $\Dd$ is locally rigid and its tensor unit is $\omega_1$-compact, then the same is true for $\IZ_{\geqslant 0}$-indexed diagrams with trace-class transition maps.
	\end{lem}
	\begin{proof}
		This is almost tautological: Since $\Mot_\Ee^\loc$ is rigid, $(\Mot_\Ee^\loc)^\mathrm{rig}\rightarrow \Mot_\Ee^\loc$ is an equivalence. Since the ind-object $\indcolim_{i\in\IQ}M_i$ is a preimage of $M$ under this equivalence, the first claim follows. The second claim is completely analogous, since the additional assumptions imply $\Dd^\mathrm{rig}\simeq \Nuc\Ind(\Dd)$, as we've seen in \cref{par:RefinedInvariants}.
	\end{proof}
	\begin{numpar}[Why computing $T^\mathrm{ref}$ is hard.]\label{par:ComputingTrefIsHard}
		In general, we're faced with at least two difficult problems:
		\begin{alphanumerate}\itshape
			\item[\,!\,] For an arbitrary motive $M\in\Mot_\Ee^\loc$, it can be very hard to decompose $M$ into pieces for which resolutions as in \cref{lem:TrefBabyCase} exist.\label{enum:TrefHardA}
			\item[\,!!\,] Even if such resolutions can be found, computing $T(M_i)$ \embrace{and the transition maps between them} can still be a very hard problem.\label{enum:TrefHardB}
		\end{alphanumerate}
		In \cref{subsec:Recipe}, we'll explain how to solve problem~\cref{enum:TrefHardA} in many cases of interest, which will include $\THHref(\IQ)$, $\THHref(\L_n^f\IS_{(p)}/\IS_{(p)})$ and $\THHref(\IS[x])$. The entirety of \crefrange{sec:RefinedTC-}{sec:Overconvergent} below will then be spent on problem~\cref{enum:TrefHardB} for $\THHref(\IQ)$, and we will only be able to obtain an answer after base change to $\ku$.
	\end{numpar}
	
	But before we dive into the difficult calculations, let us discuss another easy case. To this end, recall from \cite[Definition~\chref{1.48}]{EfimovLimits} that a dualisable $\Ee$-module category $\Xx$ is called \emph{smooth} if the coevaluation $\Sp\rightarrow \Xx^\vee\otimes_\Ee\Xx$ preserves compact objects, and \emph{proper} if the evaluation $\Xx\otimes\Xx^\vee\rightarrow \Ee$ preserves compact objects. Here $\Xx^\vee$ denotes the dual of $\Xx$ as an $\Ee$-module.
	
	\begin{lem}\label{lem:SmoothProper}
		Let $\Xx$ be a dualisable $\Ee$-module.
		\begin{alphanumerate}
			\item $\Xx$ is smooth and proper in the sense above if and only if $\Xx$ is dualisable in $\Cat_\Ee^\mathrm{dual}$.\label{enum:SmoothProperDualisable}%
			\footnote{Note that being dualisable in $\Cat_\Ee^\mathrm{dual}$ is much stronger than being a dualisable $\Ee$-module.}
			\item If this is the case, then $T^\mathrm{ref}(\Xx)\simeq T(\Xx)$.\label{enum:SmoothProperRefined}
		\end{alphanumerate}
	\end{lem}
	\begin{proof}[Proof sketch]
		Assume first that $\Xx$ is smooth and proper. We'll only explain why the coevaluation and the evaluation over $\Ee$, i.e.\ $\Ee\rightarrow \Xx^\vee\otimes_\Ee\Xx$ and $\Xx\otimes_\Ee\Xx^\vee\rightarrow \Ee$, are functors in $\Cat_\Ee^\mathrm{dual}$; the triangle identities are then straightforward to verify. Since $\Sp\rightarrow \Xx^\vee\otimes_\Ee \Xx$ is strongly continuous by smoothness, the same will be true for the composition
		\begin{equation*}
			\Ee\longrightarrow \Ee\otimes(\Xx^\vee\otimes_\Ee\Xx)\longrightarrow \Xx^\vee\otimes_\Ee\Xx
		\end{equation*}	
		by \cite[Proposition~\chref{1.12}(ii)]{EfimovLimits}. So the coevaluation is a functor in $\Cat_\Ee^\mathrm{dual}$. Moreover, we have $\Xx^\vee\simeq \Hhom_\Ee^\mathrm{dual}(\Xx,\Ee)$ by \cite[Proposition~\chref{3.4}(iii)]{EfimovLimits}. Since $\Ee$ was assumed symmetric monoidal, $\Cat_\Ee^\mathrm{dual}$ admits an internal $\Hom$, which necessarily lifts $\Hhom_\Ee^\mathrm{dual}$. Hence we get an evaluation $\Xx\otimes_\Ee\Xx^\vee\rightarrow \Ee$ in $\Ee$ as well.
		
		Now assume that $\Xx$ is dualisable in $\Cat_\Ee^\mathrm{dual}$. Then $\Ee\rightarrow \Xx^\vee\otimes_\Ee\Xx$ is strongly continuous, hence it sends the tensor unit (which is compact as $\Ee$ is rigid) to a compact object. Then the same must be true for $\Sp\rightarrow \Xx^\vee\otimes_\Ee\Xx$, proving smoothness. For properness, we already know that $\Xx\otimes_\Ee \Xx^\vee\rightarrow \Ee$ is strongly continuous, so it remains to show the same for $\Xx\otimes\Xx^\vee\rightarrow \Xx\otimes_\Ee\Xx^\vee$. To this end, write
		\begin{equation*}
			\Xx\otimes_\Ee\Xx^\vee \simeq (\Xx\otimes \Xx^\vee)\otimes_{\Ee\otimes\Ee}\Ee
		\end{equation*}
		and use that $\Ee\otimes\Ee\rightarrow \Ee$ is strongly continuous by rigidity. This finishes the proof of \cref{enum:SmoothProperDualisable}. Part~\cref{enum:SmoothProperRefined} is an immediate consequence of this and \cref{lem:TrefBabyCase}, applied to the constant $\Xx$-valued diagram, which has trace-class transition maps since the identity on any dualisable object is trace-class.
	\end{proof}
	\begin{cor}\label{cor:SelfDualPreservesTraceClass}
		Let $\Ee\rightarrow \Xx$ be a strongly continuous symmetric monoidal functor into another rigid symmetric monoidal presentable stable $\infty$-category. If $\Xx$ is smooth and proper as an $\Ee$-module, then the forgetful functor $\Cat_\Xx^\mathrm{dual}\rightarrow \Cat_\Ee^\mathrm{dual}$ preserves trace-class morphisms.
	\end{cor}
	\begin{proof}
		By \cref{lem:SmoothProper}\cref{enum:SmoothProperDualisable} and the general fact that $\Xx^\vee\simeq \Xx$ (see \cite[\chref{1.9.2.1}]{GaitsgoryRozenblyumDerivedI} or \cite[Proposition~\chref{1.3}]{EfimovLimits}), we see that $\Xx$ is a self-dual $\IE_\infty$-algebra in $\Cat_\Ee^\mathrm{dual}$. The assertion then becomes purely abstract nonsense: For $\Xx$-modules $\Mm$ and $\Nn$, the diagram
		\begin{equation*}
			\begin{tikzcd}
				\Hhom_\Xx^\mathrm{dual}(\Mm,\Xx)\otimes_\Xx\Nn\dar\rar & \Hhom_\Xx^\mathrm{dual}(\Mm,\Xx)\otimes_\Xx(\Xx\otimes_\Ee\Nn)\dar\rar["\simeq"]& \Hhom_\Ee^\mathrm{dual}(\Mm,\Ee)\otimes_\Ee\Nn\dar\\
				\Hhom_\Xx^\mathrm{dual}(\Mm,\Nn)\rar & \Hhom_\Xx^\mathrm{dual}(\Mm,\Xx\otimes_\Ee\Nn)\rar["\simeq"] & \Hhom_\Ee^\mathrm{dual}(\Mm,\Nn)
			\end{tikzcd}
		\end{equation*}
		commutes, where the horizontal arrows in the left square are given by the unit $\Nn\rightarrow \Xx\otimes_\Ee\Nn$ of the \enquote{wrong way} adjunction  between the forgetful functor and $\Xx\otimes_\Ee-\colon \Cat_\Ee^\mathrm{dual}\rightarrow \Cat_\Xx^\mathrm{dual}$.
	\end{proof}
	
	\subsection{A recipe for computation}\label{subsec:Recipe}
	We continue to fix the notation from \cref{subsec:SmoothProper} as well as a symmetric monoidal localising invariant
	\begin{equation*}
		T\colon \Mot_\Ee^\loc\longrightarrow \Dd\,.
	\end{equation*}
	From now on, we'll additionally assume that $\Dd$ is locally rigid and its tensor unit is $\omega_1$-compact, so that $\Dd^\mathrm{rig}\simeq \Nuc\Ind(\Dd)$ by \cite[Theorem~\chref{4.2}]{EfimovLimits}.
	
	Our goal in this subsection is to explain a method to compute certain values of the refinement $T^\mathrm{ref}$. This method is a more or less straightforward abstract reformulation of the method that Efimov uses in his computations  (see e.g.\ \cite[Talk~{\chref[section]{6}}]{EfimovCopenhagen}).
	
	\begin{numpar}[Motives of interest.]\label{par:MotivesOfInterest}
		Let $\Ee\rightarrow \Xx$ be a strongly continuous symmetric monoidal functor into another rigid symmetric monoidal presentable stable $\infty$-category. Assume that $\Xx$ is smooth and proper as an $\Ee$-module. We wish to compute $T^\mathrm{ref}(\Uu)$ for localisations $\Uu\subseteq\Xx$ that arise as in \cref{par:KillingAlgebras}. That is, there is some object $V_0\in\Xx$ with a left-unital multiplication such that $\Uu$ is the full sub-$\infty$-category spanned by those $U\in\Xx$ for which $\Hhom_\Xx(V_0,U)\simeq 0$. Let us additionally assume that the following is satisfied:
		\begin{alphanumerate}\itshape
			\item[V] There exists a tower of $\IE_1$-algebras in $\Xx$,\label{enum:V}
			\begin{equation*}
				V_0\longleftarrow V_1\longleftarrow V_2\longleftarrow \dotsb\,,
			\end{equation*}
			such that each $V_r$ is dualisable in $\Xx$ and contained in the thick tensor ideal \embrace{that is, the smallest full sub-$\infty$-category closed under finite limits and colimits, retracts, and $-\otimes X$ for all $X\in\Xx$} generated by $V_0$. Moreover, we assume that for all $r\geqslant 0$, the induced map $V_{r+1}\otimes V_r\rightarrow V_r\otimes V_r$ factors through the multiplication
			\begin{equation*}
				V_{r+1}\otimes V_r\overset{\mu}{\longrightarrow} V_r
			\end{equation*}
			as a map of $V_{r+1}$-$V_r$-bimodules.
		\end{alphanumerate}
		The main example to keep in mind is the following: Suppose we're given maps $v_i\colon \Ii_i\rightarrow \IUnit_\Xx$ for $i=0,1,\dotsc,n$, where each $\Ii_i$ is dualisable in $\Xx$. Then we can define $V_r$ as the iterated cofibre
		\begin{equation*}
			V_r\coloneqq \IUnit_\Xx/\bigl(v_0^{\alpha_{r,0}},\dotsc,v_n^{\alpha_{r,n}}\bigr)
		\end{equation*}
		for some entry-wise increasing sequence of $(n+1)$-tuples $\alpha_r=(\alpha_{r,1},\dotsc,\alpha_{r,n})$ and equip the tower $\{V_r\}_{r\geqslant 0}$ with Burklund-style $\IE_1$-structures. We'll discuss in \cref{subsec:Burklund} why this satisfies~\cref{enum:V} and how this allows us to recover many examples of interest, such as $\THHref(\IQ)$, $\THHref(\IS[x])$, and $\THHref(\L_n^f\IS_{(p)}/\IS_{(p)})$ (note that the last example doesn't quite fit this situation, which will cause us some pain).
	\end{numpar}
	
	\begin{thm}\label{thm:RefinedInvariants}
		Let $\Ee$ be rigid and let $T\colon \Mot_\Ee^\loc\rightarrow \Dd$ be a localising invariant such that $\Dd$ is locally rigid and its tensor unit is $\omega_1$-compact. Let $\Xx$ and $\Uu$ be as in \cref{par:MotivesOfInterest}.
		\begin{alphanumerate}
			\item The pro-object $\prolim_{r\geqslant 0}T(\RMod_{V_r}(\Xx))$ is idempotent over $T(\Xx)$ and its transition maps are trace-class.\label{enum:RefinedInvariantsProIdempotent}
			\item $T^\mathrm{ref}(\Uu)$ is obtained from $T(\Xx)$ by killing this idempotent pro-algebra. In particular, $T^\mathrm{ref}(\Uu)$ sits inside the following cofibre sequence in $\Dd^\mathrm{rig}\simeq\Nuc\Ind(\Dd)$:\label{enum:RefinedInvariantsCofibreSequence}
			\begin{equation*}
				\indcolim_{r\geqslant 0}T\bigl(\RMod_{V_r}(\Xx)\bigr)^\vee\longrightarrow T(\Xx)\longrightarrow T^\mathrm{ref}(\Uu)\,.
			\end{equation*}
		\end{alphanumerate}
	\end{thm}
	
	We start the proof of \cref{thm:RefinedInvariants} with a few easy observations about the \enquote{closed complement} of $\Uu$ in $\Xx$.
	
	\begin{lem}\label{lem:UVandX}
		Let $\Xx$ and $\Uu$ be as in \cref{par:MotivesOfInterest}.
		\begin{alphanumerate}
			\item The inclusion $\Uu\rightarrow \Xx$ admits a left adjoint $j^*\colon \Xx\rightarrow \Uu$, which can be canonically equipped with a symmetric monoidal structure.\label{enum:USymmetricMonoidal}
			\item If $\Vv\subseteq\Xx$ denotes the kernel of $j^*$, then $\Vv$ is a tensor ideal and closed under colimits, finite limits, and retracts in $\Xx$. If $S$ runs through a set of generators of $\Xx$, then $V_0\otimes S$ forms a set of generators of $\Vv$.\label{enum:VTensorIdeal}
			\item  For all $r\geqslant 0$, the $\IE_1$-algebra $V_r$ is a compact object of $\Xx$, and every left- or right-module over $V_r$ is contained in $\Vv$.\label{enum:VCompact}
		\end{alphanumerate}
	\end{lem}
	\begin{proof}
		Part~\cref{enum:USymmetricMonoidal} follows immediately from \cref{par:KillingAlgebras}. Since $j^*$ is symmetric monoidal and preserves all colimits, its kernel~$\Vv$ must be a tensor ideal and closed under colimits, finite limits, and retracts. Now let $V\in\Vv$ be an object such that
		\begin{equation*}
			0\simeq \Hom_\Xx(V_0\otimes S,V)\simeq \Hom_\Xx\bigl(S,\Hhom_\Xx(V_0,V)\bigr)
		\end{equation*}
		for all~$S$. Since $S$ runs through a set of generators of $\Xx$, this implies $\Hhom_\Xx(V_0,V)\simeq 0$. Hence also $V\in\Uu$ and so $V\simeq j^*(V)\simeq 0$. This finishes the proof of \cref{enum:VTensorIdeal}.
		
		To show \cref{enum:VCompact}, observe that any $X\in\Xx$ is dualisable if and only if it is compact (because in a rigid presentable symmetric monoidal $\infty$-category $\id_X\colon X\rightarrow X$ is trace-class if and only if it is compact; see \cite[Corollary~\chref{4.52}]{RamziLocallyRigid} or \cite[Proposition~\chref{1.7}]{EfimovLimits}). Hence $V_r$ is compact for all $r\geqslant 0$. To show that any left- or right-$V_r$-module is contained in $\Vv$, it suffices to show the same for induced modules (i.e. those of the form $V_r\otimes X$), since every module is a colimit of induced ones. By the thick tensor ideal condition in \cref{par:MotivesOfInterest}\cref{enum:V}, we can furthermore reduce to objects of the form $V_0\otimes X$. Now if $U\in\Uu$, then
		\begin{equation*}
			\Hom_\Xx(V_0\otimes X,U)\simeq \Hom_\Xx\bigl(X,\Hhom_\Xx(V_0,U)\bigr)\simeq 0\,,
		\end{equation*}
		proving $j^*(V_0\otimes X)\simeq 0$, as desired. 
	\end{proof}

	\begin{lem}\label{lem:TraceClassTransitionMaps}
		For every $r\geqslant 0$, the base change functor
		\begin{equation*}
			-\otimes_{V_{r+1}}V_r\colon \RMod_{V_{r+1}}(\Xx)\longrightarrow \RMod_{V_r}(\Xx)
		\end{equation*}
		is a trace-class morphism in $\Cat_{\Xx}^\mathrm{dual}$, hence also in $\Cat_\Ee^\mathrm{dual}$.
	\end{lem}
	\begin{proof}
		The additional assertion will follow immediately from \cref{cor:SelfDualPreservesTraceClass} once we've shown the rest. Writing $\RMod_{V_r}(\Xx)\simeq \RMod_{V_r}(\Ind(\Xx^\omega))\otimes_{\Ind(\Xx^\omega)}\Xx$, we may reduce to the case where $\Xx$ is compactly generated, as $-\otimes_{\Ind(\Xx^\omega)}\Xx$ preserves trace-class morphisms by \cref{lem:TraceClassAbstractNonsense}\cref{enum:DualTraceClass}. In the compactly generated case, we'll even show that $-\otimes_{V_{r+1}}V_r$ is trace-class in $\Pr_{\Xx,\omega}^\L$.
		
		Recall from \cite[Remark~\chref{4.8.4.8}]{HA} that $\RMod_{V_{r+1}}(\Xx)$ is dualisable in $\Pr_\Xx^\L$ with dual $\LMod_{V_{r+1}}(\Xx)$. Therefore, the base change functor is always trace-class in $\Pr_\Xx^\L$. The witnessing functor $\Xx\rightarrow \LMod_{V_{r+1}}(\Xx)\otimes_\Xx \RMod_{V_r}(\Xx)\simeq \LMod_{V_{r+1}\otimes V_r^\op}(\Xx)$ is the classifier of $V_r$ as a left module over $V_{r+1}\otimes V_r^\op$, or equivalently, a $V_{r+1}$-$V-r$-bimodule. If we work in $\Pr_{\Xx,\omega}^\L$ instead, then $\RMod_{V_{r+1}}(\Xx)$ will no longer be dualisable, but we can still form the predual
		\begin{equation*}
			\Hhom_{\Pr_{\Xx,\omega}^\L}\bigl(\RMod_{V_{r+1}}(\Xx),\Xx\bigr)\simeq \Ind\left(\Fun_{\Xx^\omega}(\RMod_{V_{r+1}}(\Xx)^\omega,\Xx^\omega)\right)\simeq \Ind\bigl(\LMod_{V_{r+1}}(\Xx^\omega)\bigl)\,,
		\end{equation*}
		where we've used \cite[Theorem~\chref{4.8.4.1}]{HA} and the fact that $V_{r+1}\in\Xx^\omega$ by \cref{lem:UVandX}\cref{enum:VCompact}. Using \cite[Theorem~\chref{4.8.4.6}]{HA}, we still have a functor
		\begin{equation*}
			\Xx\longrightarrow \Ind\bigl(\LMod_{V_{r+1}}(\Xx^\omega)\bigr)\otimes_\Xx\RMod_{V_r}(\Xx)\simeq \RMod_{V_r}\left(\Ind\bigl(\LMod_{V_{r+1}}(\Xx^\omega)\bigr)\right)
		\end{equation*}
		in $\Pr_\Xx^\L$ that classifies $V_r$ has a right $V_r$-module in $\Ind(\LMod_{V_{r+1}}(\Xx))$. For the desired trace-class property to hold, this functor needs to be contained in $\Pr_{\Xx,\omega}^\L$. That is, we need $V_r$ to be a compact object in $\RMod_{V_r}(\Ind(\LMod_{V_{r+1}}(\Xx^\omega)))$.
		
		To this end, recall our assumption~\cref{par:MotivesOfInterest}\cref{enum:V} that $V_{r+1}\otimes V_r\rightarrow V_r\otimes V_r$ factors through the multiplication $V_{r+1}\otimes V_r\rightarrow V_r$ as a map of $V_{r+1}$-$V_r$-bimodules. Consequently, $V_r$ is a retract of $V_r\otimes V_r$ in $\RMod_{V_r}(\Ind(\LMod_{V_{r+1}}(\Xx^\omega)))$. This is enough to show compactness. Indeed, the object $V_r\in\Ind(\LMod_{V_{r+1}}(\Xx^\omega))$ is compact%
		\footnote{By contrast, $V_r$ is usually \emph{not} compact in $\LMod_{V_{r+1}}(\Xx)$.}
		and so the induced right-$V_r$-module $V_r\otimes V_r$ must be compact.
	\end{proof}
	
	\begin{rem}\label{rem:Dbcoh}
		As a consequence of the proof of \cref{lem:TraceClassTransitionMaps} and \cref{lem:TraceClassAbstractNonsense}\cref{enum:DualTraceClass}, we see that the functors
		\begin{equation*}
			\Ind \LMod_{V_r}(\Xx^\omega)\otimes_{\Ind(\Xx^\omega)}\Xx\longrightarrow \Ind \LMod_{V_{r+1}}(\Xx^\omega)\otimes_{\Ind(\Xx^\omega)}\Xx\,.
		\end{equation*}
		induced by the forgetful functors $\LMod_{V_r}(\Xx^\omega)\rightarrow \LMod_{V_{r+1}}(\Xx^\omega)$ are also trace-class in $\Cat_\Xx^\mathrm{dual}$, hence in $\Cat_\Ee^\mathrm{dual}$ by \cref{cor:SelfDualPreservesTraceClass}.
		
		The reader familiar with some of Efimov's computations of refined invariants will have already seen $\Ind \LMod_{V_r}(\Xx^\omega)\otimes_{\Ind(\Xx^\omega)}\Xx$, albeit in disguise: For example, it is the abstract analogue of $\Dd_\mathrm{coh}^b(\IQ[x]/x^n)$ in Efimov's computation of $\HC^{-, \mathrm{ref}}(\IQ[x^{\pm1}]/\IQ[x])$ (see e.g.\ \cite[Talk~{\chref[section]{6}}]{EfimovCopenhagen}). Also note that the forgetful functors $\LMod_{V_r}(\Xx^\omega)\rightarrow\Xx^\omega$ will land in $\Vv$ by \cref{lem:UVandX}\cref{enum:VCompact} and so we get functors
		\begin{equation*}
			\Ind \LMod_{V_r}(\Xx^\omega)\otimes_{\Ind(\Xx^\omega)}\Xx\longrightarrow \Vv\,.
		\end{equation*}
		for all $r\geqslant 0$. These are compatible with the functors above.
	\end{rem}

	\begin{lem}\label{lem:TorsionSpectra}
		With notation as above, the functors from \cref{rem:Dbcoh} induce an equivalence of $\Xx$-linear presentable $\infty$-categories
		\begin{equation*}
			\colimit_{r\geqslant 0}\left(\Ind\LMod_{V_r}(\Xx^\omega)\otimes_{\Ind(\Xx^\omega)}\Xx\right)\overset{\simeq}{\longrightarrow} \Vv\,.
		\end{equation*}
		Here the colimit on the left-hand side is taken in $\Cat_\Xx^\mathrm{dual}$, or equivalently, in $\Cat_\Ee^\mathrm{dual}$ or $\Pr_\mathrm{st}^\L$.
	\end{lem}
	
	\begin{proof}
		We'll prove this under the assumption that $\Xx$ is compactly generated; to reduce to this special case, apply \cref{lem:ReductionToCompactlyGenerated} below for $\Ind(\Xx^\omega)\rightarrow \Xx$. Since $\Xx$ is rigid, compact objects are closed under tensor products, since they coincide with the dualisable objects. By \cref{lem:UVandX}\cref{enum:VTensorIdeal}, this implies that $\Vv$ is again compactly generated. By construction, $\Ind\LMod_{V_r}(\Xx^\omega)\rightarrow \Xx$ preserves compact objects, hence the same is true if we restrict the codomain to $\Vv$. Using that $\Pr_{\mathrm{st},\omega}^\L\rightarrow \Pr_\mathrm{st}^\L$ preserves all colimits, we deduce that
		\begin{equation*}
			L\colon \colimit_{r\geqslant 0}\Ind\LMod_{V_r}(\Xx^\omega)\longrightarrow \Vv
		\end{equation*}
		is a functor in $\Pr_{\mathrm{st},\omega}^\L$. In particular, whether $L$ is fully faithful can be checked on compact objects. So let $M$ and $N$ be compact.
		
		Writing $\colimit_{r\geqslant 0}\Ind(\LMod_{V_r}(\Xx^\omega))\simeq \Ind(\colimit_{r\geqslant 0}\LMod_{V_r}(\Xx^\omega))$, we may assume that $M$ and $N$ are $V_r$-modules for some $r$. We must then show that
		\begin{equation*}
			\colimit_{s\geqslant r}\Hom_{V_s}(M,N)\overset{\simeq}{\longrightarrow} \Hom_\Xx(M,N)\,.
		\end{equation*}
		is an equivalence. To this end, let us rewrite this map as
		\begin{equation*}
			\colimit_{s\geqslant r}\Hom_{V_r}\bigl((V_r\otimes_{V_s}V_r)\otimes_{V_r}M,N\bigr)\longrightarrow \Hom_{V_r}\bigl((V_r\otimes V_r)\otimes_{V_r}M,N\bigr)\,.
		\end{equation*}
		For all $s\geqslant r$, consider $V_r\otimes V_r$ as a right-$V_{s+1}$-module via the right action on the first tensor factor and as a left-$V_{s+1}$-module via the left action on the second tensor factor. In total, we've produced a right-$(V_{s+1}\otimes V_{s+1}^\op)$-module structure on $V_r\otimes V_r$. Since $V_r\otimes V_r$ is already a right-$(V_s\otimes V_s^\op)$-module via the same construction, the identity on $V_r\otimes V_r$ factors through $(V_r\otimes V_r)\otimes_{V_{s+1}\otimes V_{s+1}^\op}V_s\otimes V_s^\op$. By Assumption~\cref{par:MotivesOfInterest}\cref{enum:V}, $V_{s+1}\otimes V_{s+1}\rightarrow V_s\otimes V_s$ factors through $V_{s+1}$ as a map of $V_{s+1}$-$V_{s+1}$-bimodules, or equivalently, as a map of left-$V_{s+1}\otimes V_{s+1}^\op$-modules. This shows that the identity on $V_r\otimes V_r$ factors through
		\begin{equation*}
			(V_{r}\otimes V_r)\otimes_{V_{s+1}\otimes V_{s+1}^\op}V_{s+1}\simeq V_r\otimes_{V_{s+1}}V_r\,.
		\end{equation*}
		This factorisation works as $V_r$-$V_r$-bimodules, since we haven't touched the \enquote{outer} $V_r$-$V_r$-bimodule structure anywhere and have only worked with the \enquote{inner} bimodule structures. Thus, the colimit diagram above can be intertwined with the constant $\Hom_{V_r}((V_r\otimes V_r)\otimes_{V_r}M,N)$-valued diagram, which proves that we get the desired equivalence.
		
		Hence $L$ is fully faithful. Once we know this, essential surjectivity follows immediately from \cref{lem:UVandX}\cref{enum:VTensorIdeal}, so we win.
	\end{proof}
	
	\begin{lem}\label{lem:ReductionToCompactlyGenerated}
		Let $\Xx\rightarrow \Xx'$ be a symmetric monoidal colimit-preserving functor into another rigid presentable stable symmetric monoidal $\infty$-category $\Xx'$. Let $V_0'$ denote the image of $V_0$, let $\Uu'\coloneqq (\Xx')^{V_0'}\subseteq \Xx'$ and let $\Vv'$ be the kernel of the left adjoint $\Xx'\rightarrow \Uu'$ of the inclusion. Then the induced functor
		\begin{equation*}
			\Vv\otimes_{\Xx}\Xx'\overset{\simeq}{\longrightarrow} \Vv'
		\end{equation*}
		is an equivalence of $\infty$-categories.
	\end{lem}
	\begin{proof}
		It's enough to show this in the case where $\Xx$ is compactly generated, since the general case will follow by considering $\Ind(\Xx^\omega)\rightarrow \Xx\rightarrow \Xx'$. By \cref{lem:UVandX}\cref{enum:VTensorIdeal}, $\Vv$ is a tensor ideal and so the inclusion $\Vv\rightarrow \Xx$ is $\Xx$-linear. Note that its right adjoint is again $\Xx$-linear. Indeed, the right adjoint is given by $\fib(X\rightarrow j^*(X))$ for all $X\in\Xx$, so we must show that $j^*(X)\otimes Y\rightarrow j^*(X\otimes Y)$ is an equivalence for all $Y\in\Xx$. Since we assume $\Xx$ to be compactly generated, it suffices to show this in the case $Y\in \Xx^\omega$, as both sides commute with filtered colimits. But then $Y$ is dualisable as $\Xx$ is rigid. Since $\Uu'$ is stable under tensoring with dualisable objects, we obtain $j^*(X)\otimes Y\simeq j^*(j^*(X)\otimes Y)\simeq j^*(X\otimes Y)$ from \cref{par:KillingAlgebras}, as desired.
		
		It follows that $\Vv\otimes_{\Xx}\Xx\rightarrow \Xx'$ is fully faithful, since we can now just base change the fact that the unit is an equivalence. Its essential image is clearly contained in $\Vv'$, and it's clear from \cref{lem:UVandX}\cref{enum:VTensorIdeal} that $\Vv_\omega\otimes_{\Xx}\Xx'\rightarrow \Vv'$ is essentially surjective.
	\end{proof}
	
	
	\begin{proof}[Proof of \cref{thm:RefinedInvariants}]
		By \cref{lem:TraceClassTransitionMaps} and \cref{lem:TraceClassAbstractNonsense}\cref{enum:DualTraceClass} applied to the symmetric monoidal functor $T\colon \Cat_\Xx^\mathrm{dual}\simeq \Mod_\Xx(\Cat_\Ee^\mathrm{dual})\rightarrow \Mod_{T(\Xx)}(\Dd)$, the transition maps of the pro-object $\prolim_{r\geqslant 0} T(\RMod_{V_r}(\Xx))$ are trace-class morphisms in $\Mod_{T(\Xx)}(\Dd)$. To prove~\cref{enum:RefinedInvariantsProIdempotent}, it will thus be enough to check that the dual ind-object is an idempotent coalgebra.
		
		To see this, write $\RMod_{V_r}(\Xx)\simeq \RMod_{V_r}(\Ind(\Xx^\omega))\otimes_{\Ind(\Xx^\omega)}\Xx$. We've seen in the proof of \cref{lem:TraceClassTransitionMaps} that the predual of $\RMod_{V_r}(\Ind(\Xx^\omega))$ in $\Pr_{\Ind(\Xx^\omega),\omega}^\L$ is $\Ind\LMod_{V_r}(\Xx^\omega)$. Now consider the diagram of symmetric monoidal functors
		\begin{equation*}
			\begin{tikzcd}
				\Pr_{\Ind(\Xx^\omega),\omega}^\L\rar["-\otimes_{\Ind(\Xx^\omega)}\Xx"]\dar &[2.5em] \Cat_\Xx^\mathrm{dual}\rar & \Mod_\Xx(\Mot_\Ee^\loc)\rar["T"] & \Mod_{T(\Xx)}(\Dd)\\
				\Cat_{\Ind(\Xx^\omega)}^\mathrm{dual}\ar[rr] & & \Mot_{\Ind(\Xx^\omega)}^\loc\uar &
			\end{tikzcd}
		\end{equation*}
		In general, none of them preserves preduals, but once we pass to $\indcolim_{r\geqslant 0}$ this isn't a problem anymore by \cref{lem:TraceClassAbstractNonsense}\cref{enum:DualTraceClassLift}. Thus, it will be enough to check that the image of $\indcolim_{r\geqslant 0}\Ind\LMod_{V_r}(\Xx^\omega)$ is idempotent in $\Ind(\Mot_{\Ind(\Xx^\omega)}^\loc)$.
		
		For ease of notation, let us now replace $\Xx$ by $\Ind(\Xx^\omega)$, thereby assuming that $\Xx$ is compactly generated. Since $\indcolim_{r\geqslant 0}\Ind\LMod_{V_r}(\Xx^\omega)$ has trace-class transition maps and $\Nuc\Ind(\Mot_{\Xx}^\loc)\simeq \Mot_{\Xx}^\loc$ by Efimov's rigidity theorem, it will be enough to show that $\colimit_{r\geqslant 0}\Ind\LMod_{V_r}(\Xx^\omega)\simeq \Vv$ is idempotent in $\Mot_{\Xx}^\loc$. We claim that $\Vv$ is already idempotent in $\Cat_{\Xx}^\mathrm{dual}$. To see this, just observe that the same argument as in \cref{lem:TorsionSpectra} also proves that
		\begin{equation*}
			\colimit_{r\geqslant 0}\Ind\LMod_{V_r\otimes V_r}(\Xx^\omega)\overset{\simeq}{\longrightarrow} \Vv
		\end{equation*} 
		is an equivalence of $\infty$-categories. This finishes the proof of~\cref{enum:RefinedInvariantsProIdempotent}.
		
		Let us now show \cref{enum:RefinedInvariantsCofibreSequence}. In the following, we'll use several times (and in a somewhat confusing way) that $\Nuc\Ind(\Mod_\Xx(\Mot_\Ee^\loc))\simeq \Mod_\Xx(\Mot_\Ee^\loc)$ by Efimov's rigidity theorem.
		
		The proof of~\cref{enum:RefinedInvariantsProIdempotent} shows that $\prolim_{r\geqslant 0} \RMod_{V_r}(\Xx)$ is idempotent in $\Pro(\Mod_\Xx(\Mot_\Ee^\loc))$, its dual ind-object has nuclear transition map, and the dual ind-object is sent to $\Vv$ under $\Nuc\Ind(\Mod_\Xx(\Mot_\Ee^\loc))\simeq \Mod_\Xx(\Mot_\Ee^\loc)$. Since $\Vv\rightarrow \Xx\rightarrow \Uu$ becomes a cofibre sequence in $\Mot_\Xx(\Mot_\Ee^\loc)$, it follows that the preimage of $\Uu$ under $\Nuc\Ind(\Mod_\Xx(\Mot_\Ee^\loc))\simeq \Mod_\Xx(\Mot_\Ee^\loc)$ is obtained from $\Xx$ by killing the pro-idempotent $\prolim_{r\geqslant 0} \RMod_{V_r}(\Xx)$. This is necessarily also true as $\IE_\infty$-$\Xx$-algebras, since the $\IE_\infty$-structure will be idempotent over $\Xx$ by \cref{lem:NuclearIdempotentAbstract}\cref{enum:NuclearIdempotentAbstract} and thus unique. Since any symmetric monoidal functor preserves killing idempotent pro-algebras with trace-class transition maps by \cref{lem:NuclearIdempotentAbstract}\cref{enum:IdempotentBasechange}, the statement of~\cref{enum:RefinedInvariantsCofibreSequence} follows.
	\end{proof}
	
	\subsection{Burklund's \texorpdfstring{$\IE_1$}{E1}-structures and square-zero extensions}\label{subsec:Burklund}
	
	In this subsection we show that tensor products of two Burklund-style $\IE_1$-structures on quotients are often trivial square zero algebras. We then use this technical result to make \cref{thm:RefinedInvariants} applicable in many cases of interest.
	
	For the abstract setup, let $\Cc$ be a presentable stable $\IE_2$-monoidal $\infty$-category and $v\colon \Ii\rightarrow\IUnit$ be a morphism in $\Cc$ such that $\IUnit/v$ admits a right-unital multiplication. Fix $\alpha_0\geqslant 3$, so that $\IUnit/v^{\alpha_0}$ admits a preferred $\IE_2$-algebra structure by \cite[Theorem~\chref{1.5}]{BurklundMooreSpectra}. The same theorem shows that $\IUnit/v^{\alpha}$ admits a preferred $\IE_1$-algebra structure for all $\alpha\geqslant 2$. Via base change, we get an $\IE_1$-structure on $\IUnit/v^{\alpha_0}\otimes\IUnit/v^{\alpha}$  in the $\IE_1$-monoidal stable $\infty$-category $\LMod_{\IUnit/v^{\alpha_0}}(\Cc)$. 
	\begin{prop}\label{prop:SquareZeroExtensions}
		With notation and assumptions as above, suppose additionally that $\Cc$ is rigid, $\Ii$ is dualisable in $\Cc$, and $\alpha\geqslant \alpha_0+3$.
		\begin{alphanumerate}
			\item
			If we equip $\IUnit/v^{\alpha_0}\oplus \Sigma(\Ii^{\otimes \alpha}/v^{\alpha_0})$ with the trivial square-zero $\IE_1$-structure over $\IUnit/v^{\alpha_0}$, then the equivalence of left $\IUnit/v^{\alpha_0}$-modules
			\begin{equation*}
				\IUnit/v^{\alpha_0}\otimes\IUnit/v^{\alpha}\simeq \IUnit/v^{\alpha_0}\oplus \Sigma(\Ii^{\otimes \alpha}/v^{\alpha_0})
			\end{equation*}
			lifts canonically to an equivalence of $\IE_1$-algebras in $\LMod_{\IUnit/v^{\alpha_0}}(\Cc)$. Under this identification, the multiplication $\IUnit/v^{\alpha_0}\otimes\IUnit/v^{\alpha}\rightarrow \IUnit/v^{\alpha_0}$ becomes the augmentation map $\IUnit/v^{\alpha_0}\oplus \Sigma(\Ii^{\otimes \alpha}/v^{\alpha_0})\rightarrow \IUnit/v^{\alpha_0}$.\label{enum:SquareZero}
			\item For all $\alpha'\geqslant \alpha\geqslant \alpha_0+3$, the map $\IUnit/v^{\alpha_0}\otimes\IUnit/v^{\alpha'}\rightarrow \IUnit/v^{\alpha_0}\otimes\IUnit/v^{\alpha}$ agrees with the map of trivial square-zero extensions induced by $v^{\alpha'-\alpha}\colon \Ii^{\otimes \alpha'}/v^{\alpha_0}\rightarrow \Ii^{\otimes \alpha}/v^{\alpha_0}$, as maps of $\IE_1$-algebras in $\LMod_{\IUnit/v^{\alpha_0}}(\Cc)$.\label{enum:SquareZeroMap}
		\end{alphanumerate}
	\end{prop}
	\begin{rem}
		The bound $\alpha\geqslant \alpha_0+3$ doesn't seem optimal and the author suspects that \cref{prop:SquareZeroExtensions} might already be true for $\alpha\geqslant \alpha_0$. It also seems reasonable that the result should be true for any compatible $\IE_1$-structures on $\IUnit/v^{\alpha_0}$ and $\IUnit/v^{\alpha}$, but we don't know how to show this.
	\end{rem}
	\begin{rem}
		Since the bounds $\alpha_0\geqslant 3$ and $\alpha\geqslant \alpha_0+3$ ensure that the $\IE_1$-algebra structures on $\IUnit/v^{\alpha_0}$ and $\IUnit/v^{\alpha}$ refine to $\IE_2$-algebra structures, the multplication map in \cref{prop:SquareZeroExtensions}\cref{enum:SquareZero} is canonically a map of $\IE_1$-algebras. The identification with the augmentation $\IUnit/v^{\alpha_0}\oplus \Sigma(\Ii^{\otimes \alpha}/v^{\alpha_0})\rightarrow \IUnit/v^{\alpha_0}$ also holds as $\IE_1$-algebra maps (as we'll see in the proof).
	\end{rem}
	\begin{proof}[Proof of \cref{prop:SquareZeroExtensions}]
		Recall \cite[Constructions~\chref{4.7} and~\chref{4.8}]{BurklundMooreSpectra}: Let $\widetilde{\Cc}\coloneqq \Def(\Cc,\Qq)$ be the deformation of $\Cc$ that Burklund uses. The specific construction is irrelevant for the purpose of this proof; the reader only needs to know that $\widetilde{\Cc}$ is a presentable stable $\IE_2$-monoidal $\infty$-category and comes with $\IE_2$-monoidal functors $\nu\colon \Cc\rightarrow \widetilde{\Cc}$ (which is non-exact) and $(-)^{\tau=1}\colon \widetilde{\Cc}\rightarrow \Cc$ (which preserves colimits and is therefore exact) such that $\nu(-)^{\tau=1}\simeq \id_\Cc$. Let furthermore $\widetilde{\IUnit}\coloneqq \nu(\IUnit)$ denote the tensor unit of $\widetilde{\Cc}$ and let $\widetilde{\Ii}\coloneqq \Sigma^{-1}\nu(\Sigma \Ii)$. Even though $\nu$ is non-exact, $\nu(\IUnit)\rightarrow \nu(\IUnit/v)\rightarrow \nu(\Sigma\Ii)$ is still a cofibre sequence in $\widetilde{\Cc}$ and so $\nu(v)\colon \nu(\Ii)\rightarrow \widetilde{\IUnit}$ factors through a map
		\begin{equation*}
			\widetilde{v}\colon \widetilde{\Ii}\longrightarrow \widetilde{\IUnit}\,.
		\end{equation*}
		Then $\widetilde{v}$ is a deformation of $v$ in the sense that $\widetilde{v}^{\tau=1}\simeq v$.%
		\footnote{Note that $\widetilde{v}$ is usually \emph{not} the trivial deformation $\nu(v)$, as the canonical map $\nu(\Ii)\rightarrow \widetilde{\Ii}$ is usually not an equivalence. This is crucial to make Burklund's construction work.}
		It will thus be enough to show the assertions with $v$ replaced by $\widetilde{v}\colon \widetilde{\Ii}\rightarrow \widetilde{\IUnit}$.
		
		Burklund constructs $\IE_1$-structures on $\widetilde{\IUnit}/\widetilde{v}^{\alpha}$ for $\alpha\geqslant 2$ using the obstruction theory from \cite[Proposition~\chref{2.4}]{BurklundMooreSpectra} in $\widetilde{\Cc}$. The reason to replace $\Cc$ and $v$ by their deformations $\widetilde{\Cc}$ and $\widetilde{v}$ is that for the deformed versions all obstructions vanish (because the obstruction group vanishes), and the witnessing nullhomotopies are unique (because the next homotopy group also vanishes).
		
		The base-changed $\IE_1$-structure on $\widetilde{\IUnit}/\widetilde{v}^{\alpha_0}\otimes \widetilde{\IUnit}/\widetilde{v}^\alpha$ is then obtained via Burklund's obstruction theory in the $\IE_1$-monoidal%
		\footnote{Burklund's paper assumes an $\IE_2$-monoidal structure, but for the purpose of \cite[\S{\chref[section]{2}}]{BurklundMooreSpectra} only an $\IE_1$-monoidal structure is necessary.}
		presentable stable $\infty$-category $\LMod_{\smash{\widetilde{\IUnit}}/\smash{\widetilde{v}^{\alpha_0}}}(\widetilde{\Cc})$. The main step to prove both~\cref{enum:SquareZero} and~\cref{enum:SquareZeroMap} is to show that in this case too all obstructions vanish and the witnessing nullhomotopies are unique. More precisely, we'll show that for all $k\geqslant 2$ and all $\alpha'\geqslant \alpha\geqslant \alpha_0+3$,
		\begin{equation*}
			\pi_i\Hom_{\LMod_{\widetilde{\IUnit}/\widetilde{v}^{\alpha_0}}(\widetilde{\Cc})}\left(\Sigma^{-3}\bigl(\Sigma^2(\widetilde{\Ii}/\widetilde{v}^{\alpha_0})^{\otimes \alpha'}\bigr)^{\otimes k},\widetilde{\IUnit}/\widetilde{v}^{\alpha_0}\otimes \widetilde{\IUnit}/\widetilde{v}^{\alpha}\right)\cong 0\quad \text{ for }i\in\{0,1\}\,.
		\end{equation*}
		To show this, we use that $\widetilde{\IUnit}/\widetilde{v}^{\alpha_0}\otimes -\colon \widetilde{\Cc}\rightarrow \LMod_{\smash{\widetilde{\IUnit}}/\smash{\widetilde{v}^{\alpha_0}}}(\widetilde{\Cc})$ is left adjoint to the forgetful functor, that $\widetilde{\IUnit}/\widetilde{v}^{\alpha_0}\otimes \widetilde{\IUnit}/\widetilde{v}^{\alpha}\simeq \widetilde{\IUnit}/\widetilde{v}^{\alpha_0}\oplus \Sigma(\widetilde{\Ii}^{\otimes \alpha}/\widetilde{v}^{\alpha_0})$ as left-$\widetilde{\IUnit}/\widetilde{v}^{\alpha_0}$-modules, and that $\widetilde{\Ii}$ is still dualisable, with dual $\widetilde{\Ii}^\vee\simeq\Sigma \nu(\Sigma^{-1}\Ii^\vee)$. The left-hand side above can then be rewritten as follows:
		\begin{multline*}
			\pi_i\Hom_{\widetilde{\Cc}}\left(\Sigma^{2k-3}\widetilde{\Ii}^{\otimes \alpha' k},\widetilde{\IUnit}/\widetilde{v}^{\alpha_0}\oplus \Sigma(\widetilde{\Ii}^{\otimes \alpha}/\widetilde{v}^{\alpha_0})\right)\\
			\begin{aligned}
				&\cong \pi_i\Hom_{\widetilde{\Cc}}\left(\Sigma^{2k-3}\widetilde{\Ii}^{\otimes \alpha' k},\widetilde{\IUnit}/\widetilde{v}^{\alpha_0}\right)\oplus \pi_i\Hom_{\widetilde{\Cc}}\left(\Sigma^{2k-2}\widetilde{\Ii}^{\otimes \alpha' k}\otimes (\widetilde{\Ii}^\vee)^{\otimes \alpha'},\widetilde{\IUnit}/\widetilde{v}^{\alpha_0}\right)\\
				&\cong \pi_i\Hom_{\widetilde{\Cc}}\left(\Sigma^{-\alpha' k +2k-3}\nu(X),\widetilde{\IUnit}/\widetilde{v}^{\alpha_0}\right)\oplus \pi_i\Hom_{\widetilde{\Cc}}\left(\Sigma^{-\alpha' k+\alpha+2k-2}\nu(Y),\widetilde{\IUnit}/\widetilde{v}^{\alpha_0}\right)\,,
			\end{aligned}
		\end{multline*}
		where $X\simeq (\Sigma\Ii)^{\otimes \alpha' k}$ and $Y\simeq (\Sigma \Ii)^{\otimes \alpha'k}\otimes (\Sigma^{-1}\Ii^\vee)^{\otimes \alpha}$. According to \cite[Lemma~\chref{4.8}]{BurklundMooreSpectra} (which is applicable thanks to our rigidity assumption on $\Cc$), both summands on the right-hand side vanish for $i\in\{0,1\}$ as soon as $\alpha' k-\alpha-2k+1\geqslant \alpha_0$. Under our assumptions $\alpha'\geqslant \alpha\geqslant \alpha_0+3$ and $k\geqslant 2$, we can estimate
		\begin{equation*}
			\alpha'k-\alpha-2k+1\geqslant (\alpha_0+3)(k-1)-2k+1=(k-1) \alpha_0+k-2\geqslant\alpha_0\,,
		\end{equation*}
		as desired. This shows that indeed all obstructions vanish (because the obstruction group $\pi_0$ vanishes) and the witnessing nullhomotopies are unique (because $\pi_1$ also vanishes).
		
		Now~\cref{enum:SquareZeroMap} as well as the first part of~\cref{enum:SquareZero} immediately follow. Indeed, in the case $\alpha'=\alpha$,  the vanishing result above combined with \cite[Remark~\chref{2.5}]{BurklundMooreSpectra} shows that the $\IE_1$-structure on $\widetilde{\IUnit}/\widetilde{v}^{\alpha_0}\otimes \widetilde{\IUnit}/\widetilde{v}^\alpha$ is unique, so it has to be the trivial square zero structure. For general $\alpha'\geqslant \alpha$, the same argument shows that the $\IE_1$-map $\widetilde{\IUnit}/\widetilde{v}^{\alpha_0}\otimes \widetilde{\IUnit}/\widetilde{v}^{\alpha'}\rightarrow \widetilde{\IUnit}/\widetilde{v}^{\alpha_0}\otimes \widetilde{\IUnit}/\widetilde{v}^{\alpha}$ is unique, proving~\cref{enum:SquareZeroMap}. To show the second part of~\cref{enum:SquareZero}, observe that, with notation as above, we must also have
		\begin{equation*}
			\pi_i\Hom_{\widetilde{\Cc}}\left(\Sigma^{-\alpha' k +2k-3}\nu(X),\widetilde{\IUnit}/\widetilde{v}^{\alpha_0}\right)\cong 0\quad\text{for }i\in\{0,1\}\,.
		\end{equation*}
		This precisely ensures that $\widetilde{\IUnit}/\widetilde{v}^{\alpha_0}\otimes \widetilde{\IUnit}/\widetilde{v}^\alpha\rightarrow \widetilde{\IUnit}/\widetilde{v}^{\alpha_0}$ is unique as well, and so it has to be the augmentation map.
	\end{proof}
	
	\begin{cor}\label{cor:E1FactorisationSpalpha}
		If $\Ii$ is dualisable, $\alpha\geqslant \alpha_0+3$, and $\alpha'\geqslant \alpha+\alpha_0$, then
		\begin{equation*}
			\IUnit/v^{\alpha_0}\otimes\IUnit/v^{\alpha'}\longrightarrow \IUnit/v^{\alpha_0}\otimes\IUnit/v^\alpha
		\end{equation*}
		factors through the tensor unit $\IUnit/v^{\alpha_0}$ as a map of $\IE_1$-algebras in $\LMod_{\IUnit/v^{\alpha_0}}(\Cc)$.
	\end{cor}
	\begin{proof}
		By \cref{prop:SquareZeroExtensions}\cref{enum:SquareZeroMap}, it's enough to check that $v^{\alpha'-\alpha}\colon \Ii^{\otimes \alpha'}/v^{\alpha_0}\rightarrow \Ii^{\otimes \alpha}/v^{\alpha_0}$ is zero in $\LMod_{\IUnit/v^{\alpha_0}}(\Cc)$ for $\alpha'\geqslant \alpha+\alpha_0$. This reduces to $v^{\alpha_0}\colon \Ii^{\otimes \alpha_0}/v^{\alpha_0}\rightarrow \IUnit/v^{\alpha_0}$ being zero in $\LMod_{\IUnit/v^{\alpha_0}}(\Cc)$. Since $\IUnit/v^{\alpha_0}\otimes-\colon \Cc\rightarrow \LMod_{\IUnit/v^{\alpha_0}}(\Cc)$ is left adjoint to the forgetful functor, this is equivalent to $v^{\alpha_0}\colon \Ii^{\otimes {\alpha_0}}\rightarrow \IUnit/v^{\alpha_0}$ being zero in $\Cc$, which is true by construction.
	\end{proof}
	
	Thanks to \cref{cor:E1FactorisationSpalpha}, it is now easy to construct examples where Assumption~\cref{par:MotivesOfInterest}\cref{enum:V} is satisfied and thus \cref{thm:RefinedInvariants} is applicable.
	
	\begin{exm}\label{exm:THHrefk1m}
		Let $m$ be a positive integer that is either coprime to $2$ or divisible by~$4$. Then $\IS/m$ admits a right-unital multiplication and so Burklund's construction applied to $m\colon \IS\rightarrow \IS$ provides a tower of $\IE_1$-algebras
		\begin{equation*}
			\IS/m^2\longleftarrow \IS/m^3\longleftarrow \IS/m^4\longleftarrow \dotsb\,.
		\end{equation*}
		Up to passing to an appropriate subtower, this satisfies Assumption~\cref{par:MotivesOfInterest}\cref{enum:V}. Indeed, dualisability and the thick tensor ideal condition are clear and the factorisation condition follows from \cref{cor:E1FactorisationSpalpha} above. 
		
		Thus, for any $\IE_\infty$-ring spectrum~$k$, \cref{thm:RefinedInvariants} shows that $\THHref(k[1/m]/k)$ is obtained from $\THH(k/k)\simeq k$ by killing the idempotent pro-algebra $\prolim_{\alpha\geqslant 1}\THH((k\otimes\IS/m^\alpha)/k)$. In particular, there's a cofibre sequence
		\begin{equation*}
			\indcolim_{\alpha\geqslant 1}\THH\bigl((k\otimes\IS/m^\alpha)/k\bigr)^\vee\longrightarrow k\longrightarrow \THHref\bigl(k\bigl[\localise{m}\bigr]/k\bigr)\,.
		\end{equation*}
		in $\Nuc\Ind(\Mod_k(\Sp)^{\B S^1})$. Since $\THHref(-/k)$ commutes with filtered colimits, this also allows us to compute $\THHref(k\otimes\IQ/k)\simeq \colimit_{m\in\IN}\THHref(k[1/m]/k)$.
	\end{exm}
	\begin{exm}\label{exm:THHrefSx}
		If $k$ is any $\IE_\infty$-ring spectrum, we can compute $\THHref(k[x]/k)$ as follows: Let $\IP_k^1$ denote the flat projective line over~$k$, which is smooth and proper over $k$. We can construct a tower of $\IE_1$-algebras
		\begin{equation*}
			k[x^{-1}]/x^{-1}\longleftarrow k[x^{-1}]/x^{-2}\longleftarrow k[x^{-1}]/x^{-3}\longleftarrow \dotsb
		\end{equation*}
		either by hand (construct $k[x^{-1}]$ as a graded $\IE_\infty$-$k$-algebra with $x^{-1}$ in graded degree~$-1$, then truncate the grading) or by applying Burklund's construction to $\Oo_{\IP_k^1}(-1)\rightarrow \Oo_{\IP_k^1}$ (this will only give the tower from the second step onwards, but this is no problem). In either case, Assumption~\cref{par:MotivesOfInterest}\cref{enum:V} will be satisfied and so \cref{thm:RefinedInvariants} provides a cofibre sequence
		\begin{equation*}
			\indcolim_{\alpha\geqslant1}\THH\bigl((k[x^{-1}]/x^{-\alpha})/k\bigr)^\vee\longrightarrow \THH(\IP_k^1/k)\longrightarrow \THHref\bigl(k[x]/k\bigr)
		\end{equation*}
		in $\Nuc\Ind(\Mod_k(\Sp)^{\B S^1})$.
	\end{exm}
	
	As a final example, let us explain how \cref{thm:RefinedInvariants} applies to $\THHref(\L_n^{f}\IS_{(p)}/\IS_{(p)})$, where $\L_n^f$ denotes telescopic localisation to chromatic height~$\leqslant n$. First we need a technical lemma:
	
	\begin{lem}\label{lem:EmFinityTypeTechnical}
		Let $m\geqslant 2$ and $n\geqslant 0$. Let $V'\rightarrow V$ be a map of $\IE_{m+1}$-algebras whose underlying spectra are of type~$n$. Let $v\colon \Sigma^NV\rightarrow V$ be a $v_n$-self map of $V$ and $v'\colon \Sigma^{N'}V'\rightarrow V'$ a $v_n$-self map of $V'$.
		\begin{alphanumerate}
			\item Up to replacing $v'$ by a suitable power, the induced map $v'\otimes_{V'}V\colon \Sigma^{N'}V\rightarrow V$ can be chosen to be a power of $v$.\label{enum:AsymptoticUniqueness}
			\item Suppose $v$ is the fourth power of another $v_n$-self map of $V$, so that $V/v$ admits a right-unital multiplication in $\LMod_V(\Sp_{(p)})$. Furthermore, assume that $v'$ is as in~\cref{enum:AsymptoticUniqueness} and $V'/v'$ admits a right-unital multiplication in $\LMod_{V'}(\Sp)$. Then the canonical left-$V$-module map\label{enum:EmFinityTypeTechnical}
			\begin{equation*}
				V'/v'^{m+1}\otimes_{V'}V\longrightarrow V/v^{m+1}\,.
			\end{equation*}
			can be upgraded to an $\IE_m$-algebra map in $\LMod_V(\Sp)$, where we equip $V/v^{m+1}$ and $V'/v'^{m+1}$ with Burklund's $\IE_m$-structures in $\LMod_V(\Sp)$ and $\LMod_{V'}(\Sp)$, respectively. 
		\end{alphanumerate}
	\end{lem}
	\begin{proof}[Proof sketch]
		Part~\cref{enum:AsymptoticUniqueness} follows immediately from asymptotic uniqueness of $v_n$-self maps (see \cite[Lemma~\href{https://www.math.ias.edu/~lurie/252xnotes/Lecture27.pdf\#page=3}{27.10}]{LurieChromatic} for example).
		
		To show~\cref{enum:EmFinityTypeTechnical}, let us denote $V/v'^{m+1}\coloneqq V'/v'^{m+1}\otimes_{V'}V$ for short. First note that the claim is not completely automatic, since the $\IE_m$-structures on $V/v^{m+1}$ and $V/v'^{m+1}$ are constructed via different deformation categories. More precisely, let $\Qq$ and $\Qq'$ be the classes of morphisms in $\LMod_V(\Sp)^\omega$ that become split epimorphisms upon $-\otimes_VV/v$ or $-\otimes_VV/v'$, respectively. Then the $\IE_m$-structure on $V/v^{m+1}$ is constructed via $\Def(\LMod_V(\Sp);\Qq)$, whereas for $V/v'^{m+1}$ we use $\Def(\LMod_V(\Sp);\Qq')$.
		
		Our assumptions on $v$ and $v'$ imply that $V/v'\rightarrow V/v$ can be turned into an $\IE_1$-map in $\LMod_V(\Sp)$. This need not be compatible with the $\IE_1$-structures on $V/v^{m+1}$ or $V/v'^{m+1}$, but it is enough to ensure $\Qq'\subseteq \Qq$, because any morphism that becomes split after $-\otimes_VV/v'$ will also become split after $(-\otimes_{V}V/v')\otimes_{V/v'}V/v\simeq -\otimes_VV/v$. Sheafification then induces a strongly continuous $\IE_{m+1}$-monoidal functor $\Def(\LMod_V(\Sp);\Qq')\rightarrow \Def(\LMod_V(\Sp);\Qq)$ which fits into a commutative diagram
		\begin{equation*}
			\begin{tikzcd}
				\Def\bigl(\LMod_V(\Sp);\Qq'\bigr)\rar &\Def\bigl(\LMod_V(\Sp);\Qq\bigr)\\
				\LMod_V(\Sp)\uar["\nu'"]\urar["\nu"']
			\end{tikzcd}
		\end{equation*}
		where $\nu$ and $\nu'$ denote the respective Yoneda embeddings.
		
		Let us now denote deformations in $\Def(\LMod_V(\Sp);\Qq)$ by $(\widetilde{-})$ as in the proof of \cref{prop:SquareZeroExtensions}. Via the functor above and \cite[Proposition~\chref{2.4}]{BurklundMooreSpectra}, we can write $\widetilde{V}/\widetilde{v}'^{m+1}$ as an iterated pushout of $\IE_m$-algebras in $\Def(\LMod_V(\Sp);\Qq)$. This yields a sequence of obstructions to constructing an $\IE_m$-algebra map $\widetilde{V}/\widetilde{v}'^{m+1}\rightarrow \widetilde{V}/\widetilde{v}^{m+1}$. Since the functor $\Def(\LMod_V(\Sp);\Qq')\rightarrow \Def(\LMod_V(\Sp);\Qq)$ intertwines $\nu'$ and $\nu$, the obstructions are still of the form that automatically vanishes.
	\end{proof}
	
	\begin{exm}\label{exm:THHrefChromatic}
		For all $m\geqslant 2$ and $n\geqslant 0$ let us construct a tower of $\IE_m$-algebras
		\begin{equation*}
			V(n)_0\longleftarrow V(n)_1\longleftarrow V(n)_2\longrightarrow \dotsb
		\end{equation*}
		of the form $V(n)_r\simeq \IS/(p^{\alpha_{r,0}},v_1^{\alpha_{r,1}},\dotsc,v_n^{\alpha_{r,n}})$, such that Assumption~\cref{par:MotivesOfInterest}\cref{enum:V} is satisfied. Note that the dualisability condition in \cref{par:MotivesOfInterest}\cref{enum:V} is trivial and the thick tensor ideal condition is automatic by the thick subcategory theorem (see \cite[Theorem~\href{https://www.math.ias.edu/~lurie/252xnotes/Lecture26.pdf\#page=2}{26.8}]{LurieChromatic} for example). So we only have to construct the tower and verify the factorisation condition.
		
		We use induction on~$n$. Suppose we've already constructed a tower of $\IE_{m+1}$-algebras $(V(n-1)_r)_{r\geqslant 0}$ with the desired properties. We'll write $V_r\coloneqq V(n-1)_r$ for brevity. Using \cref{lem:EmFinityTypeTechnical} for $V_{r+1}\rightarrow V_r$, we can inductively construct $v_n$-self maps $v_{n,r}\colon \Sigma^{N_r}V_{r}\rightarrow Vr$ such that each of them is the fourth power of another $v_n$-self map and the quotients 
		\begin{equation*}
			\ov V_r\coloneqq V_r/v_{n,r}^{2^r(m+1)}
		\end{equation*}
		fit into a tower of $\IE_m$-algebras. Note that this would already work with $V_r/v_{n,r}^{m+1}$; the extra factor in the exponent will only be used for the factorisation condition.
		
		As in the proof of \cref{lem:TorsionSpectra}, consider the right-$V_{r+1}\otimes V_r^\op$-module structure on $\ov V_{r+1}\otimes \ov V_r$ given by its \enquote{inner} bimodule structure. Since $V_{r+1}\otimes V_r\rightarrow V_r\otimes V_r$ factors through $V_r$ by the inductive hypothesis, we see that $\ov V_{r+1}\otimes \ov V_r\rightarrow \ov V_r\otimes \ov V_r$ factors through
		\begin{equation*}
			\bigl(\ov V_{r+1}\otimes \ov V_r\bigr)\otimes_{V_{r+1}\otimes V_r^\op}V_r\simeq \bigl(\ov V_{r+1}\otimes_{V_{r+1}}V_r\bigr)\otimes_{V_r}\ov V_r
		\end{equation*}
		as a map of $\ov V_{r+1}$-$\ov V_r$-bimodules. If we now consider the composition $\ov V_{r+2}\otimes \ov V_r\rightarrow \ov V_r\otimes \ov V_r$, we see that it factors through
		\begin{equation*}
			V_r/v_{n,r+1}^{2^{r+2}(m+1)}\otimes_{V_r} \ov V_r\longrightarrow V_r/v_{n,r+1}^{2^{r+1}(m+1)}\otimes_{V_r} \ov V_r\,.
		\end{equation*}
		This, in turn, factors through $\ov V_r$ as a map of $\IE_1$-algebras in $\RMod_{\ov V_r}(\Sp)$. Indeed, this follows from \cref{cor:E1FactorisationSpalpha} via base change along $V_r/v_{n,r+1}^{2^r(m+1)}\rightarrow V_r/v_{n,r}^{2^r(m+1)}\simeq \ov V_r$. So we get the desired factorisation for $\ov V_{r+2}\otimes \ov V_r\rightarrow \ov V_r\otimes \ov V_r$. Thus, if we put $V(n)_r\coloneqq \ov V_{2r}$, we get a tower of the desired form.
		
		With these disgusting technicalities out of the way, we can finally apply \cref{thm:RefinedInvariants}: We deduce that $\THHref(\L_n^{f}\IS_{(p)}/\IS_{(p)})$ is obtained from $\IS_{(p)}$ by killing an idempotent pro-algebra of the form $\prolim_{r\geqslant 0}\THH(\IS/(p^{\alpha_{r,0}},v_1^{\alpha_{r,1}},\dotsc,v_n^{\alpha_{r,n}}))$. In particular, we get a cofibre sequence
		\begin{equation*}
			\indcolim_{r\geqslant 0}\THH\bigl(\IS/(p^{\alpha_{r,0}},v_1^{\alpha_{r,1}},\dotsc,v_n^{\alpha_{r,n}})\bigr)^\vee\longrightarrow \IS_{(p)}\longrightarrow \THHref\bigl(\L_n^f\IS_{(p)}/\IS_{(p)}\bigr)
		\end{equation*}
		in $\Nuc\Ind(\Sp_{(p)}^{\B S^1})$.
	\end{exm}

	\newpage
	
	\section{Refined \texorpdfstring{$\THH$ and $\TC^-$}{THH and TC-} over \texorpdfstring{$\ku$}{ku}}\label{sec:RefinedTC-}
	
	We've seen in \cref{exm:THHrefk1m} that to compute $\THHref(\IQ)$, one essentially has to compute an ind-object of the form $\indcolim_{\alpha\geqslant 2}\THH(\IS/p^\alpha)^\vee$ for all primes~$p$. This seems currently out of reach. However, after base change to $\ku$, we can get some control over $\THH((\ku\otimes\IS/p^\alpha)/\ku)$ thanks to the results from \cite{qdeRhamku}, and so $\THHref(\ku\otimes\IQ/\ku)$ is approachable.
	
	In this section we study $\TCref(\ku\otimes\IQ/\ku)$ and $\TCref(\KU\otimes\IQ/\KU)$, which contain the same information as $\THHref(\ku\otimes\IQ/\ku)$ and $\THHref(\KU\otimes\IQ/\KU)$ by \cref{lem:S1ActiontComplete} below. In \cref{subsec:RefinedTC-}, we compute the homotopy groups
	\begin{equation*}
		\A_{\ku}^*\coloneqq \pi_{2*}\TCref(\ku\otimes\IQ/\ku)\quad\text{and}\quad\A_{\KU}\coloneqq\pi_0\TCref(\KU\otimes\IQ/\KU)
	\end{equation*}
	in terms of certain $q$-Hodge filtrations $\fil_{\qHodge}^\star\qdeRham_{(\IZ/p^\alpha)/\IZ_p}$ and the associated $q$-Hodge complexes $\qHodge_{(\IZ/p^{\alpha})/\IZ_p}$ that we get from the chosen $\IE_1$-structures on $\IS/p^\alpha$. In \cref{subsec:ElementaryProof} we'll explain how to describe these objects explicitly. These explicit descriptions will then be used in \cref{sec:Overconvergent} to finish the proof of \cref{thm:OverconvergentNeighbourhood,thm:OverconvergentNeighbourhoodEquivariant}.
	
	\begin{numpar}[Convention]\label{conv:ImplicitCompletion3}
		Throughout~\crefrange{sec:RefinedTC-}{sec:Overconvergent}, all ($q$\nobreakdash-)de Rham complexes and $q$-Hodge complexes relative to a $p$-complete ring will be implicitly $p$-completed.
	\end{numpar}
	
	\subsection{\texorpdfstring{$q$}{q}-Hodge filtrations and \texorpdfstring{$\TCref(\ku\otimes\IQ/\ku)$}{TC-ref(ku(x)Q/ku)}}\label{subsec:RefinedTC-}
	
	We begin by showing that for complex orientable ring spectra~$k$, $\THHref(k\otimes\IQ/k)$ with its $S^1$-action contains the same information as $\TCref(k\otimes\IQ/k)$.
	\begin{lem}\label{lem:S1ActiontComplete}
		Let $k$ be a complex orientable $\IE_\infty$-ring spectrum, equipped with trivial $S^1$-action, and let $t\in\pi_{-2}(k^{\h S^1})$ be any complex orientation. Then taking $S^1$-fixed points defines a symmetric monoidal equivalence
		\begin{equation*}
			(-)^{\h S^1}\colon \Mod_k(\Sp)^{\B S^1}\overset{\simeq}{\longrightarrow} \Mod_{k^{\h S^1}}(\Sp)_t^\complete\,,
		\end{equation*}
		where $\Mod_{k^{\h S\smash{^1}}}(\Sp)_t^\complete$ denotes $\infty$-category of $t$-complete $k^{\h S^1}$-module spectra, which we equip with the $t$-completed tensor product $-\cotimes_{k^{\h S^1}}-$.
	\end{lem}
	\begin{proof}
		By construction $(-)^{\h S^1}$ is lax symmetric monoidal. To see that it is strictly symmetric monoidal, we must check whether $M^{\h S^1}\cotimes_{k^{\h S^1}} N^{\h S^1}\rightarrow (M\otimes_k N)^{\h S^1}$ is an equivalence. As both sides are $t$-complete, this can be checked modulo $t$, where it follows from \cite[Lemma~\chref{2.2.10}]{EvenFiltration} for example.
		
		By definition, $(-)^{\h S^1}\colon \Sp^{\B S^1}\rightarrow \Sp$ has a left adjoint, given by the symmetric monoidal functor $\const\colon \Sp\rightarrow \Sp^{\B S^1}$, which sends a spectrum $X$ to itself equipped with the trivial $S^1$-action. By general nonsense about how symmetric monoidal adjunctions pass to module categories, we see that $(-)^{\h S^1}\colon \Mod_{k}(\Sp)^{\B S^1}\simeq \Mod_k(\Sp^{\B S^1})\rightarrow \Mod_{k^{\h S^1}}(\Sp)$ admits a left adjoint $L$, which is given as the composition
		\begin{equation*}
			L\colon \Mod_{k^{\h S^1}}(\Sp)\xrightarrow{\const}\Mod_{k^{\h S^1}}(\Sp^{\B S^1})\xrightarrow{-\otimes_{k^{\h S^1}}k}\Mod_k\bigl(\Sp^{\B S^1}\bigr)\simeq \Mod_k(\Sp)^{\B S^1}\,.
		\end{equation*}
		In particular, on underlying $k$-modules, $L$ is simply given by $(-)/t$. Since $(-)/t$ is conservative on $t$-complete $k^{\h S^1}$-modules, it follows that $L\colon \Mod_{k^{\h S^1}}(\Sp)_t^\complete\rightarrow \Mod_k(\Sp)^{\B S^1}$ must be conservative too. Furthermore, the counit $c\colon L((-)^{\h S\smash{^1}})\Rightarrow \id$ is an equivalence, as follows from \cite[Lemma~\chref{2.2.10}]{EvenFiltration} again. Thus $(-)^{\h S^1}$ must be fully faithful. We conclude using the standard fact that an adjunction in which the right adjoint is fully faithful and the left adjoint is conservative must be a pair of inverse equivalences.
	\end{proof}
	We'll now set out to compute $\pi_*\TCref(\ku\otimes\IQ/\ku)$ and $\pi_*\TCref(\KU\otimes \IQ/\KU)$.
	\begin{numpar}[Outline of the computation.]\label{par:TC-BattlePlan}
		For convenience, let's call a positive integer $m$ \emph{high-powered} if its prime factorisation $m=\prod_pp^{\alpha_p}$ has the following property: For all primes~$p>2$ either $\alpha_p=0$ or $\alpha_p\geqslant 2$ and for $p=2$ either $\alpha_2=0$ or $\alpha_2$ is even and $\geqslant 4$. We let $\IN^\lightning\!$ denote the set of high-powered positive integers, partially ordered by divisibility.
		
		Since $\IS/4$ and $\IS/p$ admit right-unital multiplications, we can use Burklund's general construction \cite[Theorem~\chref{1.5}]{BurklundMooreSpectra}%
		\footnote{We could also use \cite[Theorem~\chref{3.2}]{BurklundMooreSpectra} to get another tower of $\IE_1$-algebras $\IS/8\leftarrow \IS/16\leftarrow\IS/32\leftarrow\dotsb$. This one is potentially different from ours (as different deformation categories are used in the construction). It will become apparent in \cref{par:ReductionToTorsionFree} why we made our choice.}
		to construct $\IE_1$-structures on 
		\begin{equation*}
			\IS/m\simeq \prod_p\IS/p^{\alpha_p}
		\end{equation*}
		for every high-powered $m$. These assemble into a functor $\IS/{-}\colon \IN^{\lightning}\!\rightarrow \Alg_{\IE_1}(\Sp)$. In the following we'll write $\ku/m\coloneqq \ku\otimes\IS/m$ and $\KU/m\coloneqq \KU\otimes\IS/m$, where it is understood that the $\IE_1$-structure is always base changed from the one on $\IS/m$ above. By \cref{exm:THHrefk1m} and \cref{lem:S1ActiontComplete}, we get a cofibre sequence
		\begin{equation*}
			\indcolim_{m\in (\IN^\lightning\!)^\op}\TC^-\bigl((\ku/m)/\ku\bigr)^\vee\longrightarrow \ku^{\h S^1}\longrightarrow \TCref\left(\ku\otimes \IQ/\ku\right)
		\end{equation*}
		(where now $(-)^\vee\coloneqq \Hom_{\ku^{\h S^1}}(-,\ku^{\h S^1})$ denotes the dual in $\ku^{\h S^1}$-modules) and a similar one for $\KU$. To compute the pro-object on the left, we'll proceed in three steps:
		\begin{alphanumerate}
			\item We compute the homotopy groups $\pi_*\TC^-((\ku/m)/\ku)$ and $\pi_*\TC^-((\KU/m)/\KU)$ using \cite[Theorem~\chref{4.27}]{qdeRhamku}. This will be the content of \cref{cor:TC-kupalpha}.\label{enum:TC-BattlePlanStepA}
			\item We compute $\pi_*\TC^-((\ku/m)/\ku)^\vee$ and $\pi_*\TC^-((\KU/m)/\KU)^\vee$, essentially showing that in this case taking duals commutes with $\pi_*$ in a derived way. This will be achieved in \cref{cor:Duals}.\label{enum:TC-BattlePlanStepB}
			\item We show that pro-idempotence and the transition maps being trace-class passes to homotopy groups in this case. This will be the content of \cref{cor:ProIdempotent,cor:ProTraceClass}.\label{enum:TC-BattlePlanStepC}
		\end{alphanumerate}
		This leads to a preliminary description of the homotopy rings $\pi_*\TCref(\ku\otimes\IQ/\ku)$ and $\pi_*\TCref(\KU\otimes \IQ/\KU)$ in \cref{thm:RefinedTC-qHodge}. 
	\end{numpar}
	
	We begin with step~\cref{enum:TC-BattlePlanStepA}.

	\begin{numpar}[Reduction to the $p$-torsion free case.]\label{par:ReductionToTorsionFree}
		Decomposing $m=\prod_pp^{\alpha_p}$ into prime powers, we have
		\begin{equation*}
			\TC^-\bigl((\ku/m)/\ku\bigr)\simeq \prod_p\TC^-\bigl((\ku/p^{\alpha_p})/\ku\bigr)\,,
		\end{equation*}
		so we may reduce to the case where $m=p^\alpha$ is a high-powered prime power. Let us remark that $\TC^-((\ku/p^\alpha)/\ku)$ is automatically $p$-complete. Indeed, it is $(\beta,t)$-complete and $\TC^-((\ku/p^\alpha)/\ku)/(\beta,t)\simeq \HH((\IZ/p^\alpha)/\IZ)$ is $p^\alpha$-torsion, hence $p$-complete.
		%
		
		To compute $\TC^-((\ku/p^\alpha)/\ku)$, we lift to a $p$-torsion free case. Let $\IZ_p\{x\}_\infty$ be the free $p$-complete perfect $\delta$-ring on a generator $x$. Since the $p$-completed cotangent complex of $\IZ_p\{x\}_\infty$ vanishes, it lifts uniquely to a $p$-complete connective $\IE_\infty$-ring spectrum, which we'll denote $\IS_{\IZ_p\{x\}_\infty}$. By \cite[Theorem~\chref{1.5}]{BurklundMooreSpectra}, we get a tower of $\IE_1$-algebras in $\IS_{\IZ_p\{x\}_\infty}$-modules
		\begin{equation*}
			\IS_{\IZ_p\{x\}_\infty}/x^2\longleftarrow\IS_{\IZ_p\{x\}_\infty}/x^3\longleftarrow \IS_{\IZ_p\{x\}_\infty}/x^4\longleftarrow\dotsb
		\end{equation*}
		for $p> 2$; the case $p=2$ needs powers of $x^2$ instead. The map of perfect $\delta$-rings $\IZ_p\{x\}_\infty\rightarrow \IZ_p$ sending $x\mapsto p$ lifts uniquely to an $\IE_\infty$-map $\IS_{\IZ_p\{x\}_\infty}\rightarrow \IS_p$. If we base change the tower above along this map, we get the tower of $\IE_1$-algebras $(\IS/p^\alpha)$ from \cref{par:TC-BattlePlan}. Indeed, this follows from the uniqueness statement in \cite[Theorem~\chref{1.5}]{BurklundMooreSpectra}.%
		\footnote{Burklund only shows that the objects in the tower are unique and therefore satisfy base change. But the same argument shows that the transition maps too are unique, so they satisfy base change as well.}
		%
	
	Now put $\ku_{\IZ_p\{x\}_\infty}\coloneqq (\ku\otimes \IS_{\IZ_p\{x\}_\infty})_p^\complete$. Then $\THH(-/\ku_{\IZ_p\{x\}_\infty})_p^\complete\simeq \THH(-/\ku)_p^\complete$ holds by the same argument as in \cite[Proposition~\chref{11.7}]{BMS2} and so we get a base change equivalence
	\begin{equation*}
		\left(\TC^-\bigl((\ku_{\IZ_p\{x\}_\infty}/x^\alpha)/\ku\bigr)\otimes_{\ku_{\IZ_p\{x\}_\infty}}\ku_p^\complete\right)_{(p,t)}^\complete\overset{\simeq}{\longrightarrow}\TC^-\bigl((\ku/p^\alpha)/\ku\bigr)\,.
	\end{equation*}
	\end{numpar}
	\begin{numpar}[A $q$-Hodge filtration for $\IZ/m$.]\label{par:qHodgeZm}
		We can apply \cite[Theorem~\chref{4.17}]{qdeRhamku} to $\IZ_p\{x\}_\infty/x^\alpha$ with its spherical $\IE_1$-lift $\IS_{\IZ_p\{x\}_\infty}/x^\alpha$ to obtain a filtration $\fil_{\qHodge}^\star\qdeRham_{(\IZ_p\{x\}_\infty/x^\alpha)/\IZ_p}$ on the $p$-completed derived de Rham complex of $\IZ_p\{x\}_\infty/x^\alpha$. This filtration doesn't depend on the choice of spherical lift (only its existence) and $q$-deforms the Hodge filtration on $\deRham_{(\IZ_p\{x\}/x^\alpha)/\IZ_p}$. We then construct a filtration on $\qdeRham_{(\IZ/p^\alpha)/\IZ_p}$ as the base change
		\begin{equation*}
			\fil_{\qHodge}^\star\qdeRham_{(\IZ/p^\alpha)/\IZ_p}\coloneqq 	\left(\fil_{\qHodge}^\star\qdeRham_{(\IZ_p\{x\}_\infty/x^\alpha)/\IZ_p}\lotimes_{\IZ_p\{x\}_\infty}\IZ_p\right)_{(p,q-1)}^\complete\,.
		\end{equation*}
		For a general high-powered positive integer $m\in\IN^{\lightning}\!$ with prime factorisation $m=\prod_pp^{\alpha_p}$, we put
		\begin{equation*}
			\fil_{\qHodge}^\star\qdeRham_{(\IZ/m)/\IZ}\coloneqq 		\prod_p\fil_{\qHodge}^\star\qdeRham_{(\IZ/p^{\alpha_p})/\IZ_p}\,,
		\end{equation*}
		and denote its completion by $\fil_{\qHodge}^\star\qhatdeRham_{(\IZ/m)/\IZ}$. We'll verify in \cref{lem:qHodgeZm} below that the filtered object $\fil_{\qHodge}^\star\qdeRham_{(\IZ/m)/\IZ}$ is indeed a \emph{$q$-Hodge filtration} in the sense of \cite[Definition~\chref{3.2}]{qWittHabiro}, as the notation suggests.
		
		We regard these filtrations as filtered modules over the filtered ring $(q-1)^\star\IZ\qpower$. In the following, this filtered ring will be identified with the graded ring $\pi_{2*}(\ku^{\h S^1})\cong \IZ[\beta]\llbracket t\rrbracket$, where $t$ sits graded degree~$-1$ and plays the role of the filtration parameter, $\beta$ sits in graded degree~$1$, and $\beta t=(q-1)$. We will also consider the \emph{$q$-Hodge complex}
		\begin{equation*}
			\qHodge_{(\IZ/m)/\IZ}\coloneqq 	\colimit\left(\fil_{\qHodge}^0\qdeRham_{(\IZ/m)/\IZ}\xrightarrow{(q-1)}\fil_{\qHodge}^1\qdeRham_{(\IZ/m)/\IZ}\xrightarrow{(q-1)}\dotsb\right)_{(q-1)}^\complete
		\end{equation*}
		as in \cite[\chref{3.5}]{qWittHabiro}.
	\end{numpar}
	\begin{rem}
		Let $m=\prod_pp^{\alpha_p}$ be an integer such that for all primes $p>2$ either $\alpha_p=0$ or $\alpha_p\geqslant 3$ and for $p=2$ either $\alpha_2=0$ or $\alpha_2$ is even and $\geqslant 6$. Then the $\IE_1$-structure on $\IS/m$ can be upgraded to an $\IE_2$-structure. We can thus apply \cite[Theorem~\chref{4.27}]{qdeRhamku} directly to obtain another $q$-Hodge filtration on $\qdeRham_{(\IZ/m)/\IZ}$, without having to go through the base change above. However, this $q$-Hodge filtration agrees with the one from \cref{par:qHodgeZm}.
		
		Indeed, note that the $\IE_1$-structure on $\IS_{\IZ_p\{x\}_\infty}/x^{\alpha_p}$ also admits an $\IE_2$-upgrade, compatible with the one on $\IS/p^{\alpha_p}$. Then the assertion follows by naturality and the observation that the solid even filtration on the already even $\IE_1$-ring spectrum $\TC_\solid^-((\ku_{\IZ_p\{x\}_\infty}/x^{\alpha_p})/\ku)$ necessarily agrees with the double-speed Whitehead filtration $\tau_{\geqslant 2\star}$.
	\end{rem}
	\begin{lem}\label{lem:qHodgeZm}
	The filtration $\fil_{\qHodge}^\star\qdeRham_{(\IZ/m)/\IZ}$ is a $q$-Hodge filtration in the sense of \cite[Definition~\chref{3.2}]{qWittHabiro}. Moreover, $\qdeRham_{(\IZ/m)/\IZ}$ and $\qHodge_{(\IZ/m)/\IZ}$ are static $(q-1)$-torsion free rings and the $q$-Hodge filtration is a descending filtration by ideals.
	\end{lem}
	\begin{proof}
		Let us verify the conditions from \cite[Definition~\chref{3.2}]{qWittHabiro}.
		For any prime~$p$, the non-$p$-completed derived $q$-de Rham complex $\qdeRham_{(\IZ/p^{\alpha_p})/\IZ}$ vanishes after $(-)[1/p]_{(q-1)}^\complete$, as $(\IZ/p^{\alpha_p})[1/p]\cong 0$. Hence it also vanishes after $(-)_p^\complete[1/p]_{(q-1)}^\complete$, as any module over the trivial ring is trivial. It follows that $\qdeRham_{(\IZ/p^{\alpha_p})/\IZ}$ is already $p$-complete and thus agrees with $\qdeRham_{(\IZ/p^{\alpha_p})/\IZ_p}$.
	
		With this observation, condition~(\href{https://guests.mpim-bonn.mpg.de/ferdinand/q-Habiro.pdf#Item.23}{$a$}) of \cite[Definition~\chref{3.2}]{qWittHabiro} is straightforward to verify. Condition~(\href{https://guests.mpim-bonn.mpg.de/ferdinand/q-Habiro.pdf#Item.24}{$b$}) follows via base change from $\IZ_p\{x\}_\infty/x^{\alpha_p}$. The other two conditions are vacuous, since the rationalisations vanish. Therefore, $\fil_{\qHodge}^\star\qdeRham_{(\IZ/m)/\IZ}$ is indeed a $q$-Hodge filtration.
	
		To verify that $\fil_{\qHodge}^\star\qdeRham_{(\IZ/m)/\IZ}$ is degree-wise static and $(q-1)$-torsion free, just observe that its reduction modulo~$(q-1)$ is $\fil_{\Hodge}^\star\deRham_{(\IZ/m)/\IZ}$, which is degree-wise static. Via base change from $\IZ_p\{x\}_\infty/x^{\alpha_p}$ it's then clear that $\fil_{\qHodge}^\star\qdeRham_{(\IZ/m)/\IZ}$ must be a descending filtration by ideals. By construction, this implies that $\qHodge_{(\IZ/m)/\IZ}$ is a static and $(q-1)$-torsion free ring, as claimed.
	\end{proof}
	
	The upshot of \cref{par:ReductionToTorsionFree}--\labelcref{lem:qHodgeZm} is the following.
	
	\begin{cor}\label{cor:TC-kupalpha}
	Let $m\in\IN^{\lightning}\!$ be a high-powered positive integer. Then the spectra $\TC^-((\ku/m)/\ku)$ and $\TC^-((\KU/m)/\KU)$ are concentrated in even degrees and we have
	\begin{align*}
		\pi_{2*}\TC^-\bigl((\ku/m)/\ku\bigr)&\cong \fil_{\qHodge}^\star\qhatdeRham_{(\IZ/m)/\IZ}\,,\\ \pi_{2*}\TC^-\bigl((\KU/m)/\KU\bigr)&\cong \qHodge_{(\IZ/m)/\IZ}[\beta^{\pm 1}]\,.
	\end{align*}
	\end{cor}
	\begin{proof}
	It's enough to check evenness modulo $t$, so we may pass from $\TC^-$ to $\THH$. Since $\THH((\ku/m)/\ku)$ is connective, we may further pass to $\THH((\ku/m)/\ku)/\beta\simeq \HH((\IZ/m)/\IZ)$, which is indeed even. This shows evenness for $\THH((\ku/m)/\ku)$ and then the same follows for $\THH((\ku/m)/\ku)[1/\beta]\simeq \THH((\KU/m)/\KU)$.
	
	By decomposing $m$ into prime factors as in \cref{par:ReductionToTorsionFree} and using the base change equivalence, we get a map
	\begin{equation*}
		\fil_{\qHodge}^\star\qhatdeRham_{(\IZ/m)/\IZ}\rightarrow\pi_{2\star}\TC^-((\ku/m)/\ku)\,.
	\end{equation*}
	Whether this is an equivalence can be checked modulo~$\beta$, where we recover the well-known fact that the even homotopy groups of $\TC^-((\ku/m)/\ku)/\beta\simeq \HC^-((\IZ/m)/\IZ)$ are the completed Hodge filtration $\fil_{\Hodge}^\star\hatdeRham_{(\IZ/m)/\IZ}$. The claim that the even homotopy groups of
	\begin{equation*}
		\TC^-\bigl((\KU/m)/\KU\bigr)\simeq \TC^-\bigl((\ku/m)/\ku\bigr)\bigl[\localise{\beta}\bigr]_t^\complete
	\end{equation*}
	are given by $\qHodge_{(\IZ/m)/\IZ}[\beta^{\pm 1}]$ follows formally.
	\end{proof}
	
	This finishes step~\cref{enum:TC-BattlePlanStepA} of our plan in~\cref{par:TC-BattlePlan}.
	We'll now commence step~\cref{par:TC-BattlePlan}\cref{enum:TC-BattlePlanStepB}. We start with a general fact (which is usually formulated as a spectral sequence).
	\begin{lem}\label{lem:ExtSpectralSequence}
	Let $k$ be an even $\IE_1$-ring spectrum and let $M,N$ be even left-$k$-modules. Then the mapping spectrum $\Hom_k(M,N)$ admits a complete exhaustive descending filtration with graded pieces
	\begin{equation*}
		\gr^*\Hom_k(M,N)\simeq \Sigma^{2*}\RHhom_{\pi_{2*}(k)}\bigl(\pi_{2*}(M),\pi_{2*}(N)\bigr)\,.
	\end{equation*}
	Here $\Sigma^{2*}\colon \Gr(\Sp)\rightarrow \Gr(\Sp)$ is the \enquote{double shearing} functor and $\RHhom_{\pi_{2*}(k)}$ denotes the derived internal Hom in graded $\pi_{2*}(k)$-modules.
	\end{lem}
	\begin{proof}
	In the usual adjunction $\colimit\colon \Fil(\Sp)\shortdoublelrmorphism \Sp\noloc \const$, the left adjoint is symmetric monoidal and the right adjoint is lax symmetric monoidal. Furthermore, $\colimit \tau_{\geqslant 2\star}(k)\simeq k$. It follows formally that $\colimit\colon \LMod_{\tau_{\geqslant 2\star}(k)}(\Fil(\Sp))\shortdoublelrmorphism \LMod_k(\Sp)\noloc \const$ is an adjunction as well and so $\Hom_k(M,N)\simeq \Hom_{\tau_{\geqslant 2\star}(k)}(\tau_{\geqslant 2\star}(M),\const N)$. Hence we may define the desired filration via 
	\begin{equation*}
		\fil^n\Hom_k(M,N)\coloneqq \Hom_{\tau_{\geqslant 2\star}(k)}\bigl(\tau_{\geqslant 2\star}(M),\tau_{\geqslant 2(\star+n)}(N)\bigr)\,.
	\end{equation*}
	This filtration is clearly complete since we may pull $0\simeq \limit_{n\rightarrow \infty}\tau_{\geqslant 2(\star+n)}(N)$ out of the $\Hom$. To show that the filtration is exhaustive, we need to check that $\const N\simeq \colimit_{n\rightarrow -\infty}\tau_{\geqslant 2(\star+n)}(N)$ can similarly be pulled out of the $\Hom$. To this end, recall that $\Fil(\Sp)$ can be equipped with the \emph{double Postnikov $t$-structure} in which objects in the image of $\tau_{\geqslant 2\star}(-)$ are connective and connective objects are closed under tensor products (see \cite[Construction~\chref{3.3.6}]{RaksitFilteredCircle} for example and double everything). Then $\Mod_{\tau_{\geqslant 2\star}(k)}(\Fil(\Sp))$ inherits a $t$-structure in which $\tau_{\geqslant 2\star}(M)$ is connective and the cofibres of $\tau_{\geqslant 2(\star+n)}(N)\rightarrow \const N$ get more and more coconnective as $n\to-\infty$. This shows that the colimit can be pulled out.
	
	It remains to determine the associated graded. By construction, the $n$\textsuperscript{th} graded piece is given by $\gr^n\Hom_k(M,N)\simeq \Hom_{\tau_{\geqslant 2\star}(k)}(\tau_{\geqslant 2*}(M),\Sigma^{2(\star+n)}\pi_{2(\star+n)}(N))$. To simplify this further, let $\IS_{\Gr}$ and $\IS_{\Fil}$ denote the tensor units in graded and filtered spectra, respectively. By abuse of notation, we identify $\IS_{\Fil}$ with its underlying graded spectrum. As remarked in \cref{conv:Filtrations}, we have $\Fil(\Sp)\simeq \Mod_{\IS_{\Fil}}(\Gr(\Sp))$; this identifies passing to the associated graded with the base change functor $-\otimes_{\IS_{\Fil}}\IS_{\Gr}$. Since the $\IS_{\Fil}$-module structure on $\Sigma^{2(\star+n)}\pi_{2(\star+n)}(N)$ already factors through $\IS_{\Fil}\rightarrow \IS_{\Gr}$, we obtain
	\begin{align*}
		\Hom_{\tau_{\geqslant 2\star}(k)}\bigl(\tau_{\geqslant 2*}(M),\Sigma^{2(\star+n)}\pi_{2(\star+n)}(N)\bigr)&\simeq \Hom_{\Sigma^{2*}\pi_{2*}(k)}\bigl(\Sigma^{2*}\pi_{2*}(M),\Sigma^{2(*+n)}\pi_{2(*+n)}(N)\bigr)\\
		&\simeq \Sigma^{2n}\Hom_{\pi_{2*}(k)}\bigl(\pi_{2*}(M),\pi_{2*}(N)(-n)\bigr)\,.
	\end{align*}
	The first step is the usual base change equivalence for $\tau_{\geqslant 2\star}(k)\rightarrow \tau_{\geqslant 2\star}(k)\otimes_{\IS_{\Fil}}\IS_{\Gr}\simeq \Sigma^{2*}\pi_{2*}(k)$, the second step uses that the shearing functor $\Sigma^{2*}\colon \Gr(\Sp)\rightarrow \Gr(\Sp)$ is an $\IE_1$-monoidal equivalence (even $\IE_2$-monoidal, see \cite[Proposition~\chref{3.10}]{ExamplesOfDiskAlgebras}, but we don't need that). Now the right-hand side is precisely the $n$\textsuperscript{th} graded piece of $\RHhom_{\pi_{2*}(k)}(\pi_{2*}(M),\pi_{2*}(N))$ and so we're done.
	\end{proof}
	
	We'll apply this now in the case $k\simeq \ku^{\h S^1}$, so that $\pi_{2*}(k)\cong \IZ[\beta]\llbracket t\rrbracket$. We also let $\Eext_{\IZ[\beta]\llbracket t\rrbracket}^i$ denote the graded $\IZ[\beta]\llbracket t\rrbracket$-module $\H_{-i}\RHhom_{\IZ[\beta]\llbracket t\rrbracket}$ for all $i\geqslant 0$.
	
	\begin{cor}\label{cor:Duals}
	Let $m\in\IN^{\lightning}\!$ be a high-powered positive integer. Then the spectra $\TC^-((\ku/m)/\ku)^\vee$ and $\TC^-((\KU/m)/\KU)^\vee$ are concentrated in odd degrees and we have
	\begin{align*}
		\pi_{-(2*+1)}\TC^-\bigl((\ku/m)/\ku\bigr)^\vee&\cong \Eext_{\IZ[\beta]\llbracket t\rrbracket}^1\left(\fil_{\qHodge}^\star\qhatdeRham_{(\IZ/m)/\IZ},\IZ[\beta]\llbracket t\rrbracket\right)\\
		\pi_{-(2*+1)}\TC^-\bigl((\KU/m)/\KU\bigr)^\vee&\cong \Ext_{\IZ\qpower}^1\left(\qHodge_{(\IZ/m)/\IZ},\IZ\qpower\right)[\beta^{\pm 1}]\,.
	\end{align*}
	\end{cor}
	\begin{proof}
	According to \cref{cor:TC-kupalpha} and \cref{lem:ExtSpectralSequence}, the spectrum $\TC^-((\ku/m)/\ku)^\vee$ admits a complete exhaustive filtration with associated graded $\Sigma^{2*}(\fil_{\qHodge}^\star\qhatdeRham_{(\IZ/m)/\IZ})^\vee$, where now the dual is taken in graded $\IZ[\beta]\llbracket t\rrbracket$-modules. It'll be enough to show that this dual is concentrated in homological degree~$-1$ (which precisely accounts for the $\Eext_{\IZ\qpower[\beta^{\pm 1}]}^1$-terms). Since $\IZ[\beta]\llbracket t\rrbracket$ is $(\beta,t)$-complete as a graded object, the same is true for any dual in graded $\IZ[\beta]\llbracket t\rrbracket$-modules, and so it'll be enough that
	\begin{equation*}
		\RHhom_{\IZ[\beta]\llbracket t\rrbracket}\left(\fil_{\qHodge}^\star\qhatdeRham_{(\IZ/m)/\IZ},\IZ[\beta]\llbracket t\rrbracket\right)/(\beta,t)\simeq \RHhom_\IZ\left(\gr_{\Hodge}^*\hatdeRham_{(\IZ/m)/\IZ},\IZ\right)
	\end{equation*}
	is concentraded in homological degree $-1$. Since $\gr_{\Hodge}^n\hatdeRham_{(\IZ/m)/\IZ}\simeq \Sigma^{-n}\bigwedge^n\L_{(\IZ/m)/\IZ}\simeq \IZ/m$, the $n$\textsuperscript{th} graded piece of the right-hand side is precisely $\RHom_\IZ(\IZ/m,\IZ)$, which is indeed concentrated in homological degree $-1$. This finishes the proof for $\TC^-((\ku/m)/\ku)^\vee$.
	
	The proof for $\TC^-((\KU/m)/\KU)^\vee$ is analogous, except that we need a different argument to show that the dual $(\qHodge_{(\IZ/m)/\IZ})^\vee$ in $\IZ\qpower$-modules is concentrated in homological degree~$-1$. By $(q-1)$-completeness, it'll be enough to check the same for $\RHom_\IZ(\qHodge_{(\IZ/m)/\IZ}/(q-1),\IZ)$. By \cite[\chref{3.8}]{qWittHabiro} we see that $\qHodge_{(\IZ/m)/\IZ}/(q-1)$ admits an exhaustive ascending filtration with associated graded given by $\gr_{\Hodge}^*\deRham_{(\IZ/m)/\IZ}$. It follows that $\RHom_\IZ(\qHodge_{(\IZ/m)/\IZ}/(q-1),\IZ)$ admits a descending filtration with associated graded $\RHhom_\IZ(\gr_{\Hodge}^*\deRham_{(\IZ/m)/\IZ},\IZ)$. This is indeed concentrated in homological degree~$-1$ as we've seen above, so we're done.
	\end{proof}
	
	This finishes step~\cref{enum:TC-BattlePlanStepB} in our plan from~\cref{par:TC-BattlePlan}. We continue with step~\cref{enum:TC-BattlePlanStepC}. Note that neither pro-idempotence of $\prolim_{m\in \smash{\IN^\lightning\!}}\TC^-((\ku/m)/\ku)$ nor the fact that its transition maps become eventually trace-class are automatically preserved under passing to homotopy groups. The problem is that $\pi_*(-)$---or really passing to the associated graded of the Whitehead filtration $\tau_{\geqslant \star}$---is not a symmetric monoidal functor. 
	
	As we'll see, in our situation, passing to the associated graded of the \emph{double-speed} Whitehead filtration $\tau_{\geqslant 2\star}$ behaves as if it were symmetric monoidal, which fixes all issues. Our starting point is the following general fact, which is quite similar to \cref{lem:ExtSpectralSequence} (and is also usually formulated as a spectral sequence).
	\begin{lem}\label{lem:TorSpectralSequence}
	Let $k$ be an even $\IE_\infty$-ring spectrum, let $t\in \pi_{2*}(k)$ be a homogeneous element, and let $M$, $N$ be even $k$-modules. Then the $t$-completed tensor product $M\cotimes_kN$ admits a complete exhaustive descending filtration with graded pieces
	\begin{equation*}
		\gr^*(M\cotimes_kN)\simeq \Sigma^{2*}\left(\pi_{2*}(M)\clotimes_{\pi_{2*}(k)}\pi_{2*}(N)\right)\,.
	\end{equation*}
	Here $-\clotimes_{\pi_{2*}(k)}-$ denotes the graded $t$-completed derived tensor product over $\pi_{2*}(k)$.
	\end{lem}
	\begin{proof}
	The filtered spectrum $\tau_{\geqslant 2\star}(M)\otimes_{\tau_{\geqslant 2\star}(k)}\tau_{\geqslant 2\star}(N)$ defines a filtration on $M\otimes_kN$. This filtration is exhaustive, since $\colimit\colon \Fil(\Sp)\rightarrow \Sp$ is symmetric monoidal, and complete, since $\tau_{\geqslant 2\star}(M)\otimes_{\tau_{\geqslant 2\star}(k)}\tau_{\geqslant 2\star}(N)$ is a connective object in the double Postnikov $t$-structure (see the proof of \cref{lem:ExtSpectralSequence}).
	
	Now consider the $t$-adically completed tensor product $\tau_{\geqslant 2\star}(M)\cotimes_{\tau_{\geqslant 2\star}(k)}\tau_{\geqslant 2\star}(N)$, where $t$ in the filtration degree corresponding to its homotopical degree. This now defines a filtration on $M\cotimes_kN$, which is clearly still complete. It is also still exhaustive. Indeed, for all $n$, the cofibre of $(\tau_{\geqslant 2\star}(M)\otimes_{\tau_{\geqslant 2\star}(k)}\tau_{\geqslant 2\star}(N))_{-n}\rightarrow M\otimes N$ is $(2n+1)$-coconnective. Upon $t$-adic completion, the coconnectivity can go down by at most~$1$, and so we see that the cofibre of $(\tau_{\geqslant 2\star}(M)\cotimes_{\tau_{\geqslant 2\star}(k)}\tau_{\geqslant 2\star}(N))_{-n}\rightarrow M\cotimes N$ will still be $2n$-coconnective. This ensures exhaustiveness.
	
	Passing to the associated graded is symmetric and commutes with $t$-adic completion (in the filtered and graded setting, respectively). Moreover, the double shearing functor $\Sigma^{2*}$ is $\IE_1$-monoidal (even $\IE_2$, but we won't need that). Hence
	\begin{equation*}
		\gr^*(M\cotimes_kN)\simeq \Sigma^{2*}\pi_{2*}(M)\cotimes_{\Sigma^{2*}\pi_{2*}(k)}\Sigma^{2*}\pi_{2*}(N)\simeq \Sigma^{2*}\left(\pi_{2*}(M)\clotimes_{\pi_{2*}(k)}\pi_{2*}(N)\right)\,.\qedhere
	\end{equation*}
	\end{proof}
	\begin{cor}\label{cor:ProIdempotent}
	$\prolim_{m\in \smash{\IN^\lightning\!}}\fil_{\qHodge}^\star\qhatdeRham_{(\IZ/m)/\IZ}$ and $\prolim_{m\in \smash{\IN^\lightning\!}}\qHodge_{(\IZ/m)/\IZ}$ are idempotent pro-algebras, respectively, in the derived $\infty$-categories of $t$-complete graded $\IZ[\beta]\llbracket t\rrbracket$-modules and of $(q-1)$-complete $\IZ\qpower$-modules. 
	\end{cor}
	\begin{proof}
	Throughout the proof, $\cotimes$ will denote a $t$-completed tensor product. We also put $\fil^\star\qhatdeRham_{m}\coloneqq \fil_{\qHodge}^\star\qhatdeRham_{(\IZ/m)/\IZ}$ and $A\coloneqq \prolim_{m\in\smash{\IN^{\lightning}\!}}\fil^\star\qhatdeRham_{m}$ for short.
	
	Since each $\fil^\star\qhatdeRham_{m}$ is a graded $\IZ[\beta]\llbracket t\rrbracket$-algebra, we get a unit map $\IZ[\beta]\llbracket t\rrbracket\rightarrow A$ and a multiplication $A\clotimes_{\IZ[\beta]\llbracket t\rrbracket} A\rightarrow A$ such that the composition
	\begin{equation*}
		A\simeq \IZ[\beta]\llbracket t\rrbracket\clotimes_{\IZ[\beta]\llbracket t\rrbracket}A \longrightarrow A\clotimes_{\IZ[\beta]\llbracket t\rrbracket} A\longrightarrow A
	\end{equation*}
	is the identity. For the other composition, let $m_1,m_2\in\IN^{\lightning}\!$ and consider the $t$-completed tensor product
	\begin{equation*}
		\TC^-\bigl((\ku/m_1\otimes_{\ku}\ku/m_2)/\ku\bigr)\simeq \TC^-\bigl((\ku/m_1)/\ku\bigr)\cotimes_{\ku^{\h S^1}}\TC^-\bigl((\ku/m_2)/\ku\bigr)\,.
	\end{equation*}
	By \cref{lem:TorSpectralSequence}, this has a complete exhaustive filtration with graded pieces given by $\Sigma^{2*}(\fil^\star\qhatdeRham_{m_1}\clotimes_{\IZ[\beta]\llbracket t\rrbracket}\fil^\star\qhatdeRham_{m_2})$. Observe that this graded completed tensor product is concentrated in homological degrees $[0,1]$. Indeed, this can be checked modulo $(\beta,t)$. Then $\fil^\star\qhatdeRham_{m_i}/(\beta,t)\simeq \gr_{\Hodge}^*\deRham_{(\IZ/m_i)/\IZ_p}$ is given by $\IZ/m_i$ in every graded degree for $i=1,2$, and $\IZ/m_1\lotimes_\IZ\IZ/m_2$ is indeed concentrated in homological degrees $[0,1]$. It follows that the filtration on $	\TC^-((\ku/p^{m_1}\otimes_{\ku}\ku/p^{m_2})/\ku)$ must be the double speed Whitehad filtration $\tau_{\geqslant 2\star}$.
	
	By \cref{cor:E1FactorisationSpalpha}, $\TC^-((\ku/m^3\otimes_{\ku}\ku/m)/\ku)\rightarrow \TC^-((\ku/m^2\otimes_{\ku}\ku/m)/\ku)$ factors through the even spectrum $\TC^-((\ku/m)/\ku)$. By passing to the associated graded of the double speed Whitehead filtration, we see that
	\begin{equation*}
		\fil^\star\qhatdeRham_{m^3}\clotimes_{\IZ[\beta]\llbracket t\rrbracket}\fil^\star\qhatdeRham_{m}\longrightarrow\fil^\star\qhatdeRham_{m^2}\clotimes_{\IZ[\beta]\llbracket t\rrbracket}\fil^\star\qhatdeRham_{m^2}
	\end{equation*}
	factors through $\fil^\star\qhatdeRham_{m}$. This finishes the proof that $A=\prolim_{m\in\smash{\IN^{\lightning}\!}}\fil_{\qHodge}^\star\qhatdeRham_{(\IZ/m)/\IZ}$ is an idempotent pro-algebra.
	
	The argument for $\prolim_{m\in\smash{\IN^{\lightning}\!}}\qHodge_{(\IZ/m)/\IZ}$ is analogous, except that we work with $\KU$ instead of $\ku$, and to show that $\qHodge_{(\IZ/m_1)/\IZ}\clotimes_{\IZ\qpower}\qHodge_{(\IZ/m_2)/\IZ}$ is concentrated in homological degrees $[0,1]$, we need a slightly different argument: First, we can reduce modulo $(q-1)$. The conjugate filtration from \cite[\chref{3.8}]{qWittHabiro} gives an ascending filtration on $\qHodge_{(\IZ/m_i)/\IZ}/(q-1)$ for $i=1,2$, whose graded pieces are copies of $\IZ/m_i$. Moreover, $\qHodge_{(\IZ/m_i)/\IZ}/(q-1)$ is an $\IZ/m_i$-algebra, since $\qHodge_{(\IZ/m_i)/\IZ}$ contains an element of the form $m_i/(q-1)$. Thus, abstractly, $\qHodge_{(\IZ/m_i)/\IZ}/(q-1)\simeq \bigoplus_\IN\IZ/m_i$. So we're done since $\IZ/m_1\lotimes_\IZ\IZ/m_2$ is concentrated in homological degrees $[0,1]$.
	\end{proof}
	\begin{cor}\label{cor:ProTraceClass}
	$\prolim_{m\in \smash{\IN^\lightning\!}}\fil_{\qHodge}^\star\qhatdeRham_{(\IZ/m)/\IZ}$ and $\prolim_{m\in \smash{\IN^\lightning\!}}\qHodge_{(\IZ/m)/\IZ}$ are equivalent to pro-objects with trace-class transition maps.
	\end{cor}
	\begin{proof}
	Throughout the proof, $\cotimes$ will denote a $t$-completed tensor product. Using \cref{cor:E1FactorisationSpalpha} and unravelling the proof of \cref{lem:TraceClassTransitionMaps}, we find that that for every high-powered $m$, $\TC^-((\ku/m^{3})/\ku)\rightarrow \TC^-((\ku/m)/\ku)$ is trace-class in $t$-complete $\ku^{\h S^1}$-modules. Hence it must be induced by a map
	\begin{equation*}
		\eta\colon \ku^{\h S^1}\longrightarrow \TC^-\bigl((\ku/m^3)/\ku\bigr)^\vee\cotimes_{\ku^{\h S^1}}\TC^-\bigl((\ku/m)/\ku\bigr)
	\end{equation*}
	By \cref{lem:TorSpectralSequence} (applied to the shift $\Sigma\TC^-((\ku/m^3)/\ku)^\vee$ to get an even spectrum, then we shift back afterwards), the right-hand side has a complete exhaustive filtration with graded pieces $(\fil_{\qHodge}^\star\qhatdeRham_{(\IZ/m^3)/\IZ})^\vee\clotimes_{\IZ[\beta]\llbracket t\rrbracket}\fil_{\qHodge}^\star\qhatdeRham_{(\IZ/m)/\IZ}$. As in the proof of \cref{cor:ProIdempotent}, one easily checks that this graded completed tensor product is concentrated in homological degrees $[-1,0]$. It follows that the filtration must be given by $\tau_{\geqslant 2\star-1}(-)$. Thus, by considering $\tau_{\geqslant 2\star-1}(\eta)$ and then passing to associated gradeds, we obtain a morphism
	\begin{equation*}
		\IZ[\beta]\llbracket t\rrbracket \longrightarrow \bigl(\fil_{\qHodge}^\star\qhatdeRham_{(\IZ/m^3)/\IZ}\bigr)^\vee\clotimes_{\IZ[\beta]\llbracket t\rrbracket}\fil_{\qHodge}^\star\qhatdeRham_{(\IZ/m)/\IZ}\,.
	\end{equation*}
	which witnesses that the morphism $\fil_{\qHodge}^\star\qhatdeRham_{(\IZ/m^3)/\IZ_p}\rightarrow\fil_{\qHodge}^\star\qhatdeRham_{(\IZ/m)/\IZ}$ is indeed trace-class, as desired.
	
	The argument for $\qHodge_{(\IZ/m^3)/\IZ}\rightarrow\qHodge_{(\IZ/m)/\IZ}$ being trace-class is analogous, except that we use $\KU$ instead of $\ku$. Moreover, we need a different argument to show that $(\qHodge_{(\IZ/m^3)/\IZ})^\vee\clotimes_{\IZ\qpower}\qHodge_{(\IZ/m)/\IZ}$ is concentrated in homological degrees $[-1,0]$: First, we can reduce modulo $(q-1)$. As we've seen in the proof of \cref{cor:ProIdempotent}, on underlying abelian groups we get an equivalence $\qHodge_{(\IZ/m)/\IZ}/(q-1)\simeq \bigoplus_\IN\IZ/m$. An analogous conclusion holds for $\qHodge_{(\IZ/m^3)/\IZ}/(q-1)$. Thus, the tensor product modulo $(q-1)$ becomes $\Sigma^{-1}\prod_\IN\IZ/m^3\lotimes_\IZ\bigoplus_\IN\IZ/m$, which is clearly concentrated in homological degrees $[-1,0]$. 
	\end{proof}
	
	This finishes step~\cref{par:TC-BattlePlan}\cref{enum:TC-BattlePlanStepC} and we arrive at the result of our computation.
	
	\begin{thm}\label{thm:RefinedTC-qHodge}
	$\TCref(\ku\otimes\IQ/\ku)$ and $\TCref(\KU\otimes\IQ/\KU)$ are concentrated in even degrees. Furthermore, their even homotopy groups are given as follows:
	\begin{alphanumerate}
		\item $\pi_{2*}\TCref(\ku\otimes\IQ/\ku)\cong \A_{\ku}^*$, where $\A_{\ku}^*$ is obtained by killing the idempotent pro-graded $\IZ[\beta]\llbracket t\rrbracket$-algebra $\prolim_{m\in \smash{\IN^\lightning\!}}\fil_{\qHodge}^\star\qhatdeRham_{(\IZ/m)/\IZ}$. In particular, there's a short exact sequence\label{enum:RefinedTC-ku}
		\begin{equation*}
			0\longrightarrow \IZ[\beta]\llbracket t\rrbracket\longrightarrow \A_{\ku}^*\longrightarrow\indcolim_{m\in (\IN^\lightning\!)^\op}\Eext_{\IZ[\beta]\llbracket t\rrbracket}^1\left(\fil_{\qHodge}^\star\qhatdeRham_{(\IZ/m)/\IZ},\IZ[\beta]\llbracket t\rrbracket\right)\longrightarrow 0\,,
		\end{equation*}
		and $\A_{\ku}^*$ is an idempotent nuclear graded $\IZ[\beta]\llbracket t\rrbracket$-algebra.
		\item $\pi_{2*}\TCref(\KU\otimes \IQ/\KU)\cong \A_{\KU}[\beta^{\pm 1}]$, where $\A_{\KU}$ is obtained by killing the idempotent pro-$\IZ\qpower$-algebra $\prolim_{m\in\smash{\IN^{\lightning}\!}}\qHodge_{(\IZ/m)/\IZ}$. In particular, there's a short exact sequence\label{enum:RefinedTC-KU}
		\begin{equation*}
			0\longrightarrow \IZ\qpower\longrightarrow \A_{\KU}\longrightarrow\indcolim_{m\in (\IN^\lightning\!)^\op}\Ext_{\IZ\qpower}^1\left(\qHodge_{(\IZ/m)/\IZ},\IZ\qpower\right)\longrightarrow 0\,,
		\end{equation*}
		and $\A_{\KU}$ is an idempotent nuclear $\IZ\qpower$-algebra.
	\end{alphanumerate}
	\end{thm}
	\begin{proof}
	We use the cofibre sequence of \cref{par:TC-BattlePlan}. To compute $\TCref(\ku\otimes\IQ/\ku)$, we must study the cofibres of $\TC^-((\ku/m)/\ku)^\vee\rightarrow \ku^{\h S^1}$ for high-powered integers $m\in\IN^{\lightning}\!$. Put
	\begin{align*}
		\fil^\star\q\ov{\deRham}_{m}&\coloneqq \cofib\left(\IZ[\beta]\llbracket t\rrbracket\rightarrow \fil_{\qHodge}^\star\qhatdeRham_{(\IZ/m)/\IZ}\right)\,,\\
		\ov{\TC}_m^-&\coloneqq \cofib\left(\ku^{\h S^1}\rightarrow \TC^-\bigl((\ku/m)/\ku\bigr)\right)\,.
	\end{align*}
	Since $\ku^{\h S^1}$ and $\TC^-((\ku/m)/\ku)$ are even spectra, the sequence of double speed Whitehead filtrations $\tau_{\geqslant 2\star}(\ku^{\h S^1})\rightarrow \tau_{\geqslant 2\star}\TC^-((\ku/m)/\ku)\rightarrow \tau_{\geqslant 2\star}\ov{\TC}_m^-$ is still a cofibre sequence in filtered spectra. Applying the construction from the proof of \cref{lem:ExtSpectralSequence}, we get complete exhaustive filtrations on the duals of $\ku^{\h S^1}$, $\TC^-((\ku/m)/\ku)$, and $\ov{\TC}_m^-$ in such a way that they fit into a cofibre sequence $\fil^\star(\ov{\TC}_m^-)^\vee\rightarrow \fil^\star\TC^-((\ku/m)/\ku)^\vee\rightarrow \fil^\star(\ku^{\h S^1})^\vee $. After passing to associated gradeds, we get a cofibre sequence of graded $\Sigma^{2*}\IZ[\beta]\llbracket t\rrbracket$-modules
	\begin{equation*}
		\gr^*(\ov{\TC}_m^-)^\vee\longrightarrow \Sigma^{2*}\bigl(\fil_{\qHodge}^\star\qhatdeRham_{(\IZ/m)/\IZ}\bigr)^\vee\longrightarrow \Sigma^{2*}\IZ[\beta]\llbracket t\rrbracket^\vee\,,
	\end{equation*}
	where $\Sigma^{2*}\colon \Gr(\Sp)\rightarrow \Gr(\Sp)$ denotes the \enquote{double shearing} functor. It's clear from the construction that the morphism on the right must really be given by $\Sigma^{2*}(-)^\vee$ applied to the unit map $\IZ[\beta]\llbracket t\rrbracket\rightarrow \fil_{\qHodge}^\star\qhatdeRham_{(\IZ/m)/\IZ}$. It follows that $\gr^*(\ov{\TC}_m^-)^\vee\simeq \Sigma^{2*}(\fil^\star\q\ov{\deRham}_m)^\vee$. Observe that $(\fil^\star\q\ov{\deRham}_m^*)^\vee$ sits in homological degree~$-1$. Indeed, this can be checked modulo $(\beta,t)$. Then $\fil^\star\q\ov{\deRham}_m/(\beta,t)\simeq \cofib(\IZ\rightarrow\gr_{\Hodge}^*\deRham_{(\IZ/m)/\IZ})$ is given by $\Sigma \IZ$ in graded degree~$0$ and $\IZ/m$ in every other graded degree, so it's straightforward to see that its graded dual over $\IZ$ sits indeed in homological degree~$-1$.
	
	Thus, $\fil^\star(\ov{\TC}_m^-)^\vee$ must be the double speed Whitehead filtration, $(\ov{\TC}_m^-)^\vee$ is concentrated in odd degrees, and $\pi_{2*-1}((\ov{\TC}_m^-)^\vee)\cong \H_{-1}(\fil^\star(\q\ov{\deRham}_m)^\vee)$ as a graded $\IZ[\beta]\llbracket t\rrbracket$-modules. Combining this with \cref{cor:Duals}, we see that the long exact homotopy sequence of the rotated cofibre sequence $ (\ku^{\h S^1})^\vee \rightarrow\Sigma (\ov{\TC}_m^-)^\vee\rightarrow\Sigma\TC^-((\ku/m)/\ku)^\vee $ breaks up into a short exact sequence of graded $\IZ[\beta]\llbracket t\rrbracket$-modules of the following form:
	\begin{equation*}
		0\longrightarrow \IZ[\beta]\llbracket t\rrbracket\longrightarrow \H_{-1}\bigl(\fil^\star(\q\ov{\deRham}_m)^\vee\bigr)\longrightarrow \Eext_{\IZ[\beta]\llbracket t\rrbracket}^1\left(\fil_{\qHodge}^\star\qhatdeRham_{(\IZ/m)/\IZ},\IZ[\beta]\llbracket t\rrbracket\right)\longrightarrow 0\,.
	\end{equation*}
	Since $\TCref(\ku\otimes\IQ/\ku)\simeq \indcolim_{m\in \smash{(\IN^\lightning\!)^\op}}\Sigma(\ov\TC_m^-)$ by the cofibre sequence from \cref{par:TC-BattlePlan}, it follows at once that $\TCref(\ku\otimes\IQ/\ku)$ is concentrated in even degrees and that $\A_{\ku}^*$ fits into the desired short exact sequence. Furthermore, it's clear from our considerations above that
	\begin{equation*}
		\bigl(\fil_{\qHodge}^\star\qhatdeRham_{(\IZ/m)/\IZ}\bigr)^\vee\simeq\Sigma^{-1}\Eext_{\IZ[\beta]\llbracket t\rrbracket}\left(\fil_{\qHodge}^\star\qhatdeRham_{(\IZ/m)/\IZ},\IZ[\beta]\llbracket t\rrbracket\right)\longrightarrow \IZ[\beta]\llbracket t\rrbracket\,,
	\end{equation*}
	induced by the short exact sequence, is given by dualising the canonical unit morphism $\IZ[\beta]\llbracket t\rrbracket\rightarrow \fil_{\qHodge}^\star\qhatdeRham_{(\IZ/m)/\IZ}$. Then the underlying graded ind-$\IZ[\beta]\llbracket t\rrbracket$-module of $\A_{\ku}^*$ must really be given by killing the pro-idempotent $\prolim_{m\in \smash{\IN^\lightning\!}}\fil_{\qHodge}^\star\qhatdeRham_{(\IZ/m)/\IZ}$. Idempotence and nuclearity of $\A_{\ku}^*$ follow from \cref{lem:NuclearIdempotentAbstract}\cref{enum:NuclearIdempotentAbstract} and \cref{cor:ProTraceClass}. Since idempotents admit a unique $\IE_\infty$-algebra structure, it follows that the desired description of $\A_{\ku}^*$ also holds as a nuclear ind-$\IZ[\beta]\llbracket t\rrbracket$-algebra. This finishes the proof of \cref{enum:RefinedTC-ku}.
	
	The proof of \cref{enum:RefinedTC-KU} is analogous; the only difference is that we need a different argument why $\cofib(\IZ\qpower\rightarrow \qHodge_{(\IZ/m)/\IZ})^\vee$ is concentrated in homological degree~$-1$. This can be checked modulo $(q-1)$. We've seen in the proof of \cref{cor:ProIdempotent} that $\qHodge_{(\IZ/m)/\IZ}/(q-1)$ is a $\IZ/m$-algebra and, abstractly, $\qHodge_{(\IZ/m)/\IZ}/(q-1)\simeq \bigoplus_\IN\IZ/m$. We can choose this decomposition in such a way that one of those summands corresponds to the unit $\IZ/m\rightarrow \qHodge_{(\IZ/m)/\IZ}/(q-1)$. It follows that 
	\begin{equation*}
		\cofib\bigl(\IZ\rightarrow \qHodge_{(\IZ/m)/\IZ}/(q-1)\bigr)^\vee\simeq \biggl(\Sigma\IZ\oplus\bigoplus_{\IN\smallsetminus\{1\}}\IZ/m\biggr)^\vee\simeq \Sigma^{-1}\IZ\oplus\Sigma^{-1}\prod_{\IN\smallsetminus\{1\}}\IZ/m
	\end{equation*}
	is indeed concentrated in homological degree~$-1$ and we're done.
	\end{proof}
	
	\subsection{Explicit \texorpdfstring{$q$}{q}-Hodge filtrations}\label{subsec:ElementaryProof}
	
	In this subsection, we'll give an explicit description of the $q$-Hodge filtration $\fil_{\qHodge}^\star\qdeRham_{(\IZ/m)/\IZ}$. This will be used in \cref{sec:Overconvergent} to prove \cref{thm:OverconvergentNeighbourhood,thm:OverconvergentNeighbourhoodEquivariant}.
	
	By construction, it will be enough to describe the $q$-Hodge filtration in the case where $m=p^\alpha$ is a prime power. In this case, the filtration is obtained via base change from $\qdeRham_{(\IZ_p\{x\}_\infty/x^\alpha)/\IZ_p}$. Using $\qdeRham_{(\IZ_p\{x\}_\infty/x^\alpha)/\IZ_p}\simeq \qdeRham_{(\IZ_p\{x\}_\infty/x^\alpha)/\IZ_p\{x\}_\infty}$ and base change, we can further reduce the problem to describing the filtration from \cite[Construction~\chref{4.21}]{qWittHabiro} on the derived $q$-de Rham complex 
	\begin{equation*}
	\qdeRham_{(\IZ_p\{x\}/x^\alpha)/\IZ_p\{x\}}\simeq \IZ_p\{x\}\qpower\left\{\frac{\phi(x^\alpha)}{[p]_q}\right\}_{(p,q-1)}^\complete\,.
	\end{equation*}
	Let us denote this ring by $\q D_\alpha$ for short and let $D_\alpha\coloneqq \deRham_{(\IZ_p\{x\}/x^\alpha)/\IZ_p\{x\}}$. Then  $D_\alpha$ is the $p$-completed PD-envelope of $(x^\alpha)\subseteq \IZ_p\{x\}$ and $\q D_\alpha/(q-1)\simeq D_\alpha$. The filtration $\fil_{\qHodge}^\star\q D_\alpha$ from \cite[Construction~\chref{4.21}]{qWittHabiro} is, by definition, given as the ($1$-categorical) preimage of the $(x^\alpha,q-1)$-adic filtration on $D_\alpha[1/p]_{\Hodge}^\complete\qpower$.
	

	\begin{numpar}[Lifts of divided powers.]\label{par:LiftsOfDividedPowers}
	Let $\gamma(-)\coloneqq (-)^p/p$ denote the divided power operation and let $\gamma^{(n)}(-)$ denote its $n$-fold iteration. To get an explicit description of the filtration $\fil_{\qHodge}^\star\q D_\alpha$, our goal is to find elements $\widetilde{\gamma}_q^{(n)}(x^\alpha)\in\fil_{\qHodge}^{p^n}\q D_\alpha$ for all $n\geqslant 0$ such that the following two conditions hold:
	\begin{alphanumerate}\itshape
		\item We have $\widetilde{\gamma}_q^{(n)}(x^\alpha)\equiv \gamma^{(n)}(x^\alpha)\mod (q-1)$.
		\item The image of $\widetilde{\gamma}_q^{(n)}(x^\alpha)$ in $D_\alpha[1/p]_{\Hodge}^\complete\qpower$ is contained in the ideal $(x^\alpha,q-1)^{p^n}$.
	\end{alphanumerate}
	Indeed, if we believe that $\fil_{\qHodge}^\star \q D_\alpha/(q-1)\cong \fil_{\Hodge}^\star D_\alpha$, then such elements must exist. Conversely, if such elements exist, then $\fil_{\qHodge}^\star \q D_\alpha/(q-1)\rightarrow \fil_{\Hodge}^\star D_\alpha$ must be surjective and thus an isomorphism by \cite[Lemma~\chref{4.26}]{qWittHabiro}. So $\fil_{\qHodge}^\star\q D_\alpha$ must be generated as a $(p,q-1)$-complete filtered $\q D_\alpha$-algebra by $(q-1)$ in filtration degree~$1$ and the elements $\widetilde{\gamma}_q^{(n)}(x^\alpha)$ in filtration degree~$p^n$ for all $n\geqslant 0$.
	\end{numpar}
	
	The following technical lemma shows existence of these lifts along with some structural information about them, and we'll even see an explicit recursive construction in the proof. Moreover, all of this works for all $\alpha\geqslant 2$ without any restrictions in the case $p=2$.
	
	\begin{lem}\label{lem:StructuralResultqHodge}
	For all primes $p$, there is a sequence $(\Gamma_n)_{n\geqslant 0}$ of polynomials in \(\IZ_p\{x\}[q]\) with the following properties:
	
	\begin{alphanumerate}
		
		\item 
		\label{item:sam:congr}
		\(\Gamma_n \equiv x^{p^n}\mod (q-1)^{p-1}\) and $\Gamma_n\in (x^p,(q-1)^{p-1})^{p^{n-1}}\).
		
		\item 
		\label{item:sam:ideal}
		\(\Gamma_n \in ((\phi^i(x),\Phi_{p^i}(q))^p, \Phi_{p^i}(q)^{p-1})^{p^{n-1-i}}\) for all \(1\leqslant i\leqslant n-1\).
		
		\item 
		\label{item:sam:bideal}
		\(\Gamma_n \in (\phi^{n}(x), \Phi_{p^n}(q))\).
		
		\item 
		\label{item:sam:prod}
		\((\Gamma_n)^{\alpha} \in \prod_{i=1}^n \Phi_{p^i}(q)^{p^{n-i}}\cdot\q D_\alpha\) for all $\alpha\geqslant 2$.
		
	\end{alphanumerate}
	In particular, for all $\alpha\geqslant 2$, $(\Gamma_n)^\alpha$ is contained in the ideal $(x^\alpha,q-1)^{p^n}$, and
	\begin{equation*}
		\widetilde{\gamma}_q^{(n)}(x^\alpha)\coloneqq\frac{(\Gamma_n)^\alpha}{\prod_{i=1}^n \Phi_{p^i}(q)^{p^{n-i}}}\in \fil_{\qHodge}^{p^n}\q D_\alpha
	\end{equation*}
	is a lift of the $n$-fold iterated divided power $\gamma^{(n)}(x^\alpha)$ and contained in the $(p^n)$\textsuperscript{th} step of the $q$-Hodge filtration on $\q D_\alpha$.
	\end{lem}
	
	\begin{proof}
	We'll do a proof by induction. For the base case of the induction, \(n=0\), let \(\Gamma_0 := x\). All of the statements are trivial in this case.
	
	For the induction step, we first want to construct the element \(\Gamma_{n}\). For this, let \(P_n, Q_n\) be some polynomials in \(\mathbb{Z}[q]\) such that \(p = P_n(q)(q-1)^{(p-1)p^{n-1}}+Q_n(q)\Phi_{p^n}(q)\). Note that such polynomials always exist, since \(\Phi_{p^n}(1) = p\) and $\Phi_{p^n}(q) \equiv (q-1)^{(p-1)p^{n-1}} \mod p$, so 
	\begin{equation*}
		\frac{\Phi_{p^n}(q) - (q-1)^{(p-1)p^{n-1}}}{p}
	\end{equation*}
	is a unit modulo \((q-1)^{(p-1)p^{n-1}}\). Now define 
	\begin{equation*}
		\label{def:sam:1}
		\Gamma_{n}:=(\Gamma_{n-1})^p+P_n(q)(q-1)^{(p-1)p^{n-1}}\delta(\Gamma_{n-1})=\phi(\Gamma_{n-1})-Q_n(q)\Phi_{p^n}(q)\delta(\Gamma_{n-1}).
	\end{equation*}
	Statement \cref{item:sam:congr} follows trivially. For \cref{item:sam:ideal} and~\cref{item:sam:bideal}, by \cref{lem:pTorsionFree} below it's enough to check that $p\cdot \Gamma_n$ is contained in these ideals. We have
	\begin{align*}
		p \cdot \Gamma_n &= p \cdot (\Gamma_{n-1})^p+P_n(q)(q-1)^{(p-1)p^{n-1}}\bigl(\phi(\Gamma_{n-1}) - (\Gamma_{n-1})^p\bigr) \\
		&=p \cdot \phi(\Gamma_{n-1})-Q_n(q)\Phi_{p^n}(q)\bigl(\phi(\Gamma_{n-1}) - (\Gamma_{n-1})^p\bigr).
	\end{align*}
	Now $(\Gamma_{n-1})^p$ and $\phi(\Gamma_{n-1})$ are contained in each one of the ideals from~\cref{item:sam:ideal}. Indeed, for $(\Gamma_{n-1})^p$, this follows from statements~\cref{item:sam:ideal} and~\cref{item:sam:bideal} of the induction hypothesis, and for $\phi(\Gamma_{n-1})$ this follows similarly from~\cref{item:sam:congr} and~\cref{item:sam:ideal}. Therefore, the first of the two equations above shows that $p\cdot \Gamma_n$ is contained in each of the ideals from~\cref{item:sam:ideal}. Similarly, using statement~\cref{item:sam:bideal} of the induction hypothesis, we get $\phi(\Gamma_{n-1})\in (\phi^n(x),\Phi_{p^n}(q))$ and so the second of the equations above shows that $p\cdot \Gamma_n$ is contained in this ideal as well. This finishes the induction step for~\cref{item:sam:ideal} and~\cref{item:sam:bideal}.
	
	It remains to show statement~\cref{item:sam:prod}. By \cite[Lemma~\chref{16.10}]{Prismatic}, $\q D_\alpha$ is $(p,q-1)$-completely flat over $\IZ_p\qpower$ and thus flat on the nose over $\IZ[q]$. Therefore
	\begin{equation*}
		\prod_{i=1}^n \Phi_{p^i}(q)^{p^{n-i}}\cdot\q D_\alpha=\bigcap_{i=1}^n \Phi_{p^i}(q)^{p^{n-i}}\cdot\q D_\alpha\,.
	\end{equation*}
	To show that $(\Gamma_n)^\alpha\in \Phi_{p^i}(q)^{p^{n-i}}\cdot\q D_\alpha$ for \(1 \leqslant i \leqslant n-1\), by the already proven statement \cref{item:sam:ideal}, it's enough to show the same for any element in the ideal $((\phi^i(x),\Phi_{p^i}(q))^p, \Phi_{p^i}(q)^{p-1})^{\alpha p^{n-1-i}}$. So consider a monomial of the form 
	\begin{equation*}
		\bigl(\phi^i(x)^j\Phi_{p^i}(q)^k\bigr)^\ell\Phi_{p^i}(q)^{(p-1)m}\,,
	\end{equation*}
	where $j+k=p$ and $\ell+m=\alpha p^{n-1-i}$. By construction, $\phi(x)^\alpha$ becomes divisible by $\Phi_p(q)$ in $\q D_\alpha$ and so $\phi^i(x)^\alpha\in \Phi_{p^i}(q)\cdot \q D_\alpha$. Hence $\phi^i(x)^{j\ell}$ is divisible by $\Phi_{p^i}(q)^{\lfloor j\ell/\alpha\rfloor}$. It will therefore be enough to show
	\begin{equation*}
		\left\lfloor \frac{j\ell}\alpha\right\rfloor+k\ell+(p-1)m\geqslant p^{n-i}\,.
	\end{equation*}
	This is straightforward: For $\ell=0$, the inequality follows from $\alpha(p-1)\geqslant p$ as $\alpha\geqslant 2$. In general, if we replace $(j,k)$ by $(j-1,k+1)$, the left-hand side changes by at least $\ell-\lfloor \ell/\alpha\rfloor -1$; for $\ell\geqslant 1$ and $\alpha\geqslant 2$ this term is always nonnegative. Therefore we may assume $j=p$, $k=0$, and we must show $\lfloor p\ell/\alpha\rfloor +(p-1)m\geqslant p^{n-i}$. If $p=2$ and $\alpha=2$, this becomes the equality $\ell+m=2^{n-i}$ and so the inequality is sharp in this case. If $p\geqslant 3$ or $\alpha\geqslant 3$, we have $(p-1)-\lfloor p/\alpha\rfloor-1\geqslant 0$ and so by the same argument as before we may assume $\ell=\alpha p^{n-1-i}$, $m=0$. The the desired inequality follows from $\alpha(p-1)\geqslant p$ again.
	
	A similar but easier argument shows that every element in $(\phi^n(x),\Phi_{p^n}(q))^\alpha$ becomes divisible by $\Phi_{p^n}(q)$ in $\q D_\alpha$ and we have an inclusion of ideals $(x^p,(q-1)^{p-1})^{\alpha p^{n-1}}\subseteq (x^\alpha,q-1)^{p^n}$ in $\IZ_p\{x\}[q]$. This finishes the proof of \cref{item:sam:prod} and shows $(\Gamma_n)^\alpha\in (x^\alpha,q-1)^{p^n}$. Hence $\widetilde{\gamma}_q^{(n)}(x^\alpha)$ is really contained in the $(p^n)$\textsuperscript{th} step of the $q$-Hodge filtration and it lifts $\gamma^{(n)}(x^\alpha)$ by \cref{item:sam:congr}.
	\end{proof}
	\begin{lem}\label{lem:pTorsionFree}
	If $J\subseteq \IZ_p\{x\}[q]$ is any of the ideals in \cref{lem:StructuralResultqHodge}\cref{item:sam:ideal} or~\cref{item:sam:bideal}, then $\IZ_p\{x\}[q]/J$ is $p$-torsion free.
	\end{lem}
	\begin{proof}
	Consider the map $\psi_i\colon \IZ_p\{x\}[q]\rightarrow \IZ_p\{x\}[q]$ given by the $i$-fold iterated Frobenius $\phi^i\colon \IZ_p\{x\}\rightarrow \IZ_p\{x\}$ and $q\mapsto \Phi_{p^i}(q)$. If we replace $\phi^i(x)$ and $\Phi_{p^i}(q)$ in the definition of $J$ by $x$ and $q$, respectively, we obtain an ideal $J_0\subseteq \IZ_p\{x\}[q]$ such that
	\begin{equation*}
		\IZ_p\{x\}/J\cong \IZ_p\{x\}/J_0\otimes_{\IZ_p\{x\}[q],\psi_i}\IZ_p\{x\}[q]\,.
	\end{equation*}
	Now $\phi^i$ is flat by \cite[Lemma~\chref{2.11}]{Prismatic} and $q\mapsto \Phi_{p^i}(q)$ is finite free, as the polynomial $\Phi_{p^i}(q)$ is monic. So $\psi_i$ is flat and it suffices to show that $\IZ_p\{x\}[q]/J_0$ is $p$-torsion free. But $\IZ_p\{x\}[q]$ is a free module over $\IZ_p$ with basis given by monomials in $x,\delta(x),\delta^2(x),\dotsc$ and $q$. By construction, $J_0$ is a free submodule on a subset of that basis. It follows that $\IZ_p\{x\}[q]/J_0$ is free over $\IZ_p$, hence $p$-torsion free.
	\end{proof}

	\newpage
	
	\section{Algebras of overconvergent functions}\label{sec:Overconvergent}
	In this section we prove \cref{thm:OverconvergentNeighbourhood,thm:OverconvergentNeighbourhoodEquivariant}. In \crefrange{subsec:OverconvergentGeneral}{subsec:GradedAdic} we'll review Clausen's and Scholze's approach to adic spaces via solid analytic rings \cite[Lecture~\href{https://youtu.be/YLQt_tV4tHo?list=PLx5f8IelFRgGmu6gmL-Kf_Rl_6Mm7juZO}{10}]{AnalyticStacks} and study algebras of overconvergent functions as well as gradings in this setup. In \cref{subsec:MainProof}, we'll then combine this with our explicit computation of the $q$-Hodge filtration on $\qdeRham_{(\IZ/p^\alpha)/\IZ_p}$ from \cref{subsec:ElementaryProof} to finish the proof of \cref{thm:OverconvergentNeighbourhood,thm:OverconvergentNeighbourhoodEquivariant}.
	
	\subsection{Adic spaces as analytic stacks}\label{subsec:OverconvergentGeneral}
	In the following, we'll use the formalism of analytic stacks from \cite{AnalyticStacks}. For the convenience of the reader, let us briefly recall the relevant notions.
	
	\begin{numpar}[Solid condensed spectra.]\label{par:CondensedRecollections}
		Let $\Cond(\Sp)$ denote the $\infty$-category of \emph{\embrace{light} condensed spectra}, that is, hypersheaves of spectra on the site of light profinite sets as defined by Clausen and Scholze \cite{AnalyticStacks}. The evaluation at the point $(-)(*)\colon \Cond(\Sp)\rightarrow \Sp$ admits a fully faithful symmetric monoidal left adjoint $(\underline{-})\colon \Sp\rightarrow\Cond(\Sp)$, sending a spectrum $X$ to the \emph{discrete} condensed spectrum $\underline{X}$.
		
		One can develop a theory of \emph{solid condensed spectra} along the lines of \cite[Lectures~\href{https://www.youtube.com/watch?v=bdQ-_CZ5tl8&list=PLx5f8IelFRgGmu6gmL-Kf_Rl_6Mm7juZO}{5}--\href{https://www.youtube.com/watch?v=KKzt6C9ggWA&list=PLx5f8IelFRgGmu6gmL-Kf_Rl_6Mm7juZO}{6}]{AnalyticStacks}. Let $\Null\coloneqq \cofib(\IS[\{\infty\}]\rightarrow\IS[\IN\cup\{\infty\}])$ be the free condensed spectrum on a null sequence. Let $\sigma\colon \Null\rightarrow \Null$ be the endomorphism induced by the shift map $(-)+1\colon \IN\cup\{\infty\}\rightarrow \IN\cup\{\infty\}$. Recall that a condensed spectrum $M$ is called \emph{solid} if
		\begin{equation*}
			1-\sigma^*\colon \Hhom_\IS(\Null,M)\overset{\simeq}{\longrightarrow}\Hhom_\IS(\Null,M)
		\end{equation*}
		is an equivalence, where $\Hhom_\IS$ denotes the internal Hom in $\Cond(\Sp)$. We let $\Sp_\solid\subseteq \Cond(\Sp)$ denote the full sub-$\infty$-category of solid condensed spectra. Then $\Sp_\solid$ is closed under all limits and colimits. This implies that the inclusion $\Sp_\solid\subseteq \Cond(\Sp)$ admits a left adjoint $(-)^\solid\colon \Cond(\Sp)\rightarrow \Sp_\solid$. It satisfies $(M\otimes N)^\solid\simeq (M^\solid\otimes N)^\solid$, which allows us to endow $\Sp_\solid$ with a symmetric monoidal structure, called the \emph{solid tensor product}, via $M\soltimes N\coloneqq (M\otimes N)^\solid$. This allows us to define the derived $\infty$-category of solid abelian groups as $\Dd(\IZ_\solid)\coloneqq\Mod_\IZ(\Sp_\solid)$.
	\end{numpar}
	\begin{numpar}[Huber pairs à la Clausen--Scholze.]\label{par:AdicRecollectionsI}
	Recall that to any Huber pair $(R,R^+)$ one can associate an \emph{analytic ring} $(R,R^+)_\solid$ in the sense of \cite[Lecture~\href{https://youtu.be/YxSZ1mTIpaA?list=PLx5f8IelFRgGmu6gmL-Kf_Rl_6Mm7juZO&t=3964}{1}]{AnalyticStacks} as follows: First consider $R$ as a condensed ring via its given topology. For $f\in R(*)$ and $M\in\Mod_R(\Dd(\IZ_\solid))$ we say that $M$ is \emph{$f$-solid} if 
	\begin{equation*}
		1-f\sigma^*\colon\RHhom_\IZ(\Null_\IZ,M)\overset{\simeq}{\longrightarrow}\RHhom_\IZ(\Null_\IZ,M)
	\end{equation*}
	is an equivalence. Here $\Null_\IZ\coloneqq \Null\otimes \IZ\simeq \cofib(\IZ[\{\infty\}]\rightarrow\IZ[\IN\cup\{\infty\}])$ denotes the free condensed abelian group on a nullsequence. The inclusion of the full sub-$\infty$-category of $f$-solid $R$-modules admits a left adjoint $(-)^{f\solid}$, called \emph{$f$-solidification}. The underlying animated condensed ring of $(R,R^+)_\solid$ is then defined as
	\begin{equation*}
		(R,R^+)_\solid^\triangleright\coloneqq \colimit_{\{f_1,\dotsc,f_r\}\subseteq R^+}R^{f_1\solid,\dotsc,f_r\solid}\,,
	\end{equation*}
	where the colimit is taken over all finite subsets of $R^+$, and $\Dd((R,R^+)_\solid)\subseteq \Mod_R(\Dd(\IZ_\solid))$ is the full sub-$\infty$-category of solid condensed $R$-modules that are $f$-solid for all $f\in R^+\subseteq R(*)$. In the following, we'll always work with Huber pairs for which $(R,R^+)_\solid^\triangleright$ is just $R$ itself.
	
	The classical notion of \emph{affinoid open subsets} fits naturally into this formalism. Suppose we're given $f_1,\dotsc,f_r\in R(*)$ generating an open ideal as well as another element $g\in R(*)$, so that $U\coloneqq \left\{x\in\Spa(R,R^+)\ \middle|\ \abs{f_1}_x,\dotsc,\abs{f_r}_x\leqslant \abs{g}_x\neq 0\right\}$ defines a rational open subset. We can define an analytic ring $\Oo(U_\solid)$ as follows: The underlying animated condensed ring is the solidification
	\begin{equation*}
		\Oo(U)\coloneqq R\bigl[\localise{g}\bigr]^{(f_1/g)\solid,\dotsc,(f_r/g)\solid}
	\end{equation*}
	and we let $\Dd(U_\solid)\coloneqq\Dd(\Oo(U_\solid))\subseteq \Mod_{R[1/g]}(\Dd((R,R^+)_\solid))$ be the full sub-$\infty$-category spanned by those $R[1/g]$-modules in $\Dd((R,R^+)_\solid)$ that are also $(f_i/g)$-solid for $i=1,\dotsc,r$. If $\Oo(U)$ is static and quasi-separated, it agrees with the Huber ring from the classical theory of adic spaces. In practice, this will almost always be the case. 
	\end{numpar}
	\begin{numpar}[Adic spaces à la Clausen--Scholze.]\label{par:AdicRecollectionsII}
	Clausen and Scholze associate to any Tate%
	\footnote{To avoid confusion with analytic stacks, we'll call an adic space \emph{Tate} rather than \emph{analytic} if, locally, there exists a topologically nilpotent unit. The restriction to Tate adic spaces makes sure that open immersions go to open immersions (see \cref{lem:TateAnalyticStacks} below); analytic stacks can be associated to any adic space.}
	adic space $X$ an \emph{analytic stack} $X_\solid\rightarrow \AnSpec \IZ_\solid$. If $X=\Spa(R,R^+)$ is Tate affinoid, we simply put $X_\solid\coloneqq \AnSpec(R,R^+)_\solid$. If $U\subseteq \Spa(R,R^+)$ is an open subset of a Tate affinoid adic space, choose a cover $V\coloneqq\coprod_{i\in I}V_i\rightarrow U$ by rational open subsets and form the \v Cech nerve $V_\bullet\coloneqq \check{\mathrm{C}}_\bullet(V\rightarrow X)$. Every $V_n$ is a disjoint union of affinoid adic spaces, hence $V_{n, \solid}$ is already defined. Then we can put $U_\solid\coloneqq \colimit_{[n]\in\IDelta^\op}V_{n, \solid}$. Finally, if $X$ is an arbitrary Tate adic space, choose a cover $W\coloneqq \coprod_{j\in J}W_j\rightarrow X$ by affinoids and form the \v Cech nerve $W_\bullet\coloneqq \check{\mathrm{C}}_\bullet(W\rightarrow X)$. Each $W_m$ is a disjoint union of open subsets of Tate affinoid adic spaces, so $W_{m, \solid}$ is already defined, and we put $X_\solid\coloneqq \colimit_{[m]\in\IDelta^\op}W_{m, \solid}$.
	
	It can be shown that these constructions are well-defined and independent of the choices involved. We'll omit the verification, but let us at least mention the crucial input.
	\end{numpar}
	\begin{lem}\label{lem:TateAnalyticStacks}
	Let $(R,R^+)$ be a Huber pair and let $X_\solid\coloneqq \AnSpec (R,R^+)_\solid$ be the associated affine analytic stack.
	\begin{alphanumerate}
		\item If $U,U'\subseteq \Spa(R,R^+)$ are rational open subsets, then\label{enum:TateIntersection}
		\begin{equation*}
			\AnSpec\Oo(U_\solid)\times_{\AnSpec(R,R^+)_\solid}\AnSpec\Oo(U'_\solid)\simeq \AnSpec \Oo\bigl((U\cap U')_\solid\bigr)\,.
		\end{equation*}
		\item If $R$ is Tate and $U\subseteq \Spa(R,R^+)$ is a rational open subset, then $j\colon U_\solid\rightarrow X_\solid$ is an open immersion of affine analytic stacks in the sense of \cite[Lecture~\href{https://youtu.be/BV0-dlAuS3U?list=PLx5f8IelFRgGmu6gmL-Kf_Rl_6Mm7juZO&t=3503}{\textup{16}}]{AnalyticStacks}. That is, $j^*$ admits a fully faithful left adjoint $j_!$ satisfying the projection formula.\label{enum:TateOpenImmersion}
		\item If $R$ is Tate and $\coprod_{i=1}^nU_i\rightarrow\Spa(R,R^+)$ is a cover by rational open subsets, then $\coprod_{i=1}^nU_{i, \solid}\rightarrow X_\solid$ is a $!$-cover of affine analytic stacks.\label{enum:TateCover}
	\end{alphanumerate}
	\end{lem}
	\begin{rem}
	The Tate condition in \cref{lem:TateAnalyticStacks}\cref{enum:TateOpenImmersion} and~\cref{enum:TateCover} is crucial and it is the reason why we restrict to the Tate case when we describe adic spaces in terms of analytic stacks. Without this assumption, \cref{enum:TateOpenImmersion} will be wrong. For example, if $R$ is a discrete ring, any Zariski-open also determines a rational open of $\Spa(R,R)$, but in this case $j^*$ almost never preserves limits, so it can't have a left adjoint $j_!$.
	\end{rem}
	\begin{proof}[Proof sketch of \cref{lem:TateAnalyticStacks}]
	Suppose $U$ and $U'$ are given by $\abs{f_1},\dotsc,\abs{f_r}\leqslant\abs{g}\neq 0$ and $\abs{f_1'},\dotsc,\abs{f_s'}\leqslant\abs{g'}\neq 0$, respectively.
	Using the description of pushouts from \cite[Lecture~\href{https://youtu.be/38PzTzCiMow?list=PLx5f8IelFRgGmu6gmL-Kf_Rl_6Mm7juZO&t=2741}{11}]{AnalyticStacks}, it's clear that $\Oo(U_\solid)\lotimes_{(R,R^+)_\solid}\Oo(U'_\solid)$ is the solidification of $R[1/(gg')]$ at the elements $f_i/g$ and $f_j'/g'$ for $i=1,\dotsc,r$, $j=1,\dotsc,s$. But that's precisely $\Oo((U\cap U')_\solid)$, proving \cref{enum:TateIntersection}.

	For \cref{enum:TateOpenImmersion}, assume $U$ is given by $\abs{f_1},\dotsc,\abs{f_r}\leqslant \abs{g}\neq 0$. Since $R$ is assumed to be Tate, the open ideal generated by $f_1,\dotsc,f_r$ must be all of $R$. Hence $g$ will aready be invertible in $R[T_1,\dotsc,T_r]/\left(gT_i-f_i\ \middle|\ i=1,\dotsc,r\right)$ and this quotient is automatically a derived quotient as well. It follows that the functor $j^*\colon \Dd(X_\solid)\rightarrow \Dd(U_\solid)$ can also be written as
	\begin{equation*}
		(-)[T_1,\dotsc,T_r]^{T_1\solid,\dotsc,T_r\solid}/\left(gT_i-f_i\ \middle|\ i=1,\dotsc,r\right)\,.
	\end{equation*}
	By \cite[Lecture~\href{https://youtu.be/fUjn2rGw9SA?list=PLx5f8IelFRgGmu6gmL-Kf_Rl_6Mm7juZO&t=2056}{7}]{AnalyticStacks}, the functor $(-)[T]^{T\solid}$ of adjoining a variable and then solidifiying it can be explicitly described as $\RHhom_{\IZ}(\IZ(\!(T^{-1})\!)/\IZ[T],-)$ and so $j^*(-)\simeq \RHhom_R(Q,-)$, where
	\begin{equation*}
		Q\coloneqq \biggl(\bigotimes_{i=1}^r\Sigma^{-1}\IZ(\!(T_i^{-1})\!)/\IZ[T_i]\lsoltimes_\IZ R\biggr)/\left(gT_i-f_i\ \middle|\ i=1,\dotsc,r\right)\,.
	\end{equation*}
	It follows immediately that $j^*$ admits a left adjoint $j_!(-)\simeq Q\lotimes_{(R,R^+)_\solid}-$. It remains to check the projection formula 
	\begin{equation*}
		j_!(M)\lotimes_{(R,R^+)_\solid}N\simeq j_!\bigl(M\lotimes_{\Oo(U_\solid)}j^*(N)\bigr)\,.
	\end{equation*}
	By the same argument as above, $Q$ is already an $R[1/g]$-module and the functor $j^*$ is insensitive to inverting $g$. Therefore, it's enough to check the projection formula in the case where $N$ is an $R[1/g]$-module. When restricting to $R[1/g]$-modules, $j^*$ is just given by successively killing the idempotent algebras $\IZ(\!(T_i^{-1})\!)\lsoltimes_{\IZ[T_i],  T_i\mapsto f_i/g}R[1/g]$ for $i=1,\dotsc,r$. Now for killing an idempotent it's completely formal to see that the left adjoint indeed satisfies the projection formula. This finishes the proof of \cref{enum:TateOpenImmersion}.
	
	To show \cref{enum:TateCover}, since we already know that each $j_i\colon U_{i, \solid}\rightarrow X_\solid$ is an open immersion, we can use the criterion from \cite[Lecture~\href{https://youtu.be/vRUmXU8ijIk?list=PLx5f8IelFRgGmu6gmL-Kf_Rl_6Mm7juZO&t=4221}{18}]{AnalyticStacks} to verify that $\coprod_{i=1}^nU_{i, \solid}\rightarrow X_\solid$ is indeed a $!$-cover. That is, if $A_i\coloneqq \cofib(j_{i,!}\Oo(U_i)\rightarrow R)$, we need to show $A_1\lotimes_{(R,R^+)_\solid}\dotsb\lotimes_{(R,R^+)_\solid}A_n\simeq 0$. Using \cite[Lemma~{\chpageref[2.6]{19}}]{HuberGeneralizationFormalRigid} and an inductive argument as in 
	\cite[Lemma~\chref{10.3}]{Condensed}, this can be reduced to the special case where $n=2$ and $U_1=\left\{x\in X\ \middle|\ 1\leqslant \abs{f}_x\right\}$, $U_2=\left\{x\in X\ \middle|\ \abs{f}_x\leqslant 1\right\}$ 
	for some $f\in R$. This is now a straightforward calculation.
	\end{proof}
	\begin{rem}
	Let $U\subseteq X$ be an open inclusion of Tate adic spaces and let $j\colon U_\solid\rightarrow X_\solid$ be the corresponding map of analytic stacks. In the following, if its clear that we're working in $\Dd(X_\solid)$, we often abuse notation and write $\Oo_U$ instead of $j_*\Oo_{U_\solid}$ for the pushforward of the structure sheaf of $U_\solid$. We also use $-\lotimes_{\Oo_{X_\solid}}\Oo_{U_\solid}$ to denote the functor $j_*j^*\colon \Dd(X_\solid)\rightarrow \Dd(X_\solid)$.
	
	Let us point out that $-\lotimes_{\Oo_{X_\solid}}\Oo_{U_\solid}$ is \emph{not} just the tensor product with $\Oo_U$ in the symmetric monoidal $\infty$-category $\Dd(X_\solid)$.  We can already see the difference if $X= \Spa(R,R^+)$ and $U\subseteq X$ is a rational open given by $\abs{f_1},\dotsc,\abs{f_r}\leqslant \abs{g}\neq 0$: In this case, 
	\begin{equation*}
		-\lotimes_{\Oo_{X_\solid}}\Oo_{U_\solid}\simeq \bigl(-\lotimes_{\Oo_{X_\solid}}\Oo_{U}\bigr)^{(f_1/g)\solid,\dotsc,(f_r/g)\solid}\,.
	\end{equation*}
	In particular, even though $\Oo_U\lotimes_{\Oo_{X_\solid}}\Oo_{U_\solid}\simeq \Oo_U$ (see \cref{lem:TateAnalyticStacks}\cref{enum:TateIntersection} and \cref{lem:TateOpenImmersions}\cref{enum:TateStacksIntersection} below), it's rarely true that $\Oo_U$ is idempotent in $\Dd(X_\solid)$. 
	%
	%
	%
	\end{rem}
	%
	Thus, there's a priori no reason to expect that sheaves of overconvergent functions $\Oo_{Z^\dagger}$ would be idempotent. In the following, we'll investigate why idempotence is satisfied in the situation of \cref{thm:OverconvergentNeighbourhood,thm:OverconvergentNeighbourhoodEquivariant}. Let's start by introducing a notion of open immersions for analytic stacks that need not be affine.
	
	\begin{numpar}[Open immersions of analytic stacks.]\label{par:NaiveOpenImmersion}
	We call a map of analytic stacks $j\colon U\rightarrow X$ a \emph{naive open immersion} if $j$ is a $!$-able monomorphism and $j^*\simeq j^!$. Since $j$ is a monomorphism, $U\times_XU\simeq U$. Combining this with proper base change, we get $j^*j_!\simeq \id_{\Dd(U)}$ and so $j_!$ is fully faithful. Then the right adjoint $j_*$ of $j^*$ must be fully faithful as well.
	
	Using the projection formula and $j^*j_!\simeq \id_{\Dd(U)}$, we see that $j_!\Oo_U\rightarrow \Oo_X$ exhibits $j_!\Oo_U$ as an idempotent coalgebra in $\Dd(X)$. Then $\cofib(j_!\Oo_U\rightarrow \Oo_X)$ must be an idempotent algebra. In this way, we can associate to any naive open immersion an idempotent algebra in $\Dd(X)$, which we call the \emph{complementary idempotent determined by $U$} and denote $\Oo_{X\smallsetminus U}$. It's straightforward to check that the forgetful functor $i_*\colon \Mod_{\Oo_{X\smallsetminus U}}(\Dd(X))\rightarrow \Dd(X)$, which is fully faithful by idempotence, fits into a recollement
	\begin{equation*}
		\begin{tikzcd}
			\Mod_{\Oo_{X\smallsetminus U}}\bigl(\Dd(X)\bigr)\rar["i_*"] & \Dd(X)\rar["j^*"]\lar[bend right=45,dashed,"i^*"',end anchor=15] \lar[bend left=45,dashed,"i^!",end anchor=-15] & \Dd(U)\lar[bend right=45,dashed,"j_!"'] \lar[bend left=45,dashed,"j_*"]
		\end{tikzcd}
	\end{equation*}
	and so $j_*\Oo_U$ is obtained from $\Oo_X$ by killing the idempotent algebra $\Oo_{X\smallsetminus U}$. As long as it's clear that we're working in $\Dd(X)$, we often abuse notation and just write $\Oo_U$ instead of $j_*\Oo_X$.
	\end{numpar}
	\begin{rem}
	Every open immersion of affine analytic stacks in the sense of \cite[Lecture~\href{https://youtu.be/BV0-dlAuS3U?list=PLx5f8IelFRgGmu6gmL-Kf_Rl_6Mm7juZO&t=3503}{\textup{16}}]{AnalyticStacks} is also a naive open immersion.
	\end{rem}
	\begin{rem}
	If $A\in \Dd(X)$ is an idempotent algebra, we can define an analytic substack $U_A\subseteq X$ by declaring that a map $f\colon Y\rightarrow X$ factors through $U_A$ if and only if $f^*\colon \Dd(X)\rightarrow \Dd(Y)$ factors through the localisation $\Dd(X)/\Mod_A(\Dd(X))$, or equivalently, if and only if $f^*(A)\simeq 0$. However, it's \emph{not} true that the constructions $U\mapsto \Oo_{X\smallsetminus U}$ and $A\mapsto U_A$ are inverses; it's not even clear why $\Dd(U_A)$ would coincide with $\Dd(X)/\Mod_A(\Dd(X))$.
	
	It's not obvious what conditions should be put on $U$ and $A$ to make these constructions mutually inverse (moreover, whatever the condition, it should be satisfied for open immersions of affine analytic stacks). This explains why we call the notion from \cref{par:NaiveOpenImmersion} \emph{naive}: An \emph{honest} open immersion of analytic stacks should be a naive open immersion for which the idempotent algebra $\Oo_{X\smallsetminus U}$ meets the putative condition. In the following, we'll work with the naive notion, since it is enough for our purposes.
	%
	%
	%
	%
	%
	%
	\end{rem}
	
	\begin{lem}\label{lem:TraceClassAnalyticStacks}
	Let $U'\rightarrow U\rightarrow X$ be naive open immersions of analytic stacks. Suppose that $U$ contains the closure of $U'$ in the sense that there exists another naive open immersion $j\colon V\rightarrow X$ such that $U'\times_XV\simeq \emptyset$ and $\Oo_{X\smallsetminus V}\lotimes_{\Oo_X}\Oo_{X\smallsetminus U}\simeq 0$. Then 
	\begin{equation*}
		\Oo_U\lotimes_{\Oo_X}\Oo_{U'}\simeq \Oo_{U'}\,.
	\end{equation*}
	Moreover, the map $\Oo_U\rightarrow \Oo_{U'}$ is trace-class in $\Dd(X)$ and factors through $\Oo_{X\smallsetminus V}$.
	\end{lem}
	\begin{proof}
	The condition $U'\times_XV\simeq \emptyset$ implies that $\Oo_{U'}$ is in the kernel of the pullback functor $j^*\colon\Dd(X)\rightarrow \Dd(V)$ and so $\Oo_{U'}$ is an algebra over the idempotent $A\coloneqq \Oo_{X\smallsetminus V}$ by \cref{par:NaiveOpenImmersion}. We also know that $\Oo_U$ is obtained from $\Oo_X$ by killing the idempotent $B\coloneqq \Oo_{X\smallsetminus U}$. Hence $\Oo_U\simeq \cofib(B^\vee\rightarrow \Oo_X)$. Since $B^\vee$ is a $B$-module, $\Oo_{U'}$ is an $A$-module, and $A\otimes B\simeq 0$, we get $B^\vee\lotimes_{\Oo_X} \Oo_{U'}\simeq 0$, hence indeed $\Oo_U\lotimes_{\Oo_X}\Oo_{U'}\simeq \Oo_{U'}$.
	
	Since the double dual $B^{\vee\vee}$ is still a $B$-module, the same argument shows $\Oo_{U}^\vee\lotimes_{\Oo_X}\Oo_{U'}\simeq \Oo_{U'}$. Hence $\Oo_U\rightarrow \Oo_{U'}$ is trace-class, with classifier given by the unit $\Oo_X\rightarrow \Oo_{U'}$. We've already seen that $\Oo_{U'}$ is an $A$-algebra. The condition $A\otimes B\simeq 0$ also implies $\RHhom_X(B,A)\simeq 0$, since $\RHhom_X(B,A)$ is both an $A$-module and a $B$-module. It follows that $A$ is contained in the image of $j_*\colon \Dd(U)\rightarrow \Dd(X)$ and hence $A$ is an $\Oo_U$-algebra. This shows that $\Oo_U\rightarrow \Oo_{U'}$ factors through $A$.
	\end{proof}
	\begin{lem}\label{lem:TateOpenImmersions}
	Let $X$ be a Tate adic space with associated analytic stack $X_\solid\rightarrow \AnSpec \IZ_\solid$, and let $U,U'\subseteq X$ be open subsets.
	\begin{alphanumerate}
		\item The map $j\colon U_\solid\rightarrow X_\solid$ is a naive open immersion of analytic stacks. Moreover, an arbitary map $f\colon Y\rightarrow X_\solid$ of analytic stacks factors through $U_\solid$ if and only if $f^*(\Oo_{X\smallsetminus U})\simeq 0$.\label{enum:TateStacksOpenImmersion}
		\item We have $U_\solid\times_{X_\solid}U'_\solid\simeq (U\cap U')_\solid$. In particular, $\Oo_U\lotimes_{\Oo_{X_\solid}}\Oo_{U'_\solid}\simeq \Oo_{U\cap U'}$ and vice versa if $U$ and $U'$ are exchanged.\label{enum:TateStacksIntersection}
		\item If $\ov{U}'\subseteq U$, then $U_\solid$ contains the closure of $U'_\solid$ in the sense of \cref{lem:TraceClassAnalyticStacks}.\label{enum:TateStacksTraceClass}
	\end{alphanumerate}
	\end{lem}
	\begin{proof}[Proof sketch]
	Let's start with \cref{enum:TateStacksIntersection}. In the case where $U$ and $U'$ are affinoid, $U_\solid\times_{X_\solid}U'_\solid\simeq (U\cap U')_\solid$ follows essentially by the construction of $X_\solid$ in \cref{par:AdicRecollectionsII}, because we can choose both $U$ and $U'$ to be part of an affinoid cover of $X$ (and to prove that said construction is independent of the choice of cover, we need \cref{lem:TateAnalyticStacks}\cref{enum:TateIntersection}). To show the general case, just cover $U$ and $U'$ by affinoid open subsets.
	
	Let's show \cref{enum:TateStacksOpenImmersion} next. Let's first consider the case where $X=\Spa(R,R^+)$ is affinoid and $U\subseteq X$ is a rational open. We've already seen in \cref{lem:TateAnalyticStacks}\cref{enum:TateOpenImmersion} that $j\colon U_\solid\rightarrow X_\solid$ is a naive open immersion. Suppose $f\colon Y\rightarrow X_\solid$ is a map of analytic stacks such that $f^*(\Oo_{X\smallsetminus U})\simeq 0$. If $Y\simeq \AnSpec S$ is affine, then the map of analytic rings $(R,R^+)_\solid\rightarrow S$ factors through $\Oo(U_\solid)$ if and only if $f^*\colon \Dd((R,R^+)_\solid)\rightarrow \Dd(S)$ factors through $\Dd(U_\solid)$. Since $f^*(\Oo_{X\smallsetminus U})\simeq 0$, this is satisfied in our case. This proves the claim in the case where $Y\simeq \AnSpec S$ is affine. In particular, $U_\solid\times_{X_\solid}\AnSpec S\simeq \AnSpec S$. For the general case, write $Y$ as a colimit of affines to see $U_\solid\times_{X_\solid}Y\simeq Y$. Then $f\colon Y\rightarrow X_\solid$ clearly factors through $U_\solid$.
	
	Now let $U$ and $X$ be arbitrary. Proving that $j\colon U_\solid\rightarrow X_\solid$ is a naive open immersion formally reduces to the special case considered above; we omit the argument. Now let $f\colon Y\rightarrow X_\solid$ be a map of analytic stacks such that $f^*(\Oo_{X\smallsetminus U})\simeq 0$. Whether $f$ factors through $U_\solid$ can be checked locally on $X_\solid$. By \cref{enum:TateStacksIntersection}, if $\Spa(R,R^+)\rightarrow X$ is an affinoid open supset, then $U_\solid\times_{X_\solid}\AnSpec(R,R^+)_\solid\simeq (U\cap \Spa(R,R^+))_\solid$, so we can reduce to the case where $X$ is affinoid. As above, we may also assume that $Y\simeq \AnSpec S$ is affine. Let $\coprod_{i\in I}U_i\rightarrow U$ be a cover by rational open subsets. Then 
	\begin{equation*}
		\Oo_{X\smallsetminus U}\simeq \colimit_{\{i_1,\dotsc,i_n\}\subseteq I}\left(\Oo_{X\smallsetminus U_{i_1}}\lotimes_{\Oo_{X_\solid}}\dotsb\lotimes_{\Oo_{X_\solid}}\Oo_{X\smallsetminus U_{i_n}}\right)\,,
	\end{equation*}
	where the colimit is taken over all finite subsets of $I$. Since the colimit is filtered and $f^*(\Oo_{X\smallsetminus U})$ is detected by the single condition $1=0$, there exists a finite subset $\{i_1,\dotsc,i_n\}\subseteq I$ such that already $f^*(\Oo_{X\smallsetminus U_{i_1}})\lotimes_{S}\dotsb\lotimes_{S}f^*(\Oo_{X\smallsetminus U_{i_n}})\simeq 0$ in $\Dd(S)$. By the criterion from \cite[Lecture~\href{https://youtu.be/vRUmXU8ijIk?list=PLx5f8IelFRgGmu6gmL-Kf_Rl_6Mm7juZO&t=4221}{18}]{AnalyticStacks}, it follows that $\coprod_{j=1}^nU_{i_j, \solid}\times_{X_\solid}\AnSpec S\rightarrow \AnSpec S$ is a $!$-cover. We may therefore replace $S$ by the constituents of this cover, and for each of them it's clear that they factor through $U_\solid$. This finishes the proof of \cref{enum:TateStacksOpenImmersion}.
	
	Part~\cref{enum:TateStacksTraceClass} is a formal consequence: If $V\coloneqq X\smallsetminus \ov{U}'$, then $V_\solid\rightarrow X_\solid$ is a naive open immersion by \cref{enum:TateStacksOpenImmersion}, $U_\solid\times_{X_\solid} V_\solid\simeq \emptyset$ follows from \cref{enum:TateStacksIntersection}, and if $A\coloneqq \Oo_{X\smallsetminus U}\lotimes_{\Oo_{X_\solid}}\Oo_{X\smallsetminus V}$, then it's formal to see that $\Mod_A(\Dd(X_\solid))$ is the kernel of the pullback functor $\Dd(X_\solid)\rightarrow \Dd(U_\solid)\times_{\Dd((U\cap V)_\solid)}\Dd(V_\solid)$. But this functor is an equivalence as $U\cup V=X$, and so $A\simeq 0$.
	\end{proof}
	We can finally show the desired criterion for idempotence.
	\begin{defi}
	If $X$ is a Tate adic space and $Z\subseteq X$ is a closed subset, the \emph{overconvergent neighbourhood of} $Z$ is the analytic stack
	\begin{equation*}
		Z^\dagger\coloneqq \limit_{U\supseteq Z}U_\solid\,,
	\end{equation*}
	where the limit is taken over all open neighbourhoods of $Z$. If it's clear that we're working in $\Dd(X_\solid)$, we often abuse notation and denote by $\Oo_{Z^\dagger}\coloneqq \colimit_{U\supseteq Z}\Oo_U\in \Dd(X_\solid)$ the \emph{sheaf of overconvergent functions on $Z$}. This is in favorable situations, but not always, the pushforward of the structure sheaf of $Z^\dagger$; see \cref{thm:OverconvergentIdempotentNuclear}\cref{enum:OZdaggerPushforward} below.
	\end{defi}
	
	\begin{thm}\label{thm:OverconvergentIdempotentNuclear}
	Let $X$ be a quasi-compact quasi-separated Tate adic space and let $Z\subseteq X$ be a closed subset such that for all points $z\in Z$ and all generalisations $z'\rightsquigarrow z$ also $z'\in Z$.
	\begin{alphanumerate}
		\item The ind-object
		\begin{equation*}
			\indcolim_{U\supseteq Z} \Oo_U\in \Ind\Dd(X_\solid)
		\end{equation*}
		is idempotent, nuclear, and obtained by killing the pro-idempotent $\prolim_{Z\cap \ov W=\emptyset} \Oo_W$, where the limit is taken over all open subsets $W\subseteq X$ such that $Z\cap \ov W=\emptyset$. In particular, $\Oo_{Z^\dagger}\in \Dd(X_\solid)$ is idempotent and nuclear.\label{enum:OZdaggerIdempotentNuclear}
		\item If for every affinoid open $j\colon\Spa(R,R^+)\rightarrow X$ the pullback $j^*(\Oo_{Z^\dagger})\in \Dd((R,R^+)_\solid)$ is connective%
		\footnote{Following discussions with Ben Antieau and Peter Scholze, we believe that connectivity can be replaced by the much weaker condition that $\Mod_{j^*(\Oo_{Z\smash{^\dagger}})}(\Dd(R))$ is left-complete, using an adaptation of \cite[Proposition~{\chpageref[2.16]{6}}]{MathewMondalDerivedRings}.}%
		, then pushforward along $Z^\dagger\rightarrow X_\solid$ induces a symmetric monoidal equivalence $\Dd(Z^\dagger)\simeq \Mod_{\Oo_{Z^\dagger}}(\Dd(X_\solid))$. In particular, in this case $\Oo_{Z^\dagger}$ is really the pushforward of the structure sheaf of $Z^\dagger$.\label{enum:OZdaggerPushforward}
	\end{alphanumerate}
	\end{thm}
	To prove \cref{thm:OverconvergentIdempotentNuclear}, we send a lemma in advance.
	\begin{lem}\label{lem:SpectralT4}
	Let $X$ be a spectral space and let $Y,Z\subseteq X$ be closed subsets such that for $z\in Z$ and $y\in Y$ there never exists a common generalisation $z\mathrel{\reflectbox{$\leadsto$}} x\leadsto y$ \embrace{in particular $Z\cap Y=\emptyset$}. Then there exist open neighbourhoods $U\supseteq Z$ and $V\supseteq Y$ such that $U\cap V=\emptyset$.
	\end{lem}
	\begin{proof}
	Fix $z\in Z$. By \cite[\stackstag{0906}]{Stacks}, $y\in Y$ there exist open neighbourhoods $U_y\ni z$ and $V_y\ni y$ such that $U_y\cap V_y=\emptyset$. By compactness of $Y$, there exist finitely many $y_1,\dotsc,y_n\in Y$ such that $Y\subseteq V_z\coloneqq V_{y_1}\cup\dotsb\cup V_{y_n}$. Let also $U_z\coloneqq U_{y_1}\cap \dotsb U_{y_n}$, so that $U_z\cap V_z=\emptyset$. By compactness of $Z$, there exist finitely many $z_1,\dotsc,z_m\in Z$ such that $Z\subseteq U\coloneqq U_{z_1}\cup\dotsb\cup U_{z_m}$. Putting $V\coloneqq V_{z_1}\cap\dotsb\cap V_{z_m}$, we have constructed $U$ and $V$ with the required properties.
	\end{proof}
	
	\begin{proof}[Proof of \cref{thm:OverconvergentIdempotentNuclear}]
	First observe that \cref{lem:SpectralT4} can be applied to any closed subset $Y\subseteq X$ such that $Z\cap Y=\emptyset$. Indeed, for any common generalisation $z\mathrel{\reflectbox{$\leadsto$}} x\leadsto y$, we would have $x\in Z$, as $Z$ is closed under generalisations, but then $y\in Z$, as $Z$ is also closed under specialisations.
	
	It follows that in the ind-object $\indcolim_{U\supseteq Z}\Oo_U$ we can restrict to open neighbouhoods of the form $U=X\smallsetminus \ov W$ for some open subset $W$ such that $Z\cap \ov W=\emptyset$. Indeed, for arbitrary $U$, apply \cref{lem:SpectralT4} to $Z$ and $X\smallsetminus U$ to get an open neighbourhood $W\supseteq (X\smallsetminus U)$ such that $Z\cap W=\emptyset$. Then $(X\smallsetminus \ov W)\subseteq U$, as desired.
	
	Let $\Oo_{\ov W}\coloneqq \Oo_{X\smallsetminus(X\smallsetminus \ov W)}\in \Dd(X_\solid)$ be the complementary idempotent determined by the open subset $X\smallsetminus \ov W$. Since each $\Oo_U$ is obtained by killing the idempotent $\Oo_{X\smallsetminus U}$, our observation implies that $\indcolim_{U\supseteq Z} \Oo_U$ is obtained by killing the pro-idempotent $\prolim_{Z\cap \ov W=\emptyset}\Oo_{\ov W}$. For all such $W$, applying \cref{lem:SpectralT4} to $Z$ and $\ov W$ provides another open neighbourhood $W'\supseteq \ov W$ such that still $Z\cap \ov{W'}=\emptyset$. By \cref{lem:TraceClassAnalyticStacks} and \cref{lem:TateOpenImmersions}\cref{enum:TateStacksTraceClass}, $\Oo_{W'}\rightarrow \Oo_W$ is trace-class and factors through $\Oo_{\ov W}$. It follows that $\prolim_{Z\cap \ov W=\emptyset}\Oo_W\simeq \prolim_{Z\cap \ov W=\emptyset}\Oo_{\ov W}$ and that the condition of \cref{lem:NuclearIdempotentAbstract} is satisfied, so that $\indcolim_{U\supseteq Z}\Oo_U$ is indeed idempotent and nuclear in $\Ind\Dd(X_\solid)$. Since $\colimit \colon \Ind\Dd(X_\solid)\rightarrow \Dd(X_\solid)$ preserves idempotents and nuclear objects, it follows that $\Oo_{Z^\dagger}\in \Dd(X_\solid)$ is idempotent and nuclear as well. This finishes the proof of \cref{enum:OZdaggerIdempotentNuclear}.
	
	For~\cref{enum:OZdaggerPushforward}, note that $Z^\dagger$ is clearly compatible with base change and so is $\Oo_{Z^\dagger}$ by \cref{enum:OZdaggerIdempotentNuclear} and \cref{lem:NuclearIdempotentAbstract}\cref{enum:IdempotentBasechange}. We may therefore assume that $X=\Spa(R,R^+)$ is affinoid and $\Oo_{Z^\dagger}$ is connective. Then $\Oo_{Z^\dagger}$ can be turned into an analytic ring using the induced analytic ring structure from $(R,R^+)_\solid$. It follows that a map $f\colon\AnSpec S\rightarrow \AnSpec(R,R^+)_\solid$ factors through $\Oo_{Z^\dagger}$ if and only if $S\simeq f^*(\Oo_{Z^\dagger})$. By \cref{lem:NuclearIdempotentAbstract}\cref{enum:NuclearIdempotentAbstract}, we have $\Oo_{Z^\dagger}\lotimes_{(R,R^+)_\solid}\Oo_{W}\simeq 0$ for all open $W$ such that $Z\cap\ov W=\emptyset$. Thus $S\simeq f^*(\Oo_{Z^\dagger})$ implies $f^*(\Oo_W)\simeq 0$ for all such $W$. By sandwiching open and closed subsets, we get $f^*(\Oo_{X\smallsetminus U})\simeq 0$ for all open neighbourhoods $U\supseteq Z$. By \cref{lem:TateOpenImmersions}\cref{enum:TateStacksOpenImmersion}, this implies that $f$ factors through $Z^\dagger\simeq \limit_{U\supseteq Z}U_\solid$.
	
	Conversely, if $f$ factors through $Z^\dagger$, then $f^*(\Oo_{X\smallsetminus U})\simeq 0$ for all $U$ and thus $f^*(\Oo_W)\simeq 0$ for all $W$ as above, using the same sandwiching argument. It follows that $S$ is a module over the nuclear idempotent ind-algebra obtained by killing $\prolim_{Z\cap \ov W=\emptyset}f^*(\Oo_W)$ in $\Dd(S)$. By \cref{lem:NuclearIdempotentAbstract}\cref{enum:IdempotentBasechange}, this is $\indcolim_{U\supseteq Z}f^*(\Oo_U)$. Then $S$ is also a module over the honest colimit $\colimit_{U\supseteq Z}f^*(\Oo_U)\simeq f^*(\Oo_{Z^\dagger})$, proving $S\simeq f^*(\Oo_{Z^\dagger})$.
	
	In conclusion, this argument shows that $Z^\dagger\simeq \AnSpec \Oo_{Z^\dagger}$ is an affine analytic stack and so $\Dd(Z^\dagger)\simeq \Mod_{\Oo_{Z^\dagger}}(\Dd((R,R^+)_\solid))$ follows by construction, as we've put the induced analytic ring structure on $\Oo_{Z^\dagger}$.
	\end{proof}
	This implies idempotence and nuclearity in the situation of \cref{thm:OverconvergentNeighbourhood}.
	\begin{cor}\label{cor:OZdaggerIdempotentNuclear}
	Let $X\coloneqq \Spa\IZ_p\qpower\smallsetminus\{p=0,q=1\}$ and let $Z\subseteq X$ be the union of the closed subsets $\Spa(\IF_p\qLaurent,\IF_p\qpower)$ and $\Spa(\IQ_p(\zeta_{p^n}),\IZ_p[\zeta_{p^n}])$ for all $n\geqslant 0$.
	\begin{alphanumerate}
		\item $Z$ is closed and closed under generalisations.\label{enum:ZGeneralising}
		\item For $n,r,s\geqslant 1$ such that $(p-1)p^n>s$, let $W_{n,r,s}\subseteq X$ be the rational open subset determined by $\abs{p^r}\leqslant \abs{q^{p^n}-1}\neq 0$, $\abs{(q-1)^s}\leqslant \abs{p}\neq 0$. Then $\Oo_{Z^\dagger}$ is idempotent, nuclear, and the colimit of the idempotent nuclear ind-algebra obtained by killing the idempotent pro-algebra $\prolim_{n,r,s}\Oo_{W_{n,r,s}}$.\label{enum:ZComplement}
	\end{alphanumerate} 
	\end{cor}
	\begin{proof}
	Let $x\in X\smallsetminus Z$. Then $\abs{p}_x\neq 0$, hence $\abs{(q-1)^s}_x\leqslant \abs{p}_x$ for $s\gge 0$. Choose such an $s$. Moreover, $\abs{q^{p^n}-1}_x\neq 0$ holds for all $n\geqslant 0$. Choose $n$ such that $(p-1)p^n>s$ and choose $r\gge 0$ such that $\abs{p^r}_x\leqslant \abs{q^{p^n}-1}_x$. Then $x\in W_{n,r,s}$. If we can show $Z\cap \ov{W}_{n,r,s}=\emptyset$, both~\cref{enum:ZGeneralising} and~\cref{enum:ZComplement} will follow. Indeed, this will imply that $X\smallsetminus Z$ is open and closed under specialisations, proving~\cref{enum:ZGeneralising}. Moreover, $X\smallsetminus Z=\bigcup_{n,r,s} W_{n,r,s}$ and so for any open subset $W$ such that $Z\cap \ov W=\emptyset$ we must have $W_{n,r,s}\supseteq \ov W$ for sufficiently large $n$, $r$, and $s$ by quasi-compactness of $\ov W$. Hence~\cref{enum:ZComplement} follows from \cref{thm:OverconvergentIdempotentNuclear}\cref{enum:OZdaggerIdempotentNuclear}.
	
	To show $Z\cap \ov{W}_{n,r,s}=\emptyset$, let $w\in W_{n,r,s}$. Since $(p-1)p^n>s$, we get $\abs{(q-1)^{(p-1)p^{i-1}}}_w<\abs{p}_x$ for all $i>n$ and so $\abs{\Phi_{p^i}(q)}_w=\abs{p}_w$, where $\Phi_{p^i}(q)$ denotes the $(p^i)$\textsuperscript{th} cyclotomic polynomial. Thus $0<\abs{p^{r+i-n}}_w\leqslant \abs{q^{p^i}-1}_w$ for $i>n$. In particular, $w\notin Z$. Even better: If $U_{i}$ denotes the rational open subset determined by $\abs{q^{p^i}-1}\leqslant \abs{p^{r+i-n+1}}\neq 0$ and $V$ denotes the rational open subset determined by $\abs{p}\leqslant \abs{(q-1)^{s+1}}\neq 0$, then the open set $\bigcup_{i\geqslant n}U_i\cup V$ contains $Z$ and doesn't intersect $W_{n,r,s}$, so indeed $Z\cap \ov W_{n,r,s}=\emptyset$.
	\end{proof}
	
	\subsection{Graded adic spaces}\label{subsec:GradedAdic}
	To deduce idempotence and nuclearity in the situation of \cref{thm:OverconvergentNeighbourhoodEquivariant}, let us describe how to encode gradings in terms of actions of the analytic stack
	\begin{equation*}
	\AdicGm\coloneqq \AnSpec \IZ[u^{\pm 1}]_\solid\,,
	\end{equation*}
	where $\IZ[u^{\pm 1}]_\solid$ is obtained from $\IZ[u^{\pm 1}]$ by solidifying both $u$ and $u^{-1}$. Equivalently, $\IZ[u^{\pm 1}]_\solid$ is the analytic ring associated to the discrete Huber pair $(\IZ[u^{\pm 1}],\IZ[u^{\pm 1}])$.
	
	\begin{numpar}[Graded adic spaces via $\AdicGm$-actions.]\label{par:GradedAdicI}
	Classically, the grading on $\IZ[\beta,t]$ in which $\beta$ and $t$ receive degree $2$ and $-2$, respectively, is encoded by an action of $\IG_m\coloneqq \Spec \IZ[u^{\pm 1}]$ on $\Spec \IZ[\beta,t]$. The action map $\Spec \IZ[\beta,t]\times \IG_m\rightarrow \Spec \IZ[\beta,t]$ corresponds to the ring map $\Delta\colon \IZ[\beta,t]\rightarrow \IZ[\beta,t]\otimes_\IZ\IZ[u^{\pm 1}]$ given by $\Delta(\beta)\coloneqq u^2\beta$, $\Delta(t)\coloneqq u^{-2}t$.
	
	In our situation, we're forced to work with the adic spectrum $\ov X^*\coloneqq \Spa \IZ[\beta,t]_{(p,t)}^\complete$ instead. But in the map $\Delta$ we can't just replace $\IZ[\beta,t]$ by its $(p,t)$-completion, since the tensor product  $\IZ[\beta,t]_{(p,t)}^\complete\otimes_\IZ\IZ[u^{\pm 1}]$ won't be $(p,t)$-complete anymore.
	
	To fix this, consider $\pi\colon \AdicGm\rightarrow \AnSpec \IZ_\solid$ and let $-\lotimes_{\IZ_\solid}\IZ[u^{\pm 1}]_\solid$ denote the pullback $\pi^*\colon \Dd(\IZ_\solid)\rightarrow \Dd(\AdicGm)$. By \cite[Lecture~\href{https://youtu.be/fUjn2rGw9SA?list=PLx5f8IelFRgGmu6gmL-Kf_Rl_6Mm7juZO&t=2056}{7}]{AnalyticStacks}, the process of adjoining a variable and then solidifying it preserves limits, and so
	\begin{equation*}
		\IZ[\beta,t]_{(p,t)}^\complete\lotimes_{\IZ_\solid}\IZ[u^{\pm 1}]_\solid\simeq \IZ[\beta,t,u^{\pm 1}]_{(p,t)}^\complete\,.
	\end{equation*}
	Thus, if we put $\ov X^*_\solid\coloneqq \AnSpec (\IZ[\beta,t]_{(p,t)}^\complete,\IZ[\beta,t]_{(p,t)}^\complete)_\solid$, we do get an action $\ov X^*_\solid\times\AdicGm\rightarrow \ov X^*_\solid$ simply by $(p,t)$-completing the map $\Delta$ above. Here and in the following, all products are taken in the $\infty$-category $\cat{AnStk}_{\IZ_\solid}$ of analytic stacks over $\IZ_\solid$. We let $\AdicGm^\bullet\colon \IDelta^\op\rightarrow \cat{AnStk}_{\IZ_\solid}$ denote the simplicial analytic stack corresponding to the underlying $\IE_1$-structure of the $\IE_\infty$-group object $\AdicGm$, and we let $\ov X^*_\solid\times\AdicGm^\bullet\colon \IDelta^\op\rightarrow \cat{AnStk}_{\IZ_\solid}$ denote the simplicial analytic stack corresponding to the $\AdicGm$-action on $\ov X^*_\solid$. Finally, let
	\begin{equation*}
		\B \AdicGm\coloneqq \colimit_{[n]\in\IDelta^\op}\AdicGm^n\quad\text{and}\quad \ov X^*_\solid/\AdicGm\coloneqq \colimit_{[n]\in\IDelta^\op}\ov X^*_\solid\times\AdicGm^n\,.
	\end{equation*}
	\end{numpar}
	\begin{numpar}[Graded objects as sheaves on $\B\AdicGm$.]\label{par:GradedAdicII}
		Let $\IG_{m, \IZ_\solid}\coloneqq \IG_m\times\AnSpec \IZ_\solid$. By adapting the usual proof, it's straightforward to show that 
		\begin{equation*}
			\Dd(\B\IG_{m, \IZ_\solid})\simeq \Gr\Dd(\IZ_\solid)
		\end{equation*}
		is the $\infty$-category of graded solid condensed abelian groups. Since we have a map of analytic stacks $c\colon \B\AdicGm\rightarrow \B\IG_{m, \IZ_\solid}$, we get a pullback functor $c^*\colon \Gr\Dd(\IZ_\solid)\rightarrow \Dd(\B\AdicGm)$. In this way, we can associate to any graded solid condensed $\IZ$-module a quasi-coherent sheaf on $\B\AdicGm$.
		
		In fact, it can be shown that $c^*$ defines a fully faithful embedding $\Gr\Dd(\IZ_\solid)\rightarrow \Dd(\B\AdicGm)$. We thank Peter Scholze for pointing out the following lemma (any errors are our own):
	\end{numpar}
	\begin{lem}\label{lem:ModulesOverProdZ}
		There exists an equivalence of $\infty$-categories
		\begin{equation*}
			\Mod_{\prod_{n\in\IZ}\IZ}\Dd(\IZ_\solid)\overset{\simeq}{\longrightarrow}\Dd(\B\AdicGm)\,.
		\end{equation*}
		Under this equivalence, the image of a graded object $M^*\in\Gr\Dd(\IZ_\solid)$ is sent to $\bigoplus_{n\in\IZ}M_n$, with component-wise action of the ring $\prod_{n\in\IZ}\IZ$; or in other words, a comodule over $\IZ[u^{\pm 1}]$ is sent to itself, regarded as a module over $\Hhom_\IZ(\IZ[u^{\pm 1}],\IZ)\cong \prod_{n\in\IZ}\IZ$.
	\end{lem}
	\begin{cor}\label{cor:c*FullyFaithful}
		The functor $c^*\colon \Gr\Dd(\IZ_\solid)\rightarrow \Dd(\B\AdicGm)$ is fully faithful.
	\end{cor}
	\begin{proof}
		It's enough to show that the functor $\Gr\Dd(\IZ_\solid)\rightarrow \Mod_{\prod_{n\in\IZ}\IZ}\Dd(\IZ_\solid)$ from \cref{lem:ModulesOverProdZ} is fully faithful. This can be reduced to the case of shifts (both in graded and homotopical direction) of the compact generator $\Null_{\IZ_\solid}\simeq \prod_\IN\IZ$, i.e.\ the case of graded object of the form $\Sigma^i\Null_{\IZ_\solid}(j)$ for some integers $i$ and $j$, which is straightforward to check.
	\end{proof}
	\begin{proof}[Proof sketch of \cref{lem:ModulesOverProdZ}]
		First observe that for any injective map $\alpha\colon [m]\rightarrow [n]$ in the simplex category $\IDelta$, the associated map $\alpha\colon \AdicGm^n\rightarrow \AdicGm^m$ is $!$-able and the pullback functor $\alpha^*\colon \Dd(\AdicGm^m)\rightarrow \Dd(\AdicGm^n)$ agrees up to shift with $\alpha^!$. It follows that $\alpha^*$ admits a left adjoint $\alpha_\natural$, which agrees up to shift with $\alpha_!$. The limit $\Dd(\B\AdicGm)\simeq \limit_{[n]\in\IDelta}\Dd(\AdicGm^n)$ can therefore be rewritten as a simplicial colimit in $\Pr^\L$ along the $\alpha_\natural$ functors. The inclusion of $\Dd(\IZ_\solid)$ into the colimit defines a functor which we'll denote $\eta_\natural \colon \Dd(\IZ_\solid)\rightarrow \Dd(\B\AdicGm)$; its right adjoint is given by pullback along the canonical map $\eta\colon \AnSpec \IZ_\solid\rightarrow \B\AdicGm$.
		
		This colimit diagram lies, in fact, in $\Pr_\omega^\L$. Indeed, each $\Dd(\AdicGm^n)$ is compactly generated and each $\alpha_\natural$ preserves compact objects, since its right adjoint $\alpha^*$ admits a further right adjoint $\alpha_*$. It follows that $\Dd(\B\AdicGm)$ is compactly generated, and the images of $\Null_{\IZ_\solid}\lotimes_{\IZ_\solid}\Oo_{\AdicGm^n_\solid}$ for all $n$ form a set of compact generators. In fact, $\eta_\natural\Null_{\IZ_\solid}$ is already a compact generator, since each $\AdicGm^n\rightarrow \B\AdicGm$ factors through $\eta\colon \AnSpec \IZ_\solid\rightarrow \B\AdicGm$. Next observe that
		\begin{equation*}
			\eta_\natural\IZ\simeq \prod_{n\in\IZ}\Oo(n)\quad\text{and}\quad \eta_\natural\Null_{\IZ_\solid}\simeq \prod_{\IN}\prod_{n\in\IZ}\Oo(n)\,,
		\end{equation*}
		where  $\Oo(n)\in\Dd(\B\AdicGm)$ denotes the image of the graded $\IZ$-module $\IZ(n)$. Indeed, to construct a map, $\eta_\natural\IZ\rightarrow \prod_{n\in\IZ}\Oo(n)$ it's enough to provide a map $\IZ\rightarrow \eta^*\big(\prod_{n\in\IZ}\Oo(n)\big)\simeq \prod_{n\in\IZ}\IZ$; we take the diagonal map. To check that this induces an equivalence, we check that it becomes an equivalence in each $\Dd(\AdicGm^n)$. This is a straightforward calculation, using the fact that the $\alpha_\natural$ functors satisfy base change (which follows from proper base change, as they agree with $\alpha_!$ up to shift). In the same way one shows the formula for $\eta_\natural\Null_{\IZ_\solid}$.
		
		Now $\eta_\natural\IZ\simeq\prod_{n\in\IZ}\Oo(n)$ admits a $\prod_{n\in\IZ}\IZ$-module structure in an apparent way, hence it induces a functor $\Mod_{\prod_{n\in\IZ}}\Dd(\IZ)\rightarrow \Dd(\B\AdicGm)$. Extending $\Dd(\IZ_\solid)$-linearly, we obtain a functor
		\begin{equation*}
			\Mod_{\prod_{n\in\IZ}}\Dd(\IZ_\solid)\simeq\left(\Mod_{\prod_{n\in\IZ}}\Dd(\IZ)\right)\otimes_{\Dd(\IZ)}\Dd(\IZ_\solid)\longrightarrow \Dd(\B\AdicGm)\,,
		\end{equation*}
		as desired. It is essentially surjective, since the compact generator $\prod_{n\in\IZ}\Null_{\IZ_\solid}$ is mapped to the compact generator $\eta_\natural\Null_{\IZ_\solid}$. To check fully faithfulness, we only need to verify that $\Hom_{\prod_{n\in\IZ}\IZ}\bigl(\prod_{n\in\IZ}\Null_{\IZ_\solid},\prod_{n\in\IZ}\Null_{\IZ_\solid}\bigr)\rightarrow\Hom_{\Dd(\B\AdicGm)}(\eta_\natural\Null_{\IZ_\solid},\eta_\natural\Null_{\IZ_\solid})$ is an equivalence. By adjunction, we may rewrite the right-hand side as $\Hom_{\Dd(\IZ)}(\Null_{\IZ_\solid},\eta^*\eta_\natural\Null_{\IZ_\solid})$ and then the claim is clear from $\eta^*\eta_\natural\Null_{\IZ_\solid}\simeq \prod_{n\in\IZ}\Null_{\IZ_\solid}$.
	\end{proof}
	In the next two lemmas, we'll deduce that the graded $\IZ_p[\beta]\llbracket t\rrbracket$-modules $\fil_{\qHodge}^\star\qhatdeRham_{(\IZ/p^{\alpha})/\IZ_p}$ can be regarded as sheaves on $\ov X^*_\solid/\AdicGm$ without loss of information.
	
	\begin{lem}\label{lem:DXTModDBT}
		Let $\Oo_{\ov X^*/\AdicGm}\in\Dd(\B\AdicGm)$ denote the pushforward of the structure sheaf of $\ov X^*_\solid/\AdicGm$. Then pushforward along $\ov X^*_\solid/\AdicGm\rightarrow \B\AdicGm$ induces a symmetric monoidal equivalence of $\infty$-categories
		\begin{equation*}
			\Dd\left(\ov X^*_\solid/\AdicGm\right)\simeq \Mod_{\Oo_{\ov X^*/\AdicGm}}\bigl(\Dd(\B\AdicGm)\bigr)\,.
		\end{equation*}
	\end{lem}
	\begin{proof}
		The same argument as in \cref{par:GradedAdicI} shows $\ov X^*_\solid\times\AdicGm^n\simeq \AnSpec (\IZ[\beta,t,u_1^{\pm 1},\dotsc,u_n^{\pm 1}]_{(p,t)}^\complete)_\solid$.
		By definition, $\Dd(\B\AdicGm)\simeq \limit_{[n]\in\IDelta}\Dd(\AdicGm^n)$ and $\Dd(\ov X^*_\solid/\AdicGm)\simeq \limit_{[n]\in\IDelta}\Dd(\ov X^*_\solid\times\AdicGm^n)$, where the cosimplicial limits are taken along the pullback functors. Observe that the pushforward functors $\pi_*\colon \Dd(\ov X^*_\solid\times\AdicGm^n)\rightarrow \Dd(\AdicGm^n)$ commute with these pullbacks. Indeed, if we would take the limit along the $!$-pullbacks, this would follow from proper base change (by passing to right adjoints). Since $\IZ\rightarrow \IZ[u^{\pm 1}]$ is smooth of relative dimension~$1$ and $\Omega_{\IZ[u^{\pm 1}]/\IZ}^1\cong \IZ[u^{\pm 1}]\d u$ is a free module of rank~$1$, we get $\pi^!\simeq \Sigma^{-1}\pi^*$ by \cite[Theorem~\chref{11.6}]{Condensed}, and so commutativity for the $*$-pullbacks follows.
		
		Therefore $\Oo_{\ov X^*/\AdicGm}\in\Dd(\B\AdicGm)$ is given by the degree-wise pushforwards of the structure sheaves $\Oo_{\ov X^*_\solid\times\AdicGm^n}$, that is, by $\IZ[\beta,t,u_1^{\pm 1},\dotsc,u_n^{\pm 1}]_{(p,t)}^\complete\in \Dd(\AdicGm^n)$ for all $[n]\in\IDelta$. In every degree, the pushforward induces an equivalence
		\begin{equation*}
			\Dd\left(\ov X^*_\solid\times\AdicGm^n\right)\overset{\simeq}{\longrightarrow}\Mod_{\IZ[\beta,t,u_1^{\pm 1},\dotsc,u_n^{\pm 1}]_{(p,t)}^\complete}\bigl(\Dd(\AdicGm^n)\bigr)\,.
		\end{equation*}
		Using this observation, $\Dd\left(\ov X^*_\solid/\AdicGm\right)\simeq \Mod_{\Oo_{\ov X^*/\AdicGm}}(\Dd(\B\AdicGm))$ is completely formal.
	\end{proof}
	
	\begin{lem}\label{lem:GradedObjectsSheavesOnBT}
		Let $\IZ_p[\beta]\llbracket t\rrbracket\in \Gr\Dd(\IZ_\solid)$ denote the graded $(p,t)$-completion of the discrete graded ring $\IZ[\beta,t]$ and equip $\Mod_{\IZ_p[\beta]\llbracket t\rrbracket}(\Gr\Dd(\IZ_\solid))_{(p,t)}^\complete$ with the $(p,t)$-completed graded solid tensor product. Then $c^*$ induces a fully faithful lax symmetric monoidal functor
		\begin{equation*}
			\Mod_{\IZ_p[\beta]\llbracket t\rrbracket}\bigl(\Gr\Dd(\IZ_\solid)\bigr)_{(p,t)}^\complete \longrightarrow \Mod_{\Oo_{\ov X^*/\AdicGm}}\bigl(\Dd(\B\AdicGm)\bigr)\,,
		\end{equation*}
		which is symmetric monoidal when restricted to the full sub-$\infty$-category spanned by those objects in $\Mod_{\IZ_p[\beta]\llbracket t\rrbracket}(\Gr\Dd(\IZ_\solid))_{(p,t)}^\complete$ that are uniformly bounded below in every graded degree.%
		\footnote{By contrast, the graded solid tensor product on $\Gr\Dd(\IZ_\solid)$ does \emph{not} preserve $p$-complete objects, not even if they're uniformly bounded below, because being $p$-complete is not preserved under infinite direct sums.}
	\end{lem}
	\begin{proof}

		To construct the desired functor, we compose $c^*$ with $(p,t)$-completion to obtain
		\begin{equation*}
			\Mod_{\IZ_p[\beta]\llbracket t\rrbracket}\bigl(\Gr\Dd(\IZ_\solid)\bigr)\overset{c^*}{\longrightarrow}\Mod_{c^*(\IZ_p[\beta]\llbracket t\rrbracket)}\bigl(\Dd(\B\AdicGm)\bigr)\xrightarrow{(-)_{(p,t)}^\complete}\Mod_{\Oo_{\ov X^*/\AdicGm}}\bigl(\Dd(\B\AdicGm)\bigr)\,.
		\end{equation*}
		The functor $c^*$ is symmetric monoidal and $(-)_{(p,t)}^\complete$ is lax symmetric monoidal. Hence the composition is lax  symmetric monoidal. Moreover, it is symmetric monoidal when restricted to graded $\IZ_p[\beta]\llbracket t\rrbracket$-modules that are uniformly bounded below in every graded degree. Indeed, the image of such objects in $\Mod_{\Oo_{\ov X^*/\AdicGm}}(\Dd(\B\AdicGm))\simeq\limit_{[n]\in\IDelta}\Dd(\ov X^*_\solid\times\AdicGm^n)$ will be bounded below and $(p,t)$-complete in every cosimplicial degree, because the pullback functors along which the limit is taken preserve bounded below and $(p,t)$-complete objects (the latter because they preserve limits; see the argument in \cref{par:GradedAdicI}). So we can reduce to the fact that the solid tensor product in $\Dd(\ov X^*_\solid\times\AdicGm^n)$ preserves bounded below $(p,t)$-complete objects.
		
		Clearly $(-)_{(p,t)}^\complete\circ c^*$ factors through $\Mod_{\IZ_p[\beta]\llbracket t\rrbracket}(\Gr\Dd(\IZ_\solid))_{(p,t)}^\complete$. The resulting functor
		\begin{equation*}
			\Mod_{\IZ_p[\beta]\llbracket t\rrbracket}(\Gr\Dd(\IZ_\solid))_{(p,t)}^\complete\longrightarrow \Mod_{\Oo_{\ov X^*/\AdicGm}}\bigl(\Dd(\B\AdicGm)\bigr)
		\end{equation*}
		is symmetric monoidal on uniformly bounded below objects. Fully faithfulness can be checked modulo $(p,t)$, so it'll be enough to check that $\Mod_{\IF_p[\beta]}(\Gr\Dd(\IZ_\solid))\rightarrow \Mod_{c^*(\IF_p[\beta])}(\Dd(\B\AdicGm))$ is fully faithful, which follows from \cref{cor:c*FullyFaithful}.
	\end{proof}

	\begin{lem}\label{lem:X*Cover}
	Let $X^*\subseteq \ov X^*$ be the subset $\Spa\IZ[\beta,t]_{(p,t)}^\complete\smallsetminus\{p=0,\beta t=0\}$. Then $X^*$ is a Tate adic space and its associated analytic stack $X^*_\solid$ can be written as the following pushout:
	\begin{equation*}
		\begin{tikzcd}
			\AnSpec\left(\IZ[\beta,t]_{(p,t)}^\complete\bigl[\localise{p\beta t}\bigr],\IZ[\beta,t]_{(p,t)}^\complete\right)_\solid\rar\dar\drar[pushout] & \AnSpec\left(\IZ[\beta,t]_{(p,t)}^\complete\bigl[\localise{\beta t}\bigr],\IZ[\beta,t]_{(p,t)}^\complete\right)_\solid\dar\\
			\AnSpec\left(\IZ[\beta,t]_{(p,t)}^\complete\bigl[\localise{p}\bigr],\IZ[\beta,t]_{(p,t)}^\complete\right)_\solid\rar & X^*_\solid
		\end{tikzcd}
	\end{equation*}
	Moreover, the $\AdicGm$-action on $\ov X^*_\solid$ restricts to an action on $X^*_\solid$, and if $\Oo_{X^*/\AdicGm}\in \Dd(\B\AdicGm)$ denotes the pushforward of the structure sheaf of $X^*_\solid/\AdicGm$, then pushforward along $X^*_\solid/\AdicGm\rightarrow\B\AdicGm$ induces a symmetric monoidal equivalence
	\begin{equation*}
		\Dd(X^*_\solid/\AdicGm)\simeq \Mod_{\Oo_{X^*/\AdicGm}}\bigl(\Dd(\B\AdicGm)\bigr)\,.
	\end{equation*}
	\end{lem}
	\begin{proof}
	By \cref{par:AdicRecollectionsII}, $X^*_\solid$ is glued together from rational open subsets of $\ov X^*$. For example, one can take $U_1=\left\{x\in\ov X^*\ \middle|\ \abs{\beta t}_x\leqslant \abs{p}_x\neq 0\right\}$ and $U_2=\left\{x\in\ov X^*\ \middle|\ \abs{p}_x\leqslant \abs{\beta t}_x\neq 0\right\}$ and then
	\begin{equation*}
		X^*_\solid\simeq U_{1, \solid}\sqcup_{(U_1\cap U_2)_\solid}U_{2, \solid}\,.
	\end{equation*}
	To show the desired pushout, it's enough that $Y_{1, \solid}\coloneqq \AnSpec(\IZ[\beta,t]_{(p,t)}^\complete[1/p],\IZ[\beta,t]_{(p,t)}^\complete)_\solid$ and $Y_{2, \solid}\coloneqq \AnSpec(\IZ[\beta,t]_{(p,t)}^\complete[1/(\beta t)],\IZ[\beta,t]_{(p,t)}^\complete)_\solid$ form a $!$-cover after pullback to $U_{1,  \solid}$ and $U_{2,  \solid}$. This is clear, as $Y_{1, \solid}\times_{\ov X^*_\solid}U_{1, \solid}\simeq U_{1, \solid}$ and similarly $Y_{2, \solid}\times_{\ov X^*_\solid}U_{2, \solid}\simeq U_{2, \solid}$.
	
	To see that the $\AdicGm$-action on $\ov X^*_\solid$ restricts to an action on $X^*_\solid$, just observe that $p$ and $\beta t$ are homogeneous elements. The pushout above implies that the pushforward $\Oo_{X^*}\in \Dd(\IZ_\solid)$ of the structure sheaf of $X^*_\solid$ is given by
	\begin{equation*}
		\Oo_{X^*}\simeq \IZ[\beta,t]_{(p,t)}^\complete\bigl[\localise{p}\bigr]\times_{\IZ[\beta,t]_{(p,t)}^\complete\bigl[\frac1{p\beta t}\bigr]}\IZ[\beta,t]_{(p,t)}^\complete\bigl[\localise{\beta t}\bigr]\,,
	\end{equation*}
	the pullback being taken in the derived sense. Now $\Dd(X^*_\solid\times\AdicGm^n)\simeq \Mod_{\Oo_{X^*\times\AdicGm^n}}(\Dd(\AdicGm^n))$ holds for all $[n]\in\IDelta$, since the same is true for $Y_{1, \solid}$, $Y_{2, \solid}$, and $Y_{1, \solid}\times_{X^*_\solid}Y_{2, \solid}$. This finally implies $\Dd(X^*_\solid/\AdicGm)\simeq \Mod_{\Oo_{X^*/\AdicGm}}(\Dd(\B\AdicGm))$, as desired.
	\end{proof}
	
	We can finally show idempotence and nuclearity in the situation of \cref{thm:OverconvergentNeighbourhoodEquivariant}.
	%
	%
	\begin{cor}\label{cor:OZTdaggerIdempotentNuclear}
	Let $Z^*\subseteq X^*$ be union of the closed subsets $\{p=0\}$ and $\{[p^n]_{\ku}(t)=0\}$ for all $n\geqslant 0$, where $[p^n]_{\ku}(t)\coloneqq ((1+\beta t)^{p^n}-1)/\beta$ denotes the $p^n$-series of the formal group law of $\ku$.
	\begin{alphanumerate}
		\item $Z^*$ is closed and closed under generalisations. Moreover, the $\AdicGm$-action on $X^*_\solid$ restricts to an action on the overconvergent neighbourhood $Z^{*,\dagger}$ of $Z^*$.\label{enum:Z*Generalising}
		\item For $n,r,s\geqslant 1$ such that $(p-1)p^n>s$, let $W_{n,r,s}^*\subseteq X^*$ be the rational open subset determined by $\abs{p^r}\leqslant \abs{[p^n]_{\ku}(t)}\neq 0$, $\abs{(\beta t)^s}\leqslant \abs{p}\neq 0$. Then  $\Oo_{Z^{*,\dagger}/\AdicGm}\in \Dd(X^*_\solid/\AdicGm)$ is idempotent, nuclear, and the colimit of the ind-algebra obtained by killing the idempotent pro-algebra $\prolim_{n,r,s} \Oo_{W_{n,r,s}^*/\AdicGm}$.\label{enum:Z*Complement}
	\end{alphanumerate}
	\end{cor}
	\begin{proof}
	The proof of \cref{cor:OZdaggerIdempotentNuclear} can be carried over to show that $Z^*\cap \ov W_{n,r,s}^*=\emptyset$ and $X^*\smallsetminus Z^*=\bigcup_{n,r,s}W_{n,r,s}^*$. Hence $Z^*$ is closed and closed under generalisations. Moreover, the $\AdicGm$-equivariant open subsets $X^*\smallsetminus \ov W_{n,r,s}^*$ are coinitial among all open neighbourhoods of $Z^*$, because for an arbitrary $U\supseteq Z^*$, the complement $X^*\smallsetminus U$ is quasi-compact and thus contained in some $W_{n,r,s}^*$. Since the $W_{n,r,s}^*$ are $\AdicGm$-equivariant, as they're defined by homogeneous elements, we see that $Z^{*,\dagger}$ acquires a $\AdicGm$-action. This finishes the proof of \cref{enum:Z*Generalising}.
	
	For part~\cref{enum:Z*Complement}, \cref{thm:OverconvergentIdempotentNuclear} shows that $\Oo_{Z^{*,\dagger}}$ is the colimit of the idempotent nuclear ind-algebra obtained by killing $\prolim_{n,r,s}\Oo_{W_{n,r,s}^*}$. Since $Z^{*,\dagger}\times\AdicGm^n\simeq \limit_{U^*\supseteq Z^*}(U_\solid^*\times\AdicGm)$, where the limit is taken over all $\AdicGm$-equivariant open neighbourhoods, and since killing pro-idempotents is compatible with base change in the nuclear case by \cref{lem:NuclearIdempotentAbstract}\cref{enum:IdempotentBasechange}, we get that $\Oo_{Z^{*,\dagger}\times\AdicGm^n}$ is similarly given by killing $\prolim_{n,r,s}\Oo_{W_{n,r,s}^*\times\AdicGm^n}$ in $\Dd(X^*_\solid\times\AdicGm^n)$. Now let $A\in \Dd(X^*_\solid/\AdicGm)$ be the colimit of the ind-algebra given by killing $\prolim_{n,r,s}\Oo_{W_{n,r,s}^*/\AdicGm}$. Then \cref{lem:TraceClassAnalyticStacks} shows that all sufficiently large transition maps in this pro-object are trace-class again. Hence $A$ is idempotent, nuclear, and the base change result from \cref{lem:NuclearIdempotentAbstract}\cref{enum:IdempotentBasechange} shows that the pullbacks of $A$ to $X^*_\solid\times\AdicGm^n$ agree with $\Oo_{Z^{*,\dagger}\times\AdicGm^n}$ for all $[n]\in\IDelta$. This implies $\Oo_{Z^{*,\dagger}/\AdicGm}\simeq A$, as both of the maps
	\begin{equation*}
		\Oo_{Z^{*,\dagger}/\AdicGm}\longrightarrow\Oo_{Z^{*,\dagger}/\AdicGm}\lotimes_{\Oo_{X^*_\solid/\AdicGm}}A\longleftarrow A
	\end{equation*}
	become equivalences after pullback to $X^*_\solid\times\AdicGm^n$ for all $[n]\in\IDelta$.
	\end{proof}

	\subsection{Proof of \texorpdfstring{\cref{thm:OverconvergentNeighbourhood,thm:OverconvergentNeighbourhoodEquivariant}}{Theorems 1.19 and 1.20}}\label{subsec:MainProof}
	
	In this final subsection, we'll give a completely explicit description of the homotopy groups of
	\begin{equation*}
	\TCref\bigl(\ku_p^\complete\otimes\IQ/\ku_p^\complete\bigr)\quad\text{and}\quad\TCref\bigl(\KU_p^\complete\otimes\IQ/\KU_p^\complete\bigr)\,.
	\end{equation*}
	By \cref{exm:THHrefk1m} and \cref{lem:S1ActiontComplete}, these objects are obtained from $(\ku_p^\complete)^{\h S^1}$ and $(\KU_p^\complete)^{\h S^1}$, respectively, by killing the idempotent pro-algebras%
	\footnote{In the case $p=2$, the pro-systems need to be indexed by $\alpha$ even and $\geqslant4$, but we'll ignore this since it makes no difference}
	\begin{equation*}
	\prolim_{\alpha\geqslant 2}\TC^-\bigl((\ku/p^\alpha)/\ku\bigr)\quad\text{and}\quad\prolim_{\alpha\geqslant 2}\TC^-\bigl((\KU/p^\alpha)/\KU\bigr)\,.
	\end{equation*}
	The arguments from \cref{subsec:RefinedTC-}, particularly Corollaries~\labelcref{cor:ProIdempotent}, \labelcref{cor:ProTraceClass}, and the proof of \cref{thm:RefinedTC-qHodge}, show that $\TCref$ is concentrated in even degrees in both cases, and the even homotopy groups are given by
	\begin{equation*}
	\pi_{2*}\TCref\bigl(\ku_p^\complete\otimes\IQ/\ku_p^\complete\bigr)\cong \Akup\,,\quad \pi_{2*}\TCref\bigl(\KU_p^\complete\otimes\IQ/\KU_p^\complete\bigr)\cong \AKUp[\beta^{\pm 1}]\,,
	\end{equation*}
	where $\Akup$ is obtained by killing the idempotent pro-algebra
	$\prolim_{\alpha\geqslant 2}\fil_{\qHodge}^\star\qhatdeRham_{(\IZ/p^{\alpha})/\IZ_p}$ in graded $(p,t)$-complete $\IZ_p[\beta]\llbracket t\rrbracket$-modules and $\AKUp$ is obtained by killing the idempotent pro-algebra $\prolim_{\alpha\geqslant 2}\qHodge_{(\IZ/p^{\alpha})/\IZ_p}$ in $(p,q-1)$-complete $\IZ_p\qpower$-modules. Moreover, we already know that $\Akup$ and $\AKUp$ are idempotent nuclear ind-objects.

	Our goal is to identify $\Akup$ and $\AKUp$ with the structure sheaves of the analytic stacks $Z^{*,\dagger}/\AdicGm$ and $Z^\dagger$, respectively (see \cref{cor:OZdaggerIdempotentNuclear,cor:OZTdaggerIdempotentNuclear}). To this end, let us first discuss how to transport $\Akup$ and $\AKUp$ into the solid condensed world.
	
	\begin{numpar}[Nuclear modules à la Efimov and à la Clausen--Scholze.]\label{par:NuclearEfimovClausenScholze}
	Let $R$ be a ring and $I\subseteq R$ a finitely generated homogeneous ideal. Efimov defines an $\infty$-category of \emph{nuclear $\widehat{R}_{I}$-modules}, which (along many equivalent characterisations) can be described as
	\begin{equation*}
		\Nuc(\widehat{R}_{I})\simeq \Nuc\Ind\bigl(\widehat{\Dd}_I(R)\bigr);
	\end{equation*}
	see \cite[Corollary~\chref{4.4}]{EfimovLimits} (also recall that $\Nuc\Ind(-)$ is set-theoretically ok thanks to \cref{rem:NucInd}). Let $\widehat{R}_{I,\solid}\coloneqq (\widehat{R}_I,\widehat{R}_I)_\solid$ be the analytic ring associated to the Huber pair $(\widehat{R}_I,\widehat{R}_I)$ (see \cref{par:AdicRecollectionsI}). Then we can also consider the $\infty$-category $\Nuc(\Dd(\widehat{R}_{I,\solid}))$ of nuclear $\widehat{R}_{I,\solid}$-modules.%
	\footnote{In fact, for \emph{any} Huber pair $(\widehat{R}_I,R^+)$ the nuclear objects $\Nuc(\Dd((\widehat{R}_I,R^+)_\solid))$ will be independent of the choice of $R^+$. See \cite[Example~\chref{3.34}]{DescentSolidQCoh} for example.}
	Efimov \cite[Corollary~\chref{7.6}]{EfimovLimits} constructs a fully faithful strongly continuous symmetric monoidal functor
	\begin{equation*}
		\Nuc\Dd(\widehat{R}_{I, \solid})\longrightarrow \Nuc(\widehat{R}_{I})\,,
	\end{equation*}
	which is an equivalence on bounded objects.
	\end{numpar}
	\begin{numpar}[$\AKUp$ and $\Akup$ as sheaves on analytic stacks.]
	Applying Efimov's result above for $R=\IZ[q]$ and $I=(p,q-1)$, we see that the bounded object $\AKUp$ is in the essential image of $\Nuc(\Dd(\IZ_p\qpower_\solid))$. Its preimage can be explicitly described: As We can regard each $\qHodge_{(\IZ/p^{\alpha})/\IZ_p}$ as a $(p,q-1)$-complete%
	\footnote{Observe that $\qHodge_{(\IZ/p^{\alpha})/\IZ_p}$ is automatically $p$-complete, since it is $(q-1)$-complete and contains an element of the form $p^\alpha/(q-1)$ by construction.}
	solid condensed $\IZ_p\qpower$-module by $(p,q-1)$-completing the associated discrete condensed abelian group. The pro-algebra $\prolim_{\alpha\geqslant2}\qHodge_{(\IZ/p^{\alpha})/\IZ_p}$ is still idempotent in $\Pro\Dd(\IZ_p\qpower_\solid)$ and has eventually trace-class transition maps. Thus, by killing it, we get an idempotent nuclear algebra in $\Ind\Dd(\IZ_p\qpower_\solid)$. Its colimit is the preimage of $\AKUp$.
	
	In a similar way, via \cref{lem:GradedObjectsSheavesOnBT}, we can regard $\prolim_{\alpha\geqslant2}\fil_{\qHodge}^\star\qhatdeRham_{(\IZ/p^{\alpha})/\IZ_p}$ as an idempotent pro-algebra in $\Mod_{\Oo_{\smash{\ov X^*_\solid}/\AdicGm}}(\Dd(\B\AdicGm))$. By killing it and taking the colimit of the result idempotent nuclear ind-algebra, we can regard $\Akup$ as an object in $\Nuc\Mod_{\Oo_{\smash{\ov X^*_\solid}/\AdicGm}}(\Dd(\B\AdicGm))$
	\end{numpar}
	
	The following lemma shows that $\Akup$ and $\AKUp$ are already sheaves on $X^*_\solid/\AdicGm$ and $X_\solid$, where we put $X^*\coloneqq \ov X^*\smallsetminus\{p=0,\beta t=0\}$ and $X\coloneqq \Spa\IZ_p\qpower\smallsetminus\{p=0,q=1\}$ as before.
	\begin{lem}\label{lem:AkuOverAnalyticLocus}
	$\Akup$ vanishes after $(p,\beta)$-completion and after $(p,t)$-completion. $\AKUp$ vanishes after $(p,q-1)$-completion. In particular, $\Akup$ and $\AKUp$ are already contained in the full sub-$\infty$-categories $\Dd(X^*_\solid/\AdicGm)\simeq \Mod_{\Oo_{X^*/\AdicGm}}(\Dd(\B\AdicGm))$ and $\Dd(X_\solid)\simeq \Mod_{\Oo_X}(\Dd(\IZ_\solid))$.
	\end{lem}
	\begin{proof}
	By Nakayama's lemma it's enough to show $\Akup/(p,\beta)\simeq 0$ and $\Akup/(p,t)\simeq 0$. Since $\AKUp[\beta^{\pm 1}]$ is a $\Akup$-algebra, this will also show $\AKUp/(p,q-1)\simeq 0$. Since $\Akup/t$ is concentrated in nonnegative graded degrees, it is automatically $\beta$-complete, so it's already enough to show $\Akup/(p,\beta)\simeq 0$. Now $\ku\rightarrow \ku/(p,\beta)\simeq \IF_p$ is a map of $\IE_\infty$-ring spectra, and it's clear from \cref{exm:THHrefk1m} and \cref{lem:S1ActiontComplete} that $\TCref(-\otimes\IQ/-)$ satisfies base change along $\IE_\infty$-maps. So $\TCref(\ku\otimes\IQ/\ku)/(p,\beta)\simeq \TCref(\IF_p\otimes\IQ/\IF_p)\simeq 0$.
	
	It follows that $(\Akup)_{(p,\beta t)}^\complete\simeq 0$. Using the pullback square from \cref{lem:X*Cover}
	, we get
	\begin{equation*}
		\Akup\simeq \Akup\lotimes_{\Oo_{\ov X_\solid^\star/\AdicGm}}\Oo_{X^*/\AdicGm}
	\end{equation*}
	and so $\Akup$ is indeed a $\Oo_{X^*/\AdicGm}$-module. The argument for $\AKUp$ is analogous.
	\end{proof}
	To finish the proof of \cref{thm:OverconvergentNeighbourhood,thm:OverconvergentNeighbourhoodEquivariant}, we analyse the pro-systems $\prolim_{n,r,s}\Oo_{W_{n,r,s}}$ and $\prolim_{n,r,s}\Oo_{W_{n,r,s}^*/\AdicGm}$ from \cref{cor:OZdaggerIdempotentNuclear,cor:OZTdaggerIdempotentNuclear} and show that they are pro-isomorphic to the pro-systems from \cref{thm:RefinedTC-qHodge}\cref{enum:RefinedTC-KU} and~\cref{enum:RefinedTC-ku}, respectively.
	\begin{lem}\label{lem:qHodgeSystemOfOpensI}
	For every fixed $\alpha\geqslant 2$ and all sufficiently large $n$, $r$, $s$, there exist maps
	\begin{align*}
		\Oo_{W_{n,r,s}^*/\AdicGm}&\longrightarrow \fil_{\qHodge}^\star\qhatdeRham_{(\IZ/p^{\alpha})/\IZ_p}\lotimes_{\Oo_{\ov X^*_\solid/\AdicGm}}\Oo_{X^*/\AdicGm}\,,\\
		\Oo_{W_{n,r,s}}&\longrightarrow \qHodge_{(\IZ/p^{\alpha})/\IZ_p}\lotimes_{\IZ_p\qpower_\solid}\Oo_{X}
	\end{align*}
	in $\Dd(X^*_\solid/\AdicGm)$ and $\Dd(X_\solid)$, respectively.
	\end{lem}
	\begin{proof}
	By construction, the $q$-de Rham complex $\qdeRham_{(\IZ/p^\alpha)/\IZ_p}$ contains elements of the form $\phi^i(\phi(p^\alpha)/\Phi_p(q))=p^\alpha/\Phi_{p^{i+1}}(q)$ for all $i\geqslant 0$, and we have $p^\alpha\in \fil_{\qHodge}^1\qhatdeRham_{(\IZ/p^{\alpha})/\IZ_p}$. When we regard $\fil_{\qHodge}^\star\qhatdeRham_{(\IZ/p^\alpha)/\IZ_p}$ as a graded $\IZ_p[\beta]\llbracket t\rrbracket$-module, this precisely means that $p^\alpha$ is divisible by~$t$. Hence we have elements of the form
	\begin{equation*}
		\frac{p^{(n+1)\alpha}}{[p^n]_{\ku}(t)}=\frac{p^\alpha}{t}\cdot \frac{\phi(p^\alpha)}{\Phi_p(q)}\dotsb\frac{\phi^{n}(p^\alpha)}{\Phi_{p^n}(q)}\in \fil_{\qHodge}^\star\qhatdeRham_{(\IZ/p^{\alpha})/\IZ_p}
	\end{equation*}
	for all $n\geqslant 0$. Similarly, there exist elements of the form $(\beta t)^N/p$ in $\fil_{\qHodge}^\star\qhatdeRham_{(\IZ/p^{\alpha})/\IZ_p}$ for sufficiently large~$N$. Indeed, the ring $\qdeRham_{(\IZ/p^\alpha)/\IZ_p}$ is $(p,\Phi_p(q))$-complete and contains an element of the form $p^\alpha/\Phi_p(q)$. Applying the nilpotence criterion from \cite[Proposition~\chref{2.5}]{BhattClausenMathewK1Local}, we see that $\Phi_p(q)$ is nilpotent in $\Fil_{\qHodge}^*\qhatdeRham_{(\IZ/p^{\alpha})/\IZ_p}/p$. Then $(q-1)^{p-1}$ must be nilpotent as well, and so $(q-1)^N$ must be divisible by $p$ in $\fil_{\qHodge}^\star\qhatdeRham_{(\IZ/p^{\alpha})/\IZ_p}$ for $N\gge 0$.
	
	In particular, as soon as we invert $\beta t$ in $\fil_{\qHodge}^\star\qhatdeRham_{(\IZ/p^{\alpha})/\IZ_p}/p$, we see that~$p$ will be invertible as well, and so
	\begin{equation*}
		\fil_{\qHodge}^\star\qhatdeRham_{(\IZ/p^{\alpha})/\IZ_p}\lotimes_{\Oo_{\ov X^*_\solid/\AdicGm}}\Oo_{X^*/\AdicGm}\simeq \fil_{\qHodge}^\star\qhatdeRham_{(\IZ/p^{\alpha})/\IZ_p}\bigl[\localise{p}\bigr]\,.
	\end{equation*} 
	Moreover, as soon as $p$ is invertible, $[p^n]_{\ku}(t)$ will be invertible for all $n\geqslant 0$. Choosing $s>N$, we see that $\fil_{\qHodge}^\star\qhatdeRham_{(\IZ/p^{\alpha})/\IZ_p}$ contains an element of the form $(\beta t)^s/p$ which is topologically nilpotent, hence automatically solid. Moreover, for $(p-1)p^n>s$ and $r>(n+1)\alpha$, we get an element of the form $p^{r}/[p^n]_{\ku}(t)$, which is again topologically nilpotent and thus solid. Thus, for such $n$, $r$, and $s$, a map $\Oo_{W_{n,r,s}^*/\AdicGm}\rightarrow\fil_{\qHodge}^\star\qhatdeRham_{(\IZ/p^{\alpha})/\IZ_p}[1/p]$ exists. The argument in the $q$-Hodge case is analogous.
	\end{proof}
	\begin{rem}
	As a consequence of \cite[Theorem~\chref{3.11}]{qWittHabiro}, $\qHodge_{(\IZ/p^{\alpha})/\IZ_p}/(q^{p^n}-1)$ is an algebra over the $p$-typical Witt vectors $\IW_{p^n}(\IZ/p^\alpha)$. Since this ring is $p^{\alpha+n}$-torsion, we already have elements of the form $p^{\alpha+n}/(q^{p^n}-1)$ in $\qHodge_{(\IZ/p^{\alpha})/\IZ_p}$ for all $n\geqslant 0$.
	\end{rem}
	\begin{lem}\label{lem:qHodgeSystemOfOpensII}
	For all fixed  $n$, $r$, $s$ such that $(p-1)p^n>s$ and all sufficiently large $\alpha\geqslant 2$, there exist canonical maps
	\begin{align*}
		\fil_{\qHodge}^\star\qhatdeRham_{(\IZ/p^{\alpha})/\IZ_p}\lotimes_{\Oo_{\ov X^*_\solid/\AdicGm}}\Oo_{X^*/\AdicGm}&\longrightarrow \Oo_{W_{n,r,s}^*/\AdicGm}\,,\\
		\qHodge_{(\IZ/p^{\alpha})/\IZ_p}\lotimes_{\IZ_p\qpower_\solid}\Oo_{X}&\longrightarrow\Oo_{W_{n,r,s}} 
	\end{align*}
	in $\Dd(X^*_\solid/\AdicGm)$ and $\Dd(X_\solid)$, respectively.
	\end{lem}
	\begin{proof}
	Let $\q D_\alpha\coloneqq \qdeRham_{(\IZ_p\{x\}/x^\alpha)/\IZ_p\{x\}}$ as in \cref{subsec:ElementaryProof} and let $\fil_{\qHodge}^\star\q \widehat{D}_\alpha$ denote its completed $q$-Hodge filtration. It follows from \cref{par:LiftsOfDividedPowers} that $\fil_{\qHodge}^\star\q \widehat{D}_\alpha$ is generated as a $(p,t)$-complete graded $\IZ_p[\beta]\llbracket t\rrbracket$-algebra by lifts of the iterated divided powers $\gamma^{(d)}(x^\alpha)$ sitting in filtration degree $2p^d$. Thanks to \cref{lem:StructuralResultqHodge}, we know that these lifts can be chosen to be of the form
	\begin{equation*}
		\frac{(\Gamma_d)^\alpha}{t^{p^d}\prod_{i=1}^d\Phi_{p^i}(q)^{p^{d-i}}}
	\end{equation*}
	for $\Gamma_d\in (x^{p},(q-1)^{p-1})^{p^{d-1}}$. The extra $t^{p^d}$ in the denominator accomodates for the fact that this element must sit in degree $2p^d$. Note that the denominators all become invertible in $\Oo_{W_{n,r,s}^*/\AdicGm}$, but that's not enough to obtain the desired map: We must send the generators to \emph{solid} elements, to ensure that the map extends over the $(p,t)$-completion.
	
	By construction, $(q-1)^s/p$ and $p^r/[p^n]_{\ku}(t)$ are solid. In particular, $p^r/(t\Phi_{p^i}(q))$ is solid for all $i=1,\dotsc,n$. For $i>n$, we have $(p-1)p^{i-1}>s$ by assumption. Hence $(q-1)^{(p-1)p^{i-1}}/p$ is topologically nilpotent in $\Oo_{W_{n,r,s}^*/\AdicGm}$. It follows that $\Phi_{p^i}(q)=p(1+w)$, where $w$ is topologically nilpotent, and so $p^r/\Phi_{p^i}(q)$ is solid in $\Oo_{W_{n,r,s}^*/\AdicGm}$ for $i>n$. Therefore the elements $p^{2r}/(t\Phi_{p^i}(q))$ are solid for all $i\geqslant 1$.
	
	By choosing $\alpha$ large enough, we can ensure that for every monomial $x^{pi}(q-1)^{(p-1)j}$ in the ideal $(x^{p},(q-1)^{p-1})^{\alpha p^{d-1}}$ we have $pi\geqslant 2rp^d$ or $(p-1)j\geqslant sp^d$. Now $(\Gamma_d)^\alpha$ is a $\IZ_p\{x\}[q]$-linear combination of such terms. It follows that the $\delta$-ring map $\IZ_p\{x\}\rightarrow \IZ_p$ sending $x\mapsto p$ can really be extended to a map $\fil_{\qHodge}^\star\q\widehat{D}_\alpha\rightarrow \Oo_{W_{n,r,s}^*/\AdicGm}$ of graded solid condensed $\IZ_p[\beta]\llbracket t\rrbracket$-algebras. Via $(p,t)$-completed base change along $\IZ_p\{x\}\rightarrow \IZ_p$ and extension of scalars to $\Oo_{X^*/\AdicGm}$, this yields the desired map
	\begin{equation*}
		\fil_{\qHodge}^\star\qhatdeRham_{(\IZ/p^{\alpha})/\IZ_p}\lotimes_{\Oo_{\ov X^*_\solid/\AdicGm}}\Oo_{X^*/\AdicGm}\longrightarrow \Oo_{W_{n,r,s}^*/\AdicGm}
	\end{equation*}
	The argument in the $q$-Hodge case is analogous.
	\end{proof}
	\begin{proof}[Proof of \cref{thm:OverconvergentNeighbourhood,thm:OverconvergentNeighbourhoodEquivariant}]
	By \cref{lem:AkuOverAnalyticLocus} and \cref{lem:NuclearIdempotentAbstract}\cref{enum:IdempotentBasechange}, we see that $\Akup$ is the colimit of the idempotent nuclear ind-algebra given by killing the pro-idempotent
	\begin{equation*}
		\prolim_{\alpha\geqslant2}\fil_{\qHodge}^\star\qhatdeRham_{(\IZ/p^{\alpha})/\IZ_p}\lotimes_{\Oo_{\ov X^*_\solid/\AdicGm}}\Oo_{X^*/\AdicGm}
	\end{equation*}
	in $\Dd(X^*_\solid/\AdicGm)$. By \cref{lem:qHodgeSystemOfOpensI,lem:qHodgeSystemOfOpensII}, we see that this pro-system is equivalent to the pro-system $\prolim_{n,r,s}\Oo_{W_{n,r,s}^*/\AdicGm}$, which proves $\Akup\simeq \Oo_{Z^{*,\dagger}/\AdicGm}$. The argument for $\AKUp\simeq \Oo_{Z^\dagger}$ is completely analogous.
	\end{proof}
	\begin{rem}
		An obvious adaptation of \cref{thm:RefinedTC-qHodge} shows that $\AKUp$ and $\Akup$ are connective. Therefore the condition from \cref{thm:OverconvergentIdempotentNuclear}\cref{enum:OZdaggerPushforward} is satisfied and so $\Oo_{Z^\dagger}$ and $\Oo_{Z^{*,\dagger}/\AdicGm}$ are really the pushforwards of the respective structure sheaves.
	\end{rem}

	\newpage 
	\renewcommand{\SectionPrefix}{}
	
	\renewcommand{\bibfont}{\small}
	\printbibliography
\end{document}